\documentclass[onefignum,onetabnum]{siamart251216} 
\newif\ifsiamart
\makeatletter%
\@ifclassloaded{siamart251216}
  {\siamarttrue}%
  {\siamartfalse}%
\makeatother%

\ifsiamart\else
\usepackage{authblk}
\fi
\usepackage{silence}
\usepackage[T1]{fontenc}
\usepackage[utf8]{inputenc}
\usepackage{csquotes}
\usepackage[UKenglish]{babel}
\usepackage{paralist}
\usepackage[shortlabels]{enumitem}
\usepackage{xcolor}
\usepackage{graphicx}
\usepackage{bbm}
\usepackage{amsmath}
\usepackage{amssymb}
\usepackage{amsfonts}
\usepackage{stmaryrd}
\usepackage{mathrsfs}
\usepackage{mathtools}
\usepackage{epstopdf}
\usepackage{comment}
\usepackage{subcaption}
\usepackage{xspace}
\usepackage{float}
\usepackage{stackengine}

\ifsiamart\else
\usepackage{amsthm}
\usepackage{hyperref}
\usepackage[margin=.6in]{geometry}
\usepackage{setspace}
\hypersetup{%
    linktocpage=true,
    colorlinks=true,
    citecolor=red,
    linkcolor=blue,
    urlcolor=blue,
}
\fi
\usepackage{amsopn}

\DeclareMathOperator{\e}{e}

\DeclareMathOperator*{\esssup}{ess\,sup}
\renewcommand{\d}{\mathrm{d}}
\newcommand{\expect}{\mathbb E}

\newcommand{\norm}[1]{\left\lVert#1\right\rVert}
\newcommand{\torus}{\mathbb T}
\newcommand{\real}{\mathbb R}

\newcommand{\ee}{{\rm e}}
\newcommand{\av}[1]{\langle #1 \rangle}

\newcommand{\mc}[1]{\mathcal{#1}}

\renewcommand{\leq}{\leqslant}
\renewcommand{\geq}{\geqslant}
\renewcommand{\le}{\leqslant}

\newcommand{\mescin}{\mu}
\newcommand{\mespos}{\nu}
\newcommand{\mesvit}{\kappa}
\newcommand{\emf}{E_{\mathrm{MF}}}

\newcommand{\emff}[1]{E_{\mathrm{MF}}\!\left[#1\right]}

\newcommand{\DeffN}{\mc{D}^{\mathrm{eff},N}}

\newcommand{\poinccu}{R_{\infty}(\beta)}
\newcommand{\normT}[1]{\norm{#1}_{L^\infty(\torus^d)}}

\newcommand{\projp}{\Pi}
\newcommand{\projvit}{\Pi_{\kappa}}
\newcommand{\adj}{\mathsf{L}_\infty^*}
\newcommand{\gennonlin}[1]{\mathsf{L}_{#1}}
\newcommand{\genlin}{\mathsf{L}_\infty}

\newcommand{\Lham}{\mathsf{L}^{\rm{ham}}}
\newcommand{\LFD}{\mathsf{L}^{\rm{FD}}}
\newcommand{\LhamN}{\mc{L}^{\rm{ham}}_N}
\newcommand{\LFDN}{\mc{L}^{\rm{FD}}_N}
\newcommand{\Lovd}{\mc{L}_{\rm{ovd}}}
\newcommand{\slangle}[1]{\left\langle #1 \right\rangle}
\newcommand{\dblangle}[1]{\left\langle\!\left\langle #1 \right\rangle\!\right\rangle}
\newcommand{\norml}[1]{\norm{#1}_{L^2(\mescin_\infty)}}
\newcommand{\scal}[1]{\slangle{#1}_{L^2(\mescin_\infty)}}
\newcommand{\eps}{\varepsilon}
\newcommand{\lambham}{\lambda_{\rm{ham}}(\beta)}
\newcommand{\kham}{K_{\rm{ham}}(\beta)}

\newcommand{\ql}{\overline q}
\newcommand{\pl}{\overline p}
\newcommand{\bl}{\overline B}
\newcommand{\vq}[1]{q^{#1}}
\newcommand{\vp}[1]{p^{#1}}
\newcommand{\vqt}[1]{q_t^{#1}}
\newcommand{\vpt}[1]{p_t^{#1}}
\newcommand{\vqz}[1]{q_0^{#1}}
\newcommand{\vpz}[1]{p_0^{#1}}
\newcommand{\vP}[1]{\mathbf{p}^{#1}}
\newcommand{\vQ}[1]{\mathbf{q}^{#1}}
\newcommand{\MNV}{M_N^V}
\newcommand{\MNW}{M_N^W}
\newcommand{\SolPois}{\Phi^\infty}
\definecolor{rose}{RGB}{202, 001, 202}

\ifsiamart
\newsiamremark{remark}{Remark}
\newcounter{assumptioncounter}
\newenvironment{assumption}[1][]{%
  \refstepcounter{assumptioncounter}%
  \trivlist
  \item[\hskip\labelsep\hskip\parindent\scshape Assumption~\theassumptioncounter.]%
  \normalfont\ignorespaces
}{\endtrivlist}
\else
\theoremstyle{plain}

\newtheorem{remark}[theorem]{Remark}

\newtheorem{assumption}[theorem]{Assumption}
\fi

\crefname{lemma}{Lemma}{Lemmas}
\crefname{remark}{Remark}{Remarks}
\crefname{assumption}{Assumption}{Assumptions}
\crefname{assumptioncounter}{Assumption}{Assumptions}
\Crefname{assumptioncounter}{Assumption}{Assumptions}
\crefname{proposition}{Proposition}{Propositions}
\crefname{section}{Section}{Sections}
\crefname{subsection}{Subsection}{Subsections}
\crefname{equation}{}{}
\Crefname{equation}{Equation}{Equations}
\newlist{lemmaenum}{enumerate}{3}
\setlist[lemmaenum]{label=(\alph*),ref=\,(\alph*)}
\crefname{lemmaenum}{Lemma}{Lemmas}
\newlist{assumpenum}{enumerate}{5}
\setlist[assumpenum]{label=(\alph*), font={\bfseries}}
\newlist{auxenum}{enumerate}{2}
\setlist[auxenum]{label=(\alph*),ref=(\alph*)}
\crefname{auxenumi}{Item}{Items}
\crefname{enumi}{}{}
\crefname{equation}{}{}
\crefname{assumpenumi}{}{}
\crefname{assumpenumii}{}{}
\Crefname{assumpenumi}{Assumption}{Assumptions}
\Crefname{assumpenumii}{Assumption}{Assumptions}
\Crefname{assumpenumii}{Assumption}{Assumptions}
\crefrangeformat{assumpenumi}{#3#1#4--#5#2#6}
\Crefname{lemmaenumi}{Part}{Parts}
\Crefname{figure}{Figure}{Figures}
\numberwithin{equation}{section}
\numberwithin{theorem}{section}

\usepackage{tikz}
\usetikzlibrary{shapes.misc}
\usetikzlibrary{positioning}

\newcounter{ucounter}

\graphicspath{{figures/}}

\ifpdf
  \DeclareGraphicsExtensions{.eps,.pdf,.png,.jpg}
\else
  \DeclareGraphicsExtensions{.eps}
\fi

\title{On the diffusive-mean field limit of a kinetic weakly interacting particle system}

\ifsiamart
    \author{%
        Raphaël Gastaldello\thanks{%
        CERMICS, CNRS, ENPC, Institut Polytechnique de Paris, Marne-la-Vallée, France \& MATHERIALS, Inria Paris, France
        (\email{raphael.gastaldello@enpc.fr}).
        } \and
        Grigorios A. Pavliotis\thanks{%
        Imperial College London, London, UK
        (\email{g.pavliotis@imperial.ac.uk}).
        } \and
        Gabriel Stoltz\thanks{%
        CERMICS, CNRS, ENPC, Institut Polytechnique de Paris, Marne-la-Vallée, France \& MATHERIALS, Inria Paris, France
        (\email{gabriel.stoltz@enpc.fr}).
        } \and
        Urbain Vaes\thanks{%
        MATHERIALS, Inria Paris, France \& CERMICS, CNRS, ENPC, Institut Polytechnique de Paris, Marne-la-Vallée, France
        (\email{urbain.vaes@inria.fr}).
        }
    }
\else
    \title{On the Diffusive-Mean Field limit of Kinetic Interacting Langevin Diffusions}
    \author[1,2]{R. Gastaldello$^{a,}$}
    \author[3]{G. A. Pavliotis$^{b,}$}
    \author[1,2]{G. Stoltz$^{c,}$}
    \author[2,1]{U. Vaes $^{d,}$}
    \affil[ ]{\footnotesize
        $^a$\email{raphael.gastaldello@enpc.fr},
        $^b$\email{g.pavliotis@imperial.ac.uk},
        $^c$\email{gabriel.stoltz@enpc.fr},
        $^d$\email{urbain.vaes@inria.fr}
    }
    \affil[1]{\footnotesize CERMICS, CNRS, ENPC, Institut Polytechnique de Paris, Marne-la-Vallée, France}
    \affil[2]{\footnotesize MATHERIALS project-team, Inria Paris, France}
    \affil[3]{\footnotesize Imperial College London, London}
    \date{\today}
\fi
\ifpdf
\hypersetup{
  pdftitle={},
  pdfauthor={R. Gastaldello, G. A. Pavliotis, G. Stoltz, and U. Vaes}
}
\fi
\headers{On the Diffusive-Mean Field Limit for Kinetic Systems}{R. Gastaldello, G. A. Pavliotis, G. Stoltz, and U. Vaes}
\begin{document}

\maketitle
\begin{abstract}
We study the joint diffusive-mean field limit for a system of weakly interacting kinetic Langevin dynamics, extending the results of~\cite{delgadino2021diffusive} to the hypoelliptic/hypocoercive case. We show that, in the absence of phase transitions, the two limits commute, and we calculate the covariance matrix of the limiting Brownian motion using the Green-Kubo/Kipnis-Varadhan formula. However, at low temperatures, and in the presence of phase transitions, the two limits may not commute. We demonstrate our findings by providing a detailed analysis of the diffusive-mean field limit for the $O(2)$ model in a magnetic field. Our analysis is based on the systematic use of recently developed hypocoercivity techniques, together with an appropriate linearization of the mean field McKean--Vlasov--Fokker--Planck partial differential equation.


\end{abstract}
%
%
\section{Introduction}
\label{sec:intro}

Interacting particle systems driven by noise have attracted a lot of attention in recent years. In addition to the well established applications to problems in statistical physics, synchronization and mathematical biology, new and exciting applications to, e.g. social dynamics~\cite{frank04}, the development of algorithms for sampling and optimization~\cite{CRS2025, GHV2025}, and to the mathematics of machine learning~\cite{GLPR2025, SirignanoSpiliopoulos2020a, RVE2022, MMN2018} have emerged. One important aspect of the type of interacting particle systems that we consider in this article, and that arise in the aforementioned applications, is that they often exhibit phase transitions in the mean field limit, i.e. nonuniqueness of stationary states.


In this paper, we will be concerned with the large-scale/long-time, behavior of systems of weakly interacting diffusions. A typical example of such a system is the weakly interacting (overdamped) Langevin dynamics where the position~$(q^i_t)_{i\in\{1\dots N\}}$ of the~$N$ particles evolves according to
\begin{equation}
    \label{eq:intro_ovd_IPS}
    \d \vqt{i} = -\nabla V(\vqt{i})\,\d t -\frac1N \sum_{j\neq i}^{N}\nabla W(\vqt{i} -\vqt{j}) \,\d t + \sqrt{\frac{2}{\beta}}\,\d B_t^i \, ,
\end{equation}
with~$i \in \{1,\dots,N\}$ and~$ \{ B^i_t \}_{t\geq0}$ standard~$d$-dimensional mutually independent Brownian motions.
Here, $V \colon \mathcal D \to \real$ is the external potential in the domain~$\mathcal D\subset\real^d$, while $W \colon \mathcal D \to \real$ is the interaction potential. The system of weakly interacting diffusions~\eqref{eq:intro_ovd_IPS} was studied, from the perspective we consider in the present work, in~\cite{delgadino2021diffusive}. In this paper, we consider the kinetic/underdamped Langevin dynamics where the configuration of the system is given by the position~$q^i$ and the momentum~$p^i$ of the~$i$-th particle, and the dynamics is given by
\begin{equation}
    \label{eq:intro_kinetic_IPS}
    \left\{\begin{aligned}
    &\d \vqt{i} =  \vpt{i}\, \d t \, , \\
    &\d \vpt{i} = -\nabla V(\vqt{i})\,\d t -\frac1N \sum_{j\neq i}^{N}\nabla W(\vqt{i} -\vqt{j}) \,\d t - \gamma \vpt{i} \,\d t + \sqrt{\frac{2 \gamma}{\beta}}\,\d B_t^i \, .
    \end{aligned}\right.
\end{equation}
The second order dynamics~\eqref{eq:intro_kinetic_IPS} models the motion of~$N$ particles that are in contact with a heat bath, in the presence a confining potential~$V$ and interacting via the interaction potential~$W$. Such systems appear, for example, in models for synchronization~\cite{Ruffo_al_2018} and in plasma physics~\cite{dressler1987stationary}.

The behavior of weakly interacting diffusions of the form~\eqref{eq:intro_ovd_IPS} or~\eqref{eq:intro_kinetic_IPS} as the number of particles~$N$ tends to infinity, \textit{i.e.}\!\! the mean field limit, has been widely studied; see, for instance, the work of Sznitman~\cite{sznitman1991topics},
the work of Jabin and coauthors~\cite{jabin2014review,jabin2016mean,jabin2017mean},
the recent review~\cite{chaintron2022propagationI,chaintron2021propagationII},
and the recent paper~\cite{delarue2025uniform} for results on uniform propagation of chaos for the dynamics on the $d$-dimensional torus~$\torus^d$.
Indeed, it can be shown~\cite{GLWZ2021, guillin2021uniform} that in the mean-field limit~$N\to +\infty$, the particle system converges to the mean field McKean SDE
\begin{align}
    \label{eq:intro_McKean_Vlasov}
    \left\{
        \begin{aligned}
            \d \ql_t &= \displaystyle \pl_t \,\d t, \, \\
            \d \pl_t &= \displaystyle -\nabla V(\ql_t) \,\d t - \nabla W \star \projp \mescin_t (\ql_t)\,\d t -\gamma \pl_t \,\d t +\sqrt{\frac{2\gamma}{\beta}}\,\d \bl_t \, ,
        \end{aligned}
    \right.
\end{align}
where~$\mescin_t$ is the law of~$(\ql_t, \pl_t)$ and $\projp \mescin_t (\ql) =\int_{\mathbb{R}^d} \mescin_t(\ql, \pl) \, \d \pl$ denotes the position marginal.
The Fokker-Planck equation that governs the evolution of $\mu_t$ is a nonlinear, nonlocal McKean--Vlasov--Fokker--Planck equation~\cite{DuPeZi_2013}, see Equation~\eqref{eq:fok-pl-vlas} below.

The kinetic Langevin mean field dynamics~\eqref{eq:intro_McKean_Vlasov} and the corresponding McKean--Vlasov--Fokker--Planck equation have been studied extensively in the last few decades; see~\cite{dressler1987stationary, dressler1990steady} for early references and~\cite{guillin2021kinetic, guillin2021uniform} for more recent work. The analysis of the McKean--Vlasov dynamics in the kinetic setting is more involved in comparison to the overdamped one, since the generator of the dynamics is hypoelliptic and hypocoercive. Uniform propagation of chaos has been proved  under strong assumptions on the confining and interaction potentials~\cite{guillin2022convergence,guillin2021uniform} or, more recently, under weaker assumptions~\cite{monmarche2024time,chen2024uniform,gong2024uniform}.

The study of the long-term behavior of overdamped and kinetic nonlinear McKean SDEs and of the corresponding McKean-Vlasov PDEs remains an active research area. Even for the linear (in the sense of McKean) diffusion process, i.e., the weakly interacting kinetic Langevin dynamics~\eqref{eq:intro_kinetic_IPS}, the rigorous study of convergence to the (unique) stationary state requires specialized techniques that have been developed in the past two decades, starting with Villani's~$H^1$ theory of hypocoercivity~\cite{villani2009hypocoercivity}; this technique can be employed in the setting of~\eqref{eq:intro_kinetic_IPS} when the position space is the torus and the interaction potential is smooth with all derivatives bounded.
Recent works use the~$L^2$-hypocoercive approach from~\cite{dolbeault2009hypocoercivity,dolbeault2015hypocoercivity} to prove exponentially fast convergence to equilibrium, see~\cite{addala20212,gervais2024well}.

In addition to the hypocoercive structure of the kinetic Langevin dynamics, the analysis of the mean field SDE~\eqref{eq:intro_McKean_Vlasov} can be further complicated by the presence of multiple stationary states at low temperatures due to phase transitions. In particular, it is well known that for non~$H$-stable interaction potentials~$W$, at sufficiently low temperatures, the Fokker--Planck equation associated with the nonlinear McKean process~\eqref{eq:intro_McKean_Vlasov} has multiple stationary solutions.
See~\cite{carrillo2020long,delgadino2021diffusive,delgadino2023phase} for a detailed study of phase transitions for~\eqref{eq:intro_ovd_IPS}. We also mention~\cite{duong2016stationary,duong2018mean} for the kinetic case; see also~\cite{duong2017vlasovfokkerplanckequi},
where a free-energy approach is considered, as well as~\cite{guillin2022convergence}, where a method based on reflective couplings~\cite{eberle2016reflection,eberle2019couplings} is used to prove long-term convergence. Finally, we refer to~\cite{monmarche2024local} for convergence results towards a local equilibrium. As is well known~\cite{dressler1987stationary, dressler1990steady}, the invariant measure of the kinetic mean-field dynamics is a product measure and, therefore, the overdamped and kinetic mean-field SDEs exhibit the same phase transitions, with a critical temperature that is independent of the friction coefficient. 



In this paper, we study the large-scale/long-time, i.e., \emph{hydrodynamic} behavior of the weakly interacting Langevin dynamics~\eqref{eq:intro_kinetic_IPS}, for the particular case where the confining and interaction potentials are smooth periodic functions. In particular, we study the joint mean-field/diffusive limits~$N \to \infty$ and~$\eps \to 0$ of~\eqref{eq:intro_kinetic_IPS}, when considered as a process in $\mathbb{R}^{d N} \times \mathbb{R}^{d N}$. In this case, we can combine the propagation of chaos property with tools from the theory of periodic homogenization~\cite{PavlSt08} and of functional central limit theorems and invariance principles~\cite{komorowski2012fluctuations, de1989invariance} to show that the large-scale/long-time behavior of the kinetic Langevin dynamics is purely diffusive, and to calculate the covariance matrix of the limiting Brownian motion using the Green-Kubo/Kipnis-Varadhan formula.

 This joint diffusive-mean field limit was studied for the overdamped dynamics in~\cite{delgadino2021diffusive}; see also~\cite{gomes2018mean, BS2023}.
In particular, it was shown in~\cite{delgadino2021diffusive} that, in the mean-field/diffusive limit, particles converge to a Brownian motion with a covariance matrix that can be calculated by solving an appropriate Poisson equation. The two limits commute when we are away from phase transitions. However, at low temperatures, and for non-$H$-stable interaction potentials, the order in which the two limits are taken leads to a different covariance matrix; this is due to the presence of phase transitions, as a result of which the particles can experience a different stationary state. Thus, the average in the Kipnis-Varadhan formula for the covariance matrix of the limiting Brownian motion is taken with respect to a different measure. We can interpret this result to mean that, in the presence of phase transitions, the \emph{transport coefficient}, i.e., the covariance matrix, depends discontinuously on the microscopic data of the problem. We illustrate in~\cref{fig:phase_transition_overdamped} the phase transition for the overdamped Langevin dynamics for the~$O(2)$ model studied in~\cite{delgadino2021diffusive}, where in~\cref{fig:a_beta} we illustrate the evolution and bifurcation of the parameter $a_\beta$, which characterizes the stationary states, with respect to $\beta$. More information can be found in~\cref{sec:kuramoto}. In~\cref{fig:D_beta}, we illustrate the behavior of the diffusion coefficient as a function of $\beta$.
\begin{figure}[ht]
    \centering
    \begin{subfigure}{0.49\textwidth}
        \centering
        \includegraphics[width=\linewidth]{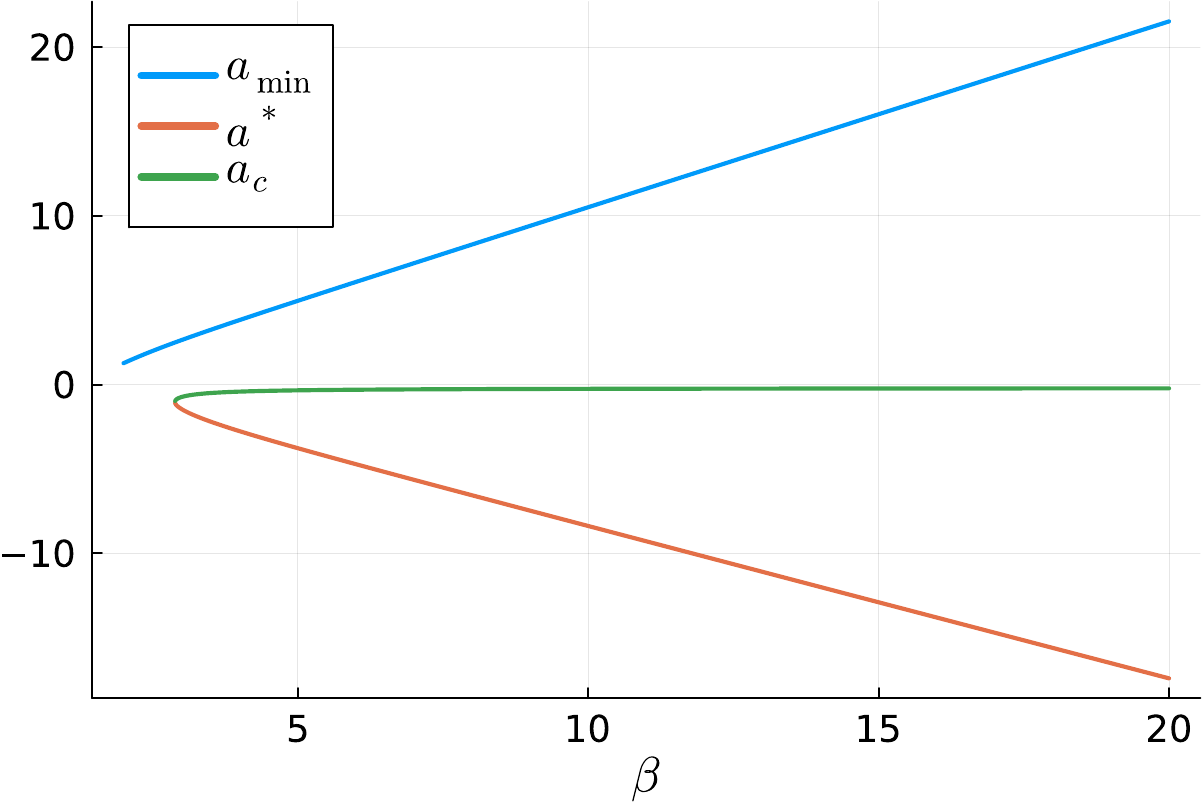}
        \caption{Zeros of the function~\eqref{eq:def_function_F} with respect to~$\beta$.}
        \label{fig:a_beta}
    \end{subfigure}
    \begin{subfigure}{0.49\textwidth}
        \centering
        \includegraphics[width=\linewidth]{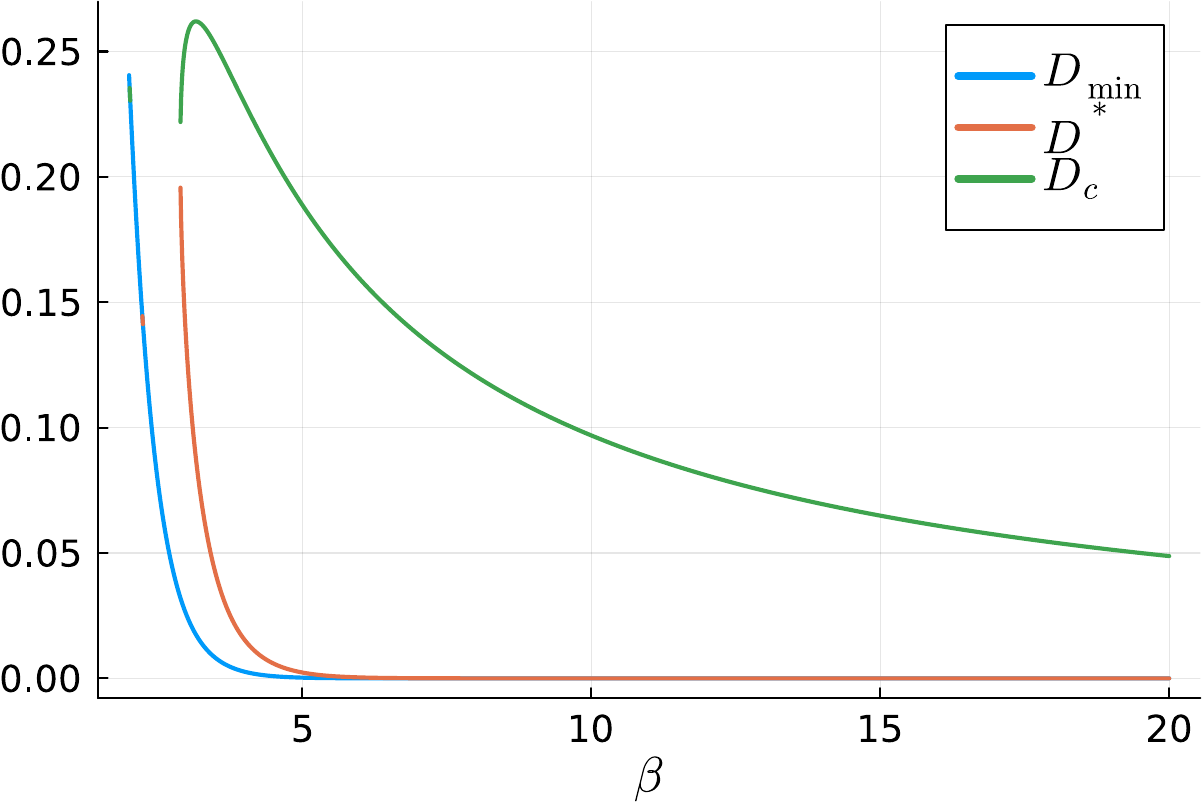}
        \caption{Diffusion coefficients of~\eqref{eq:kinetic_kuramoto_mf} in the overdamped case with respect to the corresponding steady state.}
        \label{fig:D_beta}
    \end{subfigure}
    \caption{Number of steady states/Phase transition of the overdamped~$O(2)$ model studied in~\cite{delgadino2021diffusive} with~$\eta = 0.1$.}
    \label{fig:phase_transition_overdamped}
\end{figure}

Our goal is to extend these results to the kinetic/hypoelliptic/hypocoercive Langevin dynamics~\eqref{eq:intro_kinetic_IPS}. In particular, we prove an invariance principle for the particle dynamics in the joint limits $N \to \infty$ and~$\eps \to 0$; furthermore, we show that, at low temperatures and in the presence of phase transitions, the covariance matrix of the limiting Brownian motion depends on the order in which we take the two limits.

In the following, we summarize the main results of this paper.


\smallskip

\paragraph{Our contributions} 
\begin{itemize}

    \item
        We study convergence to equilibrium of the kinetic McKean SDE using recently developed~$L^2$-hypocoercivity techniques~\cite{dolbeault2015hypocoercivity,roussel2018spectral}. In particular, we consider the generator of the nonlinear process, in the sense of McKean, as a perturbation of a linearized version of it. This idea is mainly inspired by~\cite{pavliotis2025linearizationergodicmckeansdes}.
        In particular, we prove exponential convergence in a weighted~$L^2$ of the nonlinear dynamics to a local minimizer of the mean field free energy, provided that the initial law~$\mu_0$ is sufficiently close to this local minimizer, in an appropriate metric.
    \item We study the joint diffusive/mean-field limit of the weakly interacting Langevin dynamics for periodic confining and interaction potentials; in particular, we prove an invariance principle and we calculate the covariance matrix of the limiting Brownian motion.

    \item
        We present a detailed study of the diffusive behavior of the kinetic McKean SDE for periodic interaction and confining potentials. In particular, we show that when this process admits multiple stationary distributions,
        the covariance matrix of the limiting Brownian motion depends on the specific stationary distribution~$\mu_{\infty}$ reached in the long-time limit; more precisely, the effective diffusion coefficient for the nonlinear process coincides with the one obtained by linearizing the dynamics around~$\mu_{\infty}$,
        that is, by substituting~$\mu_{\infty}$ for~$\mu_t$ in~\eqref{eq:intro_kinetic_IPS}.
        This is the content of~\cref{th:lim_nonlin_mvt_brow}.


    \item Finally, we present a detailed analysis of the $O(2)$ model in a magnetic field, extending the result of~\cite[Lemma 1.26]{delgadino2021diffusive}, and we calculate the exact number of stationary states at low temperatures and characterize their stability.
\end{itemize}

\paragraph{Outline of the paper}
In~\cref{sec:Mathematical_framework}, we present the mathematical framework of our study.
We introduce the weakly interacting kinetic Langevin dynamics, and we discuss propagation of chaos and the derivation of the mean-field McKean SDE. Then, after introducing the concept of a phase transition for the mean field dynamics, we state our main results. In~\cref{sec:proof_cvg_eq_ergo}, we prove convergence to equilibrium of the kinetic McKean SDE, using a linearization argument and hypocoercivity techniques and, as a corollary, we prove the ergodicity the kinetic McKean SDE.
In~\cref{sec:diffusive_MF} and~\cref{sec:lim_mat_system},
we give the proofs of the two main results of this paper, making systematic use of modern hypocoercivity techniques.
In~\cref{sec:kuramoto} we present a detailed study of stationary states of the $O(2)$ model in a magnetic field and present numerical experiments for this model for which we demonstrate numerically that the mean field/diffusive limits do not commute in the presence of phase transitions.
\section{Mathematical framework and statement of main results}
\label{sec:Mathematical_framework}

In this section, we state several properties of the kinetic interacting Langevin dynamics that we consider in this paper, and we state our main results. In \cref{sec:particles_system} we present the mathematical framework of the kinetic interacting particle system. In~\cref{sec:mean_field_eq} we study the mean field limit, as~$N \to \infty$, and we obtain the kinetic McKean SDE and the associated nonlinear McKean-Vlasov--Fokker--Planck partial differential equation (PDE). Furthermore, we study the long-time behavior of the McKean-Vlasov--Fokker--Planck PDE, in particular in the case when multiple stationary states exist.
We then introduce in~\cref{sec:on_diff_lim} the diffusive limit for the particle system; in \cref{sec:diff-mean-field}, we study the joint mean field-diffusive limit.
Finally, in~\cref{sec:synthesis}, we summarize in a single theorem the main contribution of this work.
%
%
\subsection{The interacting particle system}
\label{sec:particles_system}
We first present the general setting of our study.
Denote by~$\torus^d$ the~$d$-dimensional torus and introduce the state space~$\mc{E} = \torus^d \times \real^d$.
We consider a large number~$N$ of indistinguishable interacting particles~$(\vq{i},\vp{i})_{i\in\{1\dots N\}}$,
where~$\vq{i}$ is the position of the~$i$-th particle and~$\vp{i}$ is its momentum.
The configuration of the system is represented by a~$2Nd$-dimensional vector~$(\vQ{N},\vP{N}):=\left(\vq{1},\dots,\vq{N},\vp{1},\dots,\vp{N}\right)\in \torus^{dN} \times \real^{dN} =: \mc{E}^N$,
and characterized by the Hamiltonian
\begin{align}
\label{eq:potential_U}
    H_N\bigl(\vQ{N},\vP{N}\bigr)
    =\frac{|\vP{N}|^2}{2} + U_N(\vQ{N}),
    \quad
    U_N\bigl(\vQ{N}\bigr)
    := \sum_{i=1}^{N}V(\vq{i}) +\frac{1}{2N} \sum_{i=1}^{N}\sum_{j\neq i}^{N}W(\vq{i} - \vq{j}) \, ,
\end{align}
where~$|\cdot |$ denotes the Euclidean norm. The functions~$V\colon \torus^d \rightarrow \real$ and~$W\colon \torus^d \rightarrow \real$ are~$\mc{C}^2(\torus^d)$ confining and interaction potentials, respectively. We assume that~$W$ is an even function.
For~$i\in\{1\dots N\}$, the associated kinetic Langevin dynamics with friction~$\gamma >0$ and inverse temperature~$\beta$, is given by
\begin{equation}
    \label{eq:particles_system}
    \left\{\begin{aligned}
    &\d \vqt{i} =  \vpt{i}\, \d t, \\
    &\d \vpt{i} = -\nabla V(\vqt{i})\,\d t -\frac1N \sum_{j\neq i}^{N}\nabla W(\vqt{i} -\vqt{j}) \,\d t - \gamma \vpt{i} \,\d t + \sqrt{\frac{2 \gamma}{\beta}}\,\d B_t^i \, ,
    \end{aligned}\right.
\end{equation}
where $(B^i_t)_{t \geq 0}$
are $d-$dimensional independent standard Wiener processes.
We consider this system with an initial configuration consisting of i.i.d.\ random variables~$(\vqz{i},\vpz{i})$ with common law~$\mescin_0$,~\textit{i.e.} chaotic initial conditions. Since~$\nabla V$ and~$\nabla W$ are globally Lipschitz, by standard results, see, e.g.~\cite[Theorem 3.1.1]{Rockner2007},
the SDE system~\eqref{eq:particles_system} has a unique strong solution.
Finally, we make the following assumption.
\begin{assumption}
    \label{assumption:moments_p}
    For any~$\alpha \in \mathbb{N}$, the initial probability measure~$\mescin_0$ belongs to~$\mc{P}_\alpha(\mc{E})$, with~$\mc{P}_\alpha(\mc{E})$ the space of probability measures that admit finite moments of order~$\alpha$.
\end{assumption}

Under the above assumptions on the confining and interaction potentials and using the results from~\cite{mattingly2002geometric,kliemann1987recurrence}, the system~\eqref{eq:particles_system} is ergodic  (in the sense of almost sure convergence of trajectory averages) with respect to the Boltzmann--Gibbs distribution.
\begin{align}
    \label{eq:mes_equilibre}
    \mu_N(\d\vQ{N} \d\vP{N}) = \mespos_N(\d\vQ{N}) \mesvit_N(\d\vP{N}),
\end{align}
where
\begin{align}
    \label{eq:mes_posN}
    \mespos_N(\d\vQ{N}) = Z_N^{-1}\exp\!\left(-\beta U_N(\vQ{N})\right)\d\vQ{N}, \qquad Z_N = \int_{\mc{E}^N}\exp\!\left(-\beta U_N(\vQ{N})\right)\d\vQ{N},
\end{align}
and
\begin{align}
    \label{eq:mes_vitN}
    \mesvit_N(\d\vP{N}) = \left(\frac{\beta}{2\pi}\right)^{\frac{Nd}{2}}\exp\!\left(-\beta \frac{|\vP{N}|^2}{2}\right)\d\vP{N}.
\end{align}
In the sequel, we do not distinguish between a probability measure and its density
if the measure is absolutely continuous with respect to the Lebesgue measure.
Let~$\mescin_t^N$ denote the law of the process~\eqref{eq:particles_system} at time~$t$, and define $f_t$ to be the density of the law of the process with respect to the Gibbs measure~$\mescin_N$:
\begin{align}\label{e:density}
    \mescin_t^N = f_t^N \mescin_N.
\end{align}
The evolution of~$f_t^N$ is given by the Fokker--Planck equation
\begin{align}
    \label{eq:kinetic_fp}
    \partial_t f_t^N = \mc{L}^*_N f_t^N,
\end{align}
where~$\mc{L}_N$ is the infinitesimal generator of~\eqref{eq:particles_system} and where the notation~$A^*$ stands for the $L^2(\mu_N)-$adjoint of the operator~$A$.
In fact,
\begin{align}
    \label{eq:generator_mf}
    \mc{L}_N =  \LhamN + \gamma\LFDN,
\end{align}
where
\begin{align}
    \label{eq:LhamN}
    \LhamN = \vP{N} \cdot \nabla_q - \nabla U_N \cdot \nabla_p = \frac1\beta(\nabla_p^*\nabla_q - \nabla_q^*\nabla_p),
\end{align}
and
\begin{align}
    \label{eq:LSymN}
    \LFDN = -\vP{N}\cdot\nabla_p + \beta^{-1}\Delta_p = -\frac1\beta\nabla_p^*\nabla_p,
\end{align}
 are antisymmetric and symmetric, respectively, in $L^2(\mu_N)$.
 Here~$\nabla_p^*$ and $\nabla_q^*$ are the formal $L^2(\mu_N)$ adjoints of the operators~$\nabla_p$ and $\nabla_q$.
 While $\nabla_q$ and $\nabla_p$ are viewed as column vectors of operators,
$\nabla_q^*$ and $\nabla_p^*$ are viewed as row vectors of operators with components given by
\begin{align*}
    \partial_{q_i}^*  = -\partial_{q_i} + \beta \partial_{q_i} U_N , \qquad \partial_{p_i}^*  = -\partial_{p_i} + \beta p_i .
\end{align*}
%
%
\subsection{The mean-field McKean--Vlasov dynamics}
\label{sec:mean_field_eq}
In this section, we first present in~\cref{sec:inv_measure_of_McKean-Vlasov} the McKean--Vlasov dynamics associated with the interacting particle system~\eqref{eq:particles_system}.
Then in~\cref{sec:vfp_eq} we give more details on the Vlasov--Fokker--Planck PDE that governs the law of the dynamics.
In~\cref{sec:phase_transition},
we discuss phase transitions and non-uniqueness of stationary states at low temperatures, with a focus on the~$O(2)$ model in a magnetic field.
Finally, in~\cref{sec:cvg_eq_ergo}, we present some novel results on the convergence to equilibrium of the McKean--Vlasov dynamics.
%
%
\subsubsection{The mean-field dynamics and its invariant probability measures}
\label{sec:inv_measure_of_McKean-Vlasov}
In this section, we study the mean-field dynamics, obtained from~\eqref{eq:particles_system} in the limit as~$N$ goes to infinity. For results on propagation of chaos for the kinetic Langevin dynamics, we refer to~\cite{GLWZ2021, guillin2021uniform} and the references therein.

The limiting dynamics is described by the McKean--Vlasov equation
\begin{align}
    \label{eq:kin_mv_nonlin}
    \left\{
\begin{aligned}
    \d q_t &= \displaystyle p_t \,\d t, \\
    \d p_t &= \displaystyle -\nabla V(q_t) \,\d t - (\nabla W \star \Pi \mescin_t )(q_t)\,\d t -\gamma p_t \,\d t +\sqrt{\frac{2\gamma}{\beta}}\,\d B_t \, ,
\end{aligned}
    \right.
\end{align}
where~$\mescin_t$ is the law of the process~\eqref{eq:kin_mv_nonlin} with density with respect to the Lebesgue measure,
the probability measure $(\projp\mescin_t)(q)= \int_{\real^{d}} \mescin_t(q,p)\,\d p$ denotes the marginal of~$\mu_t$ on the position variable~$q$,
and~$\star$ denotes the convolution operator:
\[
    (\nabla W \star \Pi \mescin_t)(q)
    = \int_{\torus^d} \nabla W(q - \widetilde q) \, \Pi \mescin_t(\d \widetilde q) \, .
\]
Since~$V$ and~$W$ are globally Lipschitz,~\eqref{eq:kin_mv_nonlin} admits a unique strong solution; see, for example, the recent review~\cite[Proposition 1]{chaintron2022propagationI} or~\cite[Theorem 1.7]{carmona2016lectures}. We will often refer to~\eqref{eq:kin_mv_nonlin} as the nonlinear process (in the sense of McKean). Furthermore, any invariant probability measure for~\eqref{eq:kin_mv_nonlin} is necessarily of the form~\cite{dressler1987stationary, duong2016stationary}
\begin{align}
    \label{eq:mes_inv_mckean_vlas}
    \mescin(\d q\,\d p) = \left(\frac{\beta}{2\pi}\right)^{d/2} \ee^{-\beta \frac{|p|^2}{2}} \mespos(q)\,\d q\,\d p,
\end{align}
with
\begin{align}
    \label{eq:marginal_position}
    \mespos(q) = \frac{1}{Z_\infty}\ee^{-\beta(V(q)+W\star \mespos(q))}, \qquad Z_\infty = \int_{\torus^d} \ee^{-\beta(V(q)+W\star \mespos(q))}\,\d q.
\end{align}
Indeed, an invariant probability measure~$\mescin_{\infty}$ for~\eqref{eq:kin_mv_nonlin} must also be invariant for the kinetic Langevin dynamics obtained by fixing~$\mescin_t = \mescin_{\infty}$ in~\eqref{eq:kin_mv_nonlin}.
It is then well-known that, for a fixed twice-differentiable potential $V + W \star \projp \mescin_{\infty}$,
the corresponding kinetic Langevin dynamics admits~\eqref{eq:mes_inv_mckean_vlas} as its unique invariant measure.
An interesting feature of the integral equation satisfied by the position marginal~$\nu$ in~\eqref{eq:marginal_position} is that it can have multiple solutions for non-$H-$stable potentials and at low temperatures~$\beta^{-1}$. This non-uniqueness of stationary states describes the phenomenon of phase transitions. This will be discussed in more detail in \cref{sec:phase_transition}.
Moreover, any probability measure~$\mu$ satisfying~\eqref{eq:mes_inv_mckean_vlas} is also a stationary state of the McKean--Vlasov--Fokker--Planck equation
\begin{align}
    \label{eq:fok-pl-vlas}
    \partial_t \mescin_t = -p\cdot\nabla_q\mescin_t + (\nabla_q V+\nabla_q W \star \projp \mescin_t)\cdot\nabla_p \mescin_t + \gamma\nabla_p\cdot\!\left(p \mescin_t + \beta^{-1}\nabla_p\mescin_t\right).
\end{align}
Finally, an invariant probability measure of~\eqref{eq:kin_mv_nonlin}, or a steady state of~\eqref{eq:fok-pl-vlas}, is a critical point of the free energy functional~\cite{Dolbeault_1999}
\begin{align}
    \label{eq:mf_free_energy}
    \emff{\rho} = \int_{\mc{E}} \left(\frac{1}{\beta}\log \rho(q,p) + \frac{|p|^2}{2} + V(q) +  \frac12 ( W \star \projp\rho)(q) \right)\!\rho(q,p)\,\d q \, \d p.
\end{align}
Indeed, we have the following result that gathers all previous statements.
\begin{proposition}[{\cite{duong2016stationary,duong2017vlasovfokkerplanckequi}}]
    \label{prop:eq_min_FE}
    Fix~$\beta >0$. Let~$\mescin \in \mc{P}(\mc{E})$. Then, the following statements are equivalent:
    \begin{enumerate}
        \item[(1)] $\mescin$ is invariant under the dynamics~\eqref{eq:kin_mv_nonlin}.
        \item[(2)] $\mescin$ is a classical solution of the stationary Vlasov--Fokker--Planck equation~\eqref{eq:fok-pl-vlas}.
        \item[(3)] $\projp\mescin = \mespos$ is a solution of~\eqref{eq:marginal_position} and~$\mescin = \mespos\otimes\mc{N}(0,\beta^{-1}\text{I}_d)$.
        \item[(4)] $\mescin$ is a critical point of the free energy~\eqref{eq:mf_free_energy}.
    \end{enumerate}
\end{proposition}
\cref{prop:eq_min_FE} is a summary of the results in~\cite{duong2016stationary,duong2017vlasovfokkerplanckequi}.
We now make an important distinction between critical points and global minimizers of the free energy~\eqref{eq:mf_free_energy}.
Although any invariant probability measure of~\eqref{eq:kin_mv_nonlin} is a critical point of~\eqref{eq:mf_free_energy}, not all critical points are global minimizers.
However, under suitable assumptions on the potentials~$V$ and~$W$,
it is possible to show that~\eqref{eq:mf_free_energy} admits at least one global minimizer in~$\mc{P}(\mc{E})$.
See, for example,~\cite{messer1982statistical, Bavaud_1991, delgadino2021diffusive} for more details.  
Let~$i\in\mathbb{N}$ and define
\begin{align}
    \label{eq:marginals}
   (\mescin_N)_i = \int_{\mc{E}^{N-1}} \mescin_N(\d q_1\d p_1,\dots,\d q_{i-1}\d p_{i-1},\cdot,\d q_{i+1}\d p_{i+1},\dots,\d q_N\d p_N),
\end{align}
to be the $i$-th marginal of the invariant probability measure~$\mescin_N$ in~\eqref{eq:particles_system}. 
We now state a result that is fundamental in the sequel.
\begin{proposition}
    \label{prop:weak_limit_marginals}
    Suppose that~\eqref{eq:mf_free_energy} admits a unique global minimizer~$\mescin_{\min}$ (not necessarily a unique critical point),
    and, for any~$i \in \{1, 2\dotsc\}$, let~$(\mescin_N)_i$ denote the $i$-th marginal, in the sense of~\eqref{eq:marginals}, of the invariant probability measure~$\mescin_N$ given in~\eqref{eq:particles_system}.
    Then the sequence of the densities of probability measures~$((\mescin_N)_i)_{N\geq i}$ converges in~$\mc{C}^0(\torus^d \times \real ^d)$ in the limit~$N \to \infty$ to~$\mescin_{\min}$.
\end{proposition}
\begin{proof}
    If~$\mescin_{\min}$ is the unique global minimizer of~\eqref{eq:mf_free_energy}, then,~$\mespos_{\min}$ is the unique global minimizer of the free energy functional associated with the overdamped dynamics. Based on the results of~\cite[Section 3]{delgadino2021diffusive}, $(\mespos_N)_i$ converges in~$\mc{C}^0(\torus^d)$ to~$\mespos_{\min}$. Since~$(\mescin_N)_i= (\mespos_N)_i \otimes \mc{N}(0,\beta^{-1}\text{I}_d)$, we deduce that~$(\mescin_N)_i$ converges in~$\mc{C}^0(\torus^d \times \real ^d)$ to~$\mescin_{\min}$, which concludes the proof.
\end{proof}
%
%
\subsubsection{The McKean-Vlasov-Fokker-Planck equation}
\label{sec:vfp_eq}
In this section, we give more details on the McKean--Vlasov--Fokker--Planck equation~\eqref{eq:fok-pl-vlas} which governs the evolution of the law~$(\mescin_t)_{t \geq 0}$ of the process~\eqref{eq:kin_mv_nonlin}.
Suppose that~$\mescin_\infty$ is a stationary solution of~\eqref{eq:fok-pl-vlas},
and, similarly to~\eqref{e:density}, introduce the density with respect to the invariant measure~$\mescin_t = f_t\mescin_\infty$.
Then,~$(f_t)_{t \geq 0}$ satisfies the following equation.
\begin{lemma}
    Let~$(\mescin_t)_{t\geq0}$ and~$\mescin_\infty$ be a solution of~\eqref{eq:fok-pl-vlas} and a stationary state of~\eqref{eq:fok-pl-vlas}, respectively. Define~$f_t$ as~$\mescin_t = f_t\mescin_\infty$. Then, the evolution of $(f_t)_{t \geq 0}$ is governed by the following nonlinear PDE:
    \begin{align}
    \label{eq:backward_vlas_fok_planck}
    \partial_t f_t =  \gennonlin{t}^* f_t -  \beta p\cdot\nabla_q W \star \projp (\mescin_t - \mescin_\infty) f_t,
\end{align}
where we define
\begin{align}
    \label{eq:generator_nonlin}
    \gennonlin{t}^* = -\Lham_t + \gamma\LFD,
\end{align}
with
\begin{align*}
    \Lham_t = p\cdot \nabla_q - (\nabla_q V+ \nabla_q W \star \projp \mescin_t) \cdot \nabla_p,
\end{align*}
and
\begin{align}
    \label{eq:def_LFD}
    \LFD = -p\cdot\nabla_p + \beta^{-1}\Delta_p.
\end{align}
\end{lemma}
The slight abuse of notation, and definition, for the nonlinear operator~$\gennonlin{t}^*$ will become clear in the sequel.
We emphasize that $\gennonlin{t}^*$ is not the formal $L^2(\mu_{\infty})-$adjoint of~$\gennonlin{t}$,
unless $\mu_t = \mu_{\infty}$.
\begin{proof}
    By replacing~$\mescin_t$ by~$f_t\mescin_\infty$ in~\eqref{eq:fok-pl-vlas}, we have
\begin{align*}
    \partial_t (f_t\mescin_\infty) = -p\cdot \nabla_q (f_t\mescin_\infty) + \nabla U_t\cdot \nabla_p (f_t\mescin_\infty) + \gamma \nabla_p\cdot \Bigl(\beta^{-1}\nabla_p(f_t\mescin_\infty)+p(f_t\mescin_\infty)\Bigr),
\end{align*}
with
\begin{align}
    \label{eq:time_dependent_potential}
    U_t = V + W \star \projp \mescin_t.
\end{align}
Straightforward computations give for the fluctuation-dissipation part
\begin{align*}
    \gamma \nabla_p\cdot (\beta^{-1}\nabla_p(f_t\mescin_\infty)+p(f_t\mescin_\infty)) = \left(-\gamma p\cdot \nabla_p f_t +\frac{\gamma}{\beta}\Delta_p f_t\right)\mescin_\infty.
\end{align*}
For the remaining terms, we have
\begin{align*}
    -p\cdot \nabla_q (f_t\mescin_\infty) = -(p\cdot\nabla_q f_t)\mescin_\infty +(\beta p\cdot \nabla U_\infty )f_t\mescin_\infty,
\end{align*}
where,(see~\eqref{eq:time_dependent_potential})
\begin{equation}\label{eq:inf_potential}
U_\infty = V+ W \star \projp \mescin_\infty.
\end{equation}
Moreover,
\begin{align*}
    \nabla U_t\cdot \nabla_p (f_t\mescin_\infty) =( \nabla U_t \cdot \nabla_p f_t)\mescin_\infty - (\beta p\cdot\nabla U_t )f_t\mescin_\infty.
\end{align*}
By gathering all terms,~$f_t$ is the solution of
\begin{align*}
    \partial_t f_t =  \gennonlin{t}^* f_t -  \beta p\cdot\nabla W \star \projp (\mescin_t - \mescin_\infty) f_t,
\end{align*}
which concludes the proof.
\end{proof}

\begin{remark}
We note that, the McKean--Vlasov--Fokker--Planck equation~\eqref{eq:backward_vlas_fok_planck} for the density with respect to an invariant measure contains an additional term when compared to the backward Kolmogorov equation obtained in the linear case; see~\cite[Sec. 6.1]{PavliotisStochasticProcesses2014}.
\end{remark}
\subsubsection{Phase transitions for the~\texorpdfstring{$O(2)$}{O(2)} model in a magnetic field}
\label{sec:phase_transition}
In this section, we adapt some of the results in~\cite{carrillo2020long}
using results from~\cite{duong2016stationary} to obtain information on the number of stationary states for the kinetic model. Let us start with the following definition.
\begin{definition}[{\cite[Definition 5.1]{carrillo2020long}}]
    \label{def:phase_transition}
    The mean field system~\eqref{eq:fok-pl-vlas} is said to admit a critical temperature~$0<\beta_{\rm c} < \infty$
    if~$\beta_{\rm c}$ satisfies the following conditions:
    \begin{itemize}
        \item For~$0<\beta <\beta_{\rm c}$,~\eqref{eq:mf_free_energy} admits a unique critical point.
        \item For~$\beta =\beta_{\rm c}$,~\eqref{eq:mf_free_energy} admits a critical point. 
        \item For~$\beta > \beta_{\rm c}$,~\eqref{eq:mf_free_energy} admits at least two critical points.
    \end{itemize}
\end{definition}
As expected, we have the uniqueness of stationary states at high temperatures.
\begin{proposition}[Uniqueness at high temperature]
    \label{prop:uniq_stationary}
    For any~$0<\beta < \infty$,
    the mean field system~\eqref{eq:fok-pl-vlas} has at least one steady state, which is a minimizer of the mean field energy~$\emf$.
    Furthermore,
    there exists~$\beta_{\rm{c}}$ such that, for any~$\beta <\beta_{\rm{c}}$,
    there is a unique steady state~$\mescin_\infty$ of~\eqref{eq:fok-pl-vlas},
    which corresponds to the unique minimizer of~$\emf$.
\end{proposition}
\begin{proof}
    This result has been proven in~\cite[Theorem 2.3 and Proposition 2.8]{carrillo2020long} for the overdamped system where standard fixed point and compactness arguments are used to show the existence of a solution to the integral equation in~\eqref{eq:marginal_position}.
    Then the existence of~$\beta_{\rm c}$ is proved by showing the convexity of the free energy for~$\beta$ sufficiently small.
    Extending this result to the kinetic framework is a consequence of the fact that the probability measure~\eqref{eq:mes_inv_mckean_vlas} is tensorized in position and momentum; see~\cite{duong2016stationary}.
    Then, applying~\cite[Theorem 1]{duong2016stationary} (see also~\cite{dressler1987stationary,dressler1990steady}),
    where the result can be adapted to~$\torus^d$,
    the result follows from the existence of~$\beta_{\rm c}$ in the overdamped context.
\end{proof}
\begin{remark}
    In the following, we only rely on~\cref{def:phase_transition,prop:uniq_stationary} to describe phase transitions for the mean-field dynamics. See, for instance,~\cite{carrillo2020long,delgadino2023phase} and references therein for alternative approaches and additional results. Our aim is to ensure the uniqueness of a global minimizer for~\eqref{eq:mf_free_energy}. Let us also mention recent the work~\cite[Section 4]{monmarche2024local} on the long-time behavior of the McKean--Vlasov--Fokker--Planck PDE when~\eqref{eq:mf_free_energy} has more than one global minimizer.
\end{remark}

As mentioned in the introduction, we consider in this paper the (kinetic) noisy Kuramoto/$O(2)$ mean field plane rotator model in a magnetic field. When we study the problem on the torus, this can be thought of as a spin system. Indeed, let~$d = 1$ and~$V(q) = \eta W(q) = -\eta \cos(2\pi q)$, with~$\eta \in (0,1)$. The potential part of the Hamiltonian of the system is given by
\begin{align}
    \label{eq:hamiltonian_kuramoto}
V_N(q) = -\eta \sum_{i=1}^N \cos(2\pi q_i) -\frac{1}{2N}\sum_{i,j=1}^N \cos\bigl(2\pi(q_i - q_j)\bigr) = \sum_{i=1}^N h\cdot S_i -\frac{1}{2N}\sum_{i,j=1}^N S_i\cdot S_j,
\end{align}
with
\[
S_i = \bigl(\cos(2\pi q_i),\sin(2\pi q_i)\bigr),\qquad h = (-\eta,0).
\]
The potential energy~\eqref{eq:hamiltonian_kuramoto} corresponds to the classical Heisenberg XY
model for lattice systems with continuous spins and mean field interaction. See~\cite[Chapter 9]{friedli2017statistical} or~\cite{bertini2010dynamical} (without an external field) for more details on this model. See also~\cite{BeMe_2020} for a detailed spectral analysis of the generator of the $O(N)$ model in the low-temperature regime.

The kinetic version of the mean field limit of the noisy Kuramoto model is given by the McKean SDE:
\[
    \left\{
\begin{array}{ll}
    \d q_t = \displaystyle p_t \,\d t,& \\
    \d p_t = \displaystyle -V'(q_t) \,\d t - W' \star \projp \mescin_t ( q_t)\,\d t -\gamma  p_t \d t +\sqrt{\frac{2\gamma}{\beta}}\,\d B_t.&
\end{array}
    \right.
\]
The following lemma extends~\cite[Lemma 1.26]{delgadino2021diffusive} by giving the exact number of stationary states at low temperatures.
\begin{lemma}
    \label{th:Kuramoto}
    Consider the one-dimensional McKean--Vlasov--Fokker--Planck equation~\eqref{eq:fok-pl-vlas} with~$W(q) = -\cos(2\pi q)$ and $V(q) = -\eta\cos(2\pi q)$ for a fixed~$\eta \in (0,1)$. Denote
    \begin{align*}
        \mu_a(q,p) = Z(a)^{-1} \e^{a\cos(2\pi q)}\mesvit(p), \qquad Z(a) = \int_{\torus}\e^{a\cos(2\pi q)}\,\d q,
    \end{align*}
    where~$\mesvit$ is defined in~\eqref{eq:mes_inv_mckean_vlas}. Then, for any~$\beta >0$, there exists a unique global minimizer of the free energy~\eqref{eq:mf_free_energy} given by
    \begin{align}
            \label{eq:mes_in_kuramoto}
            \mescin_{\min} = \mu_{a_\beta^{\min}},
        \end{align}
    for some~$a_\beta^{\min}>0$.
    Furthermore, there is a unique value~$\beta_{\rm c} > 0$ such that:
    \begin{itemize}
        \item For~$\beta < \beta_{\rm c}$, the probability measure $\mescin_{\min}$ is the unique steady state of~\eqref{eq:fok-pl-vlas}.

        \item For~$\beta = \beta_{\rm c}$, there exist exactly two steady states of~\eqref{eq:fok-pl-vlas}, given by
        \begin{align*}
            \mescin_{\min} = \mu_{a_\beta^{\min}},\qquad
            \mescin^{\rm c} = \mu_{a_{\beta_{\rm c}}},
        \end{align*}
        where~$a_{\beta_{\rm c}} < 0 < a_\beta^{\min}$.
        \item For~$\beta > \beta_{\rm c}$, there exist exactly three steady states of~\eqref{eq:fok-pl-vlas}, where the first one is given by~\eqref{eq:mes_in_kuramoto}
            and the others by
        \begin{align*}
            \mescin^i = \mu_{a_\beta^i},
        \end{align*}
        with~$i \in \{1,2\}$ and~$a_\beta^2 < a_\beta^1 < 0 < a_\beta^{\min}$.
    \end{itemize}
    Furthermore,~$\mescin^{\rm c}$ and $\mescin^i$, for~$i \in \{1,2\}$, are unstable critical points of~\eqref{eq:mf_free_energy}, that is to say they are not local minimizers.
\end{lemma}
The fact that~$\mescin_{\min}$ is the unique global minimizer of the free energy~\eqref{eq:mf_free_energy} is established in~\cite[Lemma 1.26]{delgadino2021diffusive}. The other results are proved in~\cref{sec:kuramoto}.
\subsubsection{Convergence to equilibrium and ergodicity}
\label{sec:cvg_eq_ergo}
In this section, we provide some novel results on the convergence to a stationary state for the McKean--Vlasov--Fokker--Planck PDE~\eqref{eq:fok-pl-vlas}. As a consequence, we obtain information on the ergodic properties of the process~\eqref{eq:kin_mv_nonlin}; in particular, we prove the convergence in~$L^2$ of time averages. This last result will be necessary to establish~\cref{th:lim_nonlin_mvt_brow}, since, as expected, the long-time behavior depends on the quadratic variation of an appropriately defined martingale. Let~$\mu,\nu \in \mc{P}(\torus^d \times \real^d)$ be two probability measures such that~$\mu \ll \nu$. Below, we use the relative entropy or Kullback--Leibler divergence
$
    \text{Ent}(\mu|\nu) = \int_{\torus^d \times \real^d} \log\!\left(\frac{ \mu}{ \nu}\right)\!\d \mu,
$ and the $2-$Wasserstein distance
$
    \mc{W}_2(\mu,\nu) = \inf_{\pi \in \Pi (\mu,\nu)} \left(\int_{\mc{E}\times \mc{E}} |x-y|^2 \,\d \pi(x,y)\right)^{1/2},
$
where~$\Pi(\mu,\nu)$ is the set of all couplings of~$\mu$ and~$\nu$.

We can now state our result on the convergence to equilibrium for the McKean--Vlasov--Fokker--Planck PDE~\eqref{eq:fok-pl-vlas}.
\begin{proposition}
    \label{prop:exp_decay_nonlin}
   Let~$\mescin_\infty$ denote a stationary solution of~\eqref{eq:fok-pl-vlas}, assumed to be a local minimizer of the free energy~\eqref{eq:mf_free_energy}. Suppose that the initial distribution~$\mescin_0$ satisfies~$\mescin_0/\mescin_\infty \in L^2(\mescin_\infty)$.
Then, for any~$\beta > 0$, there exist~$L_\beta,\lambda_\beta,\delta_\beta > 0$ such that if~$\mc{W}_2(\mescin_0, \mescin_\infty) \leq \delta_\beta$,
then
    \begin{align}
        \label{eq:exp_decay_CI}
    \norml{\displaystyle\frac{\mescin_t}{\mescin_\infty} -1} \leq L_\beta\e^{-\lambda_\beta t}.
    \end{align}
\end{proposition}
Let us emphasize that the critical point considered in~\cref{prop:exp_decay_nonlin} is a local minimizer of the free energy~\eqref{eq:mf_free_energy} and not just a critical point. Thus, this convergence result is very specific and deeply relies on the results of~\cite[Theorem 20]{monmarche2024local}. Hence, \cref{prop:exp_decay_nonlin} should be seen as slight improvement of~\cite[Theorem 20]{monmarche2024local} with a convergence in the~$\chi^2$-divergence. 
A corollary of~\cref{prop:exp_decay_nonlin} is an ergodic theorem for time averages. Before stating the result, we introduce, for a function~$\mc{K}$, the weighted space
\begin{align*}
    L^{\infty}_{\mathcal K } = \left\{f:\torus^d \times \real^d \rightarrow \real,\,\text{measurable},\quad \|f\|_{L^{\infty}_{\mathcal K }} := \esssup_{(q,p) \in \torus^d \times \real^d} \frac{|f(q,p)|}{\mathcal K (q,p)} < \infty \right\}.
\end{align*}
\begin{corollary}
    \label{lem:ergodicity}
    Let~$\mescin_\infty$ denote a stationary solution of~\eqref{eq:fok-pl-vlas}, and let the process~$(q_t, p_t)_{t \geq 0}$ denote the solution to the mean field SDE
    \[
        \left\{
            \begin{aligned}
                \d q_t &= \displaystyle p_t \,\d t, \\
                \d p_t &= \displaystyle -\nabla V(q_t) \,\d t - \nabla W \star \Pi \mu_t (q_t)\,\d t -\gamma p_t \,\d t + \sqrt{\frac{2\gamma}{\beta}}\,\d B_t \, ,
            \end{aligned}
        \right.
    \]
    with an initial condition that satisfies~\cref{assumption:moments_p} and where $(\mu_t)_{t \geq 0}$ is the law of~$(q_t, p_t)_{t \geq 0}$.
    Moreover, assume that there exist positive constants~$\lambda > 0$ and $L > 0$ such that,
    \[
        \forall s \geq 0, \qquad
        \left\|\frac{\mu_s}{\mu_{\infty}} - 1\right\|_{L^2(\mu_{\infty})}
        \leq L \e^{-\lambda s} \, .
    \]
    Consider, for $f \in L^{\infty}_{\mathcal K }$ with $\mathcal K (q, p) = 1 + |p|^{\alpha}$ for some~$\alpha > 0$,
    the following estimator of~$\theta:= \expect_{\mescin_\infty}[f]$:
    \[
        \widehat \theta(T) = \frac{1}{T} \int_{0}^{T} f(q_t, p_t) \, \d t \, .
    \]
    Then the mean squared error of this estimator satisfies the following inequality: there exists a constant~$C > 0$ such that
    \[
        \forall T > 0, \qquad
        \expect \left[ \left| \widehat \theta(T) - \theta \right|^2  \right] \leq\frac{C}{T} \, .
    \]
\end{corollary}
The proofs of these two results can be found in~\cref{sec:proof_cvg_eq_ergo}.

\begin{remark}
This ergodic theorem will be crucial in the study of the diffusive limit of the kinetic McKean SDE. In particular, when applying It\^o's lemma, the diffusive limit will depend on the quadratic variation of the martingale term, which we can study using~\cref{lem:ergodicity}. We emphasize that the proof of the above result is not straightforward, since the underlying infinitesimal generator is time-dependent and not self-adjoint and the results of~\cite{bhattacharya1982functional} are not directly applicable.
\end{remark}
%
%
\subsection{The diffusive limit}
\label{sec:on_diff_lim}
In~\cref{sec:mean_field_eq} we introduced the mean-field limit of the interacting particle system~\eqref{eq:particles_system} as~$N \rightarrow \infty$. We now introduce the diffusive limit of~\eqref{eq:particles_system} when the confining and interaction potentials are periodic. Fix~$N$, consider~\eqref{eq:particles_system} as a process defined in $\real^{dN} \times \real^{dN}$, and introduce the following additive functional defined in the whole space~$\real^d$: starting from~$\textbf{Q}_0^N = \vQ{N}_0$,
\begin{align}
    \label{eq:rescaled_processN}
    \textbf{Q}_t^N = \textbf{Q}_0^N + \int_{0}^{t}\vP{N}_s\,\d s.
\end{align}
Furthermore, denote by
\begin{align}
\label{eq:centered_L2}
    L^2_0(\rho) = \left\{f \in L^2(\rho) \,\middle|\, \int_E f \, \d\rho = 0 \right\}\!,
\end{align}
for a given probability measure~$\rho$ in a space~$E$. We have the following invariance principle.
\begin{proposition}[Diffusive limit~{\cite{pavliotis2008diffusive}}]
    \label{th:diff_lim_N}
    Consider for~$\eps>0$ the diffusively rescaled process
    \[
    \textbf{Q}_t^{N,\eps} = \eps \textbf{Q}_{t/\eps^2}^N.
    \]
Then, as~$\eps \rightarrow 0$, the process~$\textbf{Q}_t^{N,\eps}$ starting from a given initial condition~$\textbf{Q}_0^N$ weakly converges on finite time intervals to an~$Nd$-dimensional Brownian motion starting from~$\textbf{Q}_0^N$ and with covariance matrix
    \begin{align}
        \label{eq:matrix_cov_N}
        \DeffN &= \int_{\mc{E}^N}\Phi^N(\vQ{N},\vP{N}) \otimes \vP{N} \, \mu_N(\d\vQ{N}\d\vP{N}) \\
        &=\frac{\gamma}{\beta} \int_{\mc{E}^N}\bigl(\nabla_p\Phi^N(\vQ{N},\vP{N})\bigr)^\top\nabla_p\Phi^N(\vQ{N},\vP{N}) \, \mu_N(\d\vQ{N}\d\vP{N}),
    \end{align}
where for any~$i \in\{1,\dots, Nd\}$, the~$i$-th coordinate~$\Phi_i^N$ of~$\Phi^N$ is defined as the unique solution in~$L^2_0(\mescin_N)$ to the Poisson equation
\begin{align}
    \label{eq:poisson_eq_coordinate}
    -\mathcal{L}_N\Phi_i^N = p_i,
\end{align}
where~$p_i$ is the~$i$-th coordinate of~$\vP{N}$.
\end{proposition}
~\cref{th:diff_lim_N} is a standard result. The diffusive limit of the interacting particle system can be studied using techniques from periodic homogenization; see, for instance,~\cite[Theorem 13.1]{komorowski2012fluctuations} or~\cite{bensoussan2011asymptotic,papanicolaou1985ohrnstein,hairer2004periodic,hairer2008ballistic}. Finally, for any~$N \in \mathbb{N}$, \cite[Lemma~6.11]{PavliotisStochasticProcesses2014} ensures that~$\DeffN$ is positive definite. Moreover, it is a standard result that the Poisson equation~\eqref{eq:poisson_eq_coordinate} is well posed in~$L^2_0(\mescin_N)$ under our assumptions on~$V$ and~$W$; see, for instance,~\cite[Section 3.1.3]{lelievre2016partial}.
%
\subsection{The diffusive-mean-field limit}
\label{sec:diff-mean-field}
In this section, we investigate the commutativity of the diffusive and mean-field limits. The case $N \to \infty$ followed by $\varepsilon \to 0$ is treated in~\cref{sec:N_inf_then_eps_0}, while the reverse order is considered in \cref{sec:eps_0_then_N_inf}.
\subsubsection{The limit~\texorpdfstring{$N \rightarrow \infty$}{N to infinity} followed by~\texorpdfstring{$\eps\rightarrow 0$}{eps to 0}}
\label{sec:N_inf_then_eps_0}
In~\cref{sec:mean_field_eq} we introduced the mean field, \textit{i.e.} the limit~$N \to \infty$. Now we consider the diffusive limit of the mean-field process~\eqref{eq:kin_mv_nonlin}. Before that, we start by considering the linearized version of~\eqref{eq:kin_mv_nonlin}. This approach is inspired by~\cite{addala20212} or more recently~\cite{pavliotis2025linearizationergodicmckeansdes}.
First, let~$\mescin_\infty$ be a steady state of~\eqref{eq:fok-pl-vlas} which is, in view of~\eqref{eq:mes_inv_mckean_vlas}, of the form
\begin{align}
    \label{eq:mes_inv}
    \mescin_\infty(\d q\,\d p)
    &= \left(\frac{\beta}{2\pi}\right)^{d/2} \ee^{-\beta \frac{|p|^2}{2}}\mespos_\infty(q)\,\d q\,\d p, \qquad \mespos_\infty(q)= \frac{1}{Z_\infty}\ee^{-\beta(V(q)+W*\mespos_\infty(q))}.
\end{align}
We introduce the linearized version of~\eqref{eq:kin_mv_nonlin}, namely
\begin{align}
    \label{eq:kin_mv_lin}
    \left\{
\begin{array}{ll}
    \d \overline q_t = \displaystyle \overline p_t \,\d t  , & \\
    \d \overline p_t = \displaystyle -\nabla V(\overline q_t) \,\d t - \nabla W \star \left( \projp \mescin_\infty \right) (\overline q_t)\d t -\gamma \overline p_t \,\d t +\sqrt{\frac{2\gamma}{\beta}}\,\d B_t.&
\end{array}
    \right.
\end{align}
The probability measure~\eqref{eq:mes_inv} is the unique invariant probability measure for the process~\eqref{eq:kin_mv_lin}.
Let $\overline{\mescin}_t$ be the law of~\eqref{eq:kin_mv_lin} and~$\overline{f}_t$ be the density of~$\overline{\mescin}_t$ with respect to~$\mescin_\infty$, so that~$\overline{\mescin}_t = \overline{f}_t \mescin_\infty$.
Then~$\overline{f}_t$ is the solution of the following PDE:
\begin{align}
    \label{eq:linearized_adjoint_eq}
    \partial_t \overline{f}_t =  \adj \overline{f}_t,
\end{align}
where~$\adj$ is the $L^2(\mescin_\infty)$ adjoint operator of~$\genlin$, given by
\begin{align}
    \label{eq:generator_lin}
    \adj &= -\Lham_\infty + \gamma\LFD,
\end{align}
with~$\LFD$ defined in~\eqref{eq:def_LFD} and
\begin{align}
    \Lham_\infty = p\cdot \nabla_q - \!\left(\nabla_q V + \nabla_q W * \projp\mescin_\infty\right) \cdot \nabla_p.
\end{align}
In analogy to~\eqref{eq:rescaled_processN}, define
\begin{align*}
    \overline{Q}_t = Q_0 +\int_0^{t} \overline{p}_s \,\d s,
\end{align*}
starting from~$Q_0 = q_0 \in \real^d$. We can state the following result for the diffusive limit of~\eqref{eq:kin_mv_lin}.
\begin{theorem}
    \label{th:lim_mvt_brow}
Let~$\mescin_\infty$ be a stationary state of~\eqref{eq:fok-pl-vlas} and suppose that~\cref{assumption:moments_p} holds. Define for~$\eps>0$ the diffusively rescaled process~$\overline{Q}_t^{\eps} = \eps \overline{Q}_{t/\eps^2} \in \real^d$. Then, as~$\eps \rightarrow 0$, the process~$\overline{Q}_t^{\eps}$ starting from a given initial condition~$Q_0$ weakly converges on finite time intervals to a~$d$-dimensional Brownian motion starting from~$Q_0$ and with covariance matrix
\begin{align}
    \label{eq:cov_mat_diff}
    \mc{D}_\infty = \frac{\gamma}{\beta}\int_{\torus^d \times \real^d} (\nabla_p \SolPois(q,p))^\top \nabla_p \SolPois(q,p) \, \d\mescin_{\infty}(q,p),
\end{align}
where, for any~$i \in\{1,\dots, d\}$, the~$i$-th coordinate~$\SolPois_i$ of~$\SolPois$ is defined as the unique solution in~$L^2_0(\mescin_\infty)$ to the Poisson equation
\begin{align}
    \label{eq:sol_poisson_MV_lin}
    -\genlin \SolPois_i = p_i,
\end{align}
with~$p_i \in \real$ the~$i$-th coordinate of~$p \in \real^d$.
\end{theorem}
Since the process~\eqref{eq:kin_mv_lin} is a linear diffusion process in the sense of McKean,~\cref{th:lim_mvt_brow} is similar to the diffusive limit for the~$N$-particle system as given by~\cref{th:diff_lim_N}. In addition,~$\mc{D}_\infty$ is positive definite by~\cite[Lemma 6.11]{PavliotisStochasticProcesses2014}.
We now state the main result of this section. Let~\eqref{eq:kin_mv_nonlin} be the standard McKean--Vlasov process and introduce
\begin{align}
    \label{eq:integrated_McKV}
    Q_t = Q_0 +\int_0^{t}p_s \,\d s,\qquad Q_t^{\eps} = \eps Q_{t/\eps^2} \in \real^d.
\end{align}
Then, under suitable assumptions,~$\lim\limits_{\eps \rightarrow 0} \overline{Q}_t^\eps = \lim\limits_{\eps \rightarrow 0} Q_t^\eps$ in a weak sense, as made precise by the next result.
\begin{theorem}
    \label{th:lim_nonlin_mvt_brow}
    Let~$\mescin_\infty$ be a stationary state of~\eqref{eq:fok-pl-vlas} and suppose that~\cref{assumption:moments_p} holds. Consider, for~$\eps>0$, the diffusively rescaled process~$Q_t^{\eps} \in \real^d$.
    Suppose that there exist~$C,\lambda > 0$ such that
    \begin{align}
        \label{eq:hyp_cvg_eq}
        \norm{\mu_t - \mescin_\infty}_{L^1(\torus^d \times \real^d)} \leq C\e^{-\lambda t}.
    \end{align}
    Then, as~$\eps \rightarrow 0$, the process~$Q_t^{\eps}$ starting from a given initial condition~$Q_0$ converges weakly on finite time intervals to a~$d$-dimensional Brownian motion starting from~$Q_0$ and with covariance matrix~$\mc{D}_\infty$~defined in~\eqref{eq:cov_mat_diff}.
\end{theorem}
The proof of~\cref{th:lim_nonlin_mvt_brow} can be found in~\cref{sec:diffusive_MF}. Furthermore,~\cref{prop:exp_decay_nonlin} gives a sufficient condition for~\eqref{eq:hyp_cvg_eq} to hold.
%
%
\subsubsection{The limit~\texorpdfstring{$\eps \rightarrow 0$}{eps to 0} followed by~\texorpdfstring{$N\rightarrow\infty$}{N to infinity}}
\label{sec:eps_0_then_N_inf}
We conclude this section by characterizing the limit~$N \rightarrow \infty$ of the diffusive limit discussed in~\cref{sec:on_diff_lim}.
\begin{theorem}
    \label{th:eps_then_N}
    Assume that the mean field energy~$\emf$~\eqref{eq:mf_free_energy} admits a unique minimizer~$\mescin_{\min}$.
    Let~$u_i = (\delta_{i,j})_{1\leq j\leq dN}$ and~$e_i = u_i \otimes \text{I}_d$, with~$\text{I}_d$ being the $d \times d$ identity matrix. Then,
    \begin{align}
        \label{eq:lim_mat_con_N}
        \lim_{N \to +\infty} e_1^\top\DeffN e_1 = \mc{D}_{\min},
    \end{align}
    where the form of~$\mc{D}_{\min}$ is given in~\eqref{eq:cov_mat_diff}. Moreover, for~$i,j$ such that~$(k_1-1)d \leq i \leq k_1d$ and~$(k_2-1)d \leq j \leq k_2d$ with~$k_1 \neq k_2$, we have
    \begin{align*}
        \!\left|u_{i}^\top\DeffN u_{j}\right|\!\leq \frac{2}{\gamma\beta\sqrt{N-1}},
    \end{align*}
     Finally, exchangeability of the system, the convergence~\eqref{eq:lim_mat_con_N} holds for any particle. In particular, for each particle index \(i\), the corresponding marginal process converges in law, as \(N \to \infty\), to a Brownian motion with covariance matrix~\(\mc{D}_{\min}\).
\end{theorem}
The proof of~\cref{th:eps_then_N} is given in~\cref{sec:proof_lim_cov_particle_system}. An important consequence of this result is that, if~\eqref{eq:mf_free_energy} admits a unique minimizer, not necessarily a unique critical point, then each diagonal block of size~$d \times d$ converges to the same matrix; this is the diffusion matrix of the homogenization problem for the linearized McKean--Vlasov equation, obtained by calculating the convolution with respect to the global minimizer.
%
%
\subsection{Summary of our main results}
\label{sec:synthesis}
We provide a synthesis of the results previously presented in this section. It gives condition under which the diffusive and mean field limit are compatible.
\begin{theorem}
    Consider the interacting particle system~\eqref{eq:particles_system} with initial conditions~$(\vQ{N}_0,\vP{N}_0) \sim \mescin_0^{\otimes N}$ and its mean field limit~\eqref{eq:kin_mv_nonlin} with initial condition~$(q,p) \sim\mescin_0$. Suppose that~\eqref{eq:mf_free_energy} admits a unique global minimizer~$\mescin_{\min}$ and that~\cref{assumption:moments_p} is satisfied. Then, for any~$\beta >0$, each diagonal block  of size~$d \times d$ of the matrix~$\DeffN$ defined in~\eqref{eq:matrix_cov_N} converges, as~$N\to \infty$, to the covariance matrix
        \begin{align*}
        \mc{D}_{\min}= \frac{\gamma}{\beta}\int_{\torus^d \times \real^d} \bigl(\nabla_p \SolPois(q,p)\bigr)^\top \bigl(\nabla_p \SolPois(q,p)\bigr)\, \d\mescin_{\min}(q,p),
        \end{align*}
        associated to the process~\eqref{eq:kin_mv_nonlin}, while the off diagonal blocks converge to 0. Furthermore,
    \begin{itemize}
        \item If~$\beta< \beta_{\rm c}$ or if~$\mescin_0 = \mescin_{\min}$, the diffusive rescaling of the mean field limit~\eqref{eq:kin_mv_nonlin} converges to a Brownian motion with covariance matrix~$\mc{D}_{\min}$.
        \item If~$\beta \geq \beta_{\rm c}$ and~$\mescin_0 \neq \mescin_{\min}$, there exists at least another critical point~$\mescin_{\infty}$ of~\eqref{eq:mf_free_energy} which is not a global minimizer of the free energy.  If~$\mescin_\infty$ is a local minimizer of the free energy, there exists~$\delta_\beta >0$ such that, if~$\mescin_0/\mescin_\infty\in L^2(\mescin_\infty)$ and
            \begin{align}
            \label{eq:ic_wasserstein_commutation}
            \mc{W}_2(\mescin_0,\mescin_\infty) \leq \delta_\beta,
            \end{align}
            the diffusive rescaling of the mean field limit~\eqref{eq:kin_mv_nonlin} converges to a Brownian motion with covariance matrix~$\mc{D}_\infty$. If~$\mescin_\infty$ is not a local minimizer of the free energy, the diffusive rescaling of the mean field limit~\eqref{eq:kin_mv_nonlin} converges to a Brownian motion with covariance matrix~$\mc{D}_\infty$ if~$\mescin_0 = \mescin_{\infty}$.
    \end{itemize}
\end{theorem}
The proof is a direct consequence of~\cref{prop:exp_decay_nonlin} and~\cref{th:eps_then_N}.

In conclusion, the diffusive and mean field limits commute if the system is above the critical temperature. However, below the critical temperature, the two limits may not commute, depending on the choice of the initial conditions, see the condition~\eqref{eq:ic_wasserstein_commutation}.
%
%
\section{Proof of~\texorpdfstring{\cref{prop:exp_decay_nonlin}}{Proposition~\ref{prop:exp_decay_nonlin}} and~\texorpdfstring{\cref{lem:ergodicity}}{Lemma~\ref{lem:ergodicity}}}
\label{sec:proof_cvg_eq_ergo}
This section is devoted to the proof of~\cref{prop:exp_decay_nonlin} and~\cref{lem:ergodicity}.
In~\cref{sec:cvg_eq} we prove~\cref{prop:exp_decay_nonlin} by adapting a hypocoercivity argument to the kinetic McKean Langevin SDE. In~\cref{subsec:ergo_nonlin}, we next prove~\cref{lem:ergodicity}, which shows that the process~\eqref{eq:kin_mv_nonlin} has good ergodic properties.
%
%
\subsection{Convergence to equilibrium for the kinetic McKean SDE}
\label{sec:cvg_eq}
In this section, we prove~\cref{prop:exp_decay_nonlin}. To this end, we use an~$L^2$-hypocoercivity result to prove convergence to equilibrium of the process~$(q_t,p_t)$ in~\eqref{eq:kin_mv_nonlin}, by adapting the approach proposed in~\cite{dolbeault2009hypocoercivity,dolbeault2015hypocoercivity}.
This method was applied to the linearized Vlasov--Poisson--Fokker-Planck PDE in~\cite{addala20212} and recently to the fully nonlinear PDE in~\cite{gervais2024well}.
The main idea is to think of the nonlinear, nonlocal kinetic PDE as a perturbation of the linear PDE~\eqref{eq:linearized_adjoint_eq}. Let~$\mescin_\infty$ be a stationary state of~\eqref{eq:fok-pl-vlas} and suppose that the initial condition~$\mescin_0$ is such that~$\mescin_0/\mescin_\infty \in L^2(\mescin_\infty)$. Introduce~$h_t := \displaystyle\frac{\mescin_t}{\mescin_\infty} - \textbf{1} = f_t - \textbf{1}$.
Then,
\begin{align}
    \label{eq:conv_eq_L1}
    \norm{\mescin_t - \mescin_\infty}_{L^1(\torus^d \times \real^d)} = \norm{\frac{\mescin_t}{\mescin_\infty} - \textbf{1}}_{L^1(\mescin_\infty)} = \norm{h_t}_{L^1(\mescin_\infty)} \leq \norm{h_t}_{L^2(\mescin_\infty)}.
\end{align}
Furthermore, using the fact that~$f_t$ and~$h_t$ have the same derivatives and using~\eqref{eq:backward_vlas_fok_planck}, we obtain
\begin{align}
    \label{eq:deriv_h_t_one}
    \partial_t h_t &= \partial_t f_t = \gennonlin{t}^*f_t - \beta p\cdot\nabla_q W \star \projp (\mescin_t - \mescin_\infty)f_t.
\end{align}
We combine expressions~\eqref{eq:generator_nonlin} and~\eqref{eq:generator_lin} to obtain
\begin{align}
    \label{eq:gen_adj_non_comp_lin}
    \gennonlin{t}^* = \adj + \nabla_q W \star  \projp(\mescin_t -\mescin_\infty)\cdot \nabla_p.
\end{align}
Applying~\eqref{eq:gen_adj_non_comp_lin} to~\eqref{eq:deriv_h_t_one} gives
\[
\partial_t h_t = \adj h_t + \nabla_q W \star  \projp(\mescin_t -\mescin_\infty)\cdot \nabla_p h_t - \beta p\cdot\nabla_q W \star \projp (\mescin_t - \mescin_\infty)f_t.
\]
Since~$\nabla_p^* = -\nabla_p + \beta p$, we can rewrite the above expression as
\begin{align}
    \label{eq:deriv_h_t_two}
    \partial_t h_t = \adj h_t - \nabla_q W \star  \projp(\mescin_t -\mescin_\infty)\cdot \nabla_p^* h_t - \beta p\cdot\nabla_q W \star \projp (\mescin_t - \mescin_\infty).
\end{align}
Finally, by differentiating with respect to time~$\norml{h_t}^2$ and using~\eqref{eq:deriv_h_t_two}, we obtain
\begin{align}
    \label{eq:deriv_density2}
    \frac{\d}{\d t}\!\left(\norm{h_t}_{L^2(\mescin_\infty)}^2\right) &= 2\slangle{\adj h_t, h_t}_{L^2(\mescin_\infty)}- 2\slangle{\nabla_q W \star \projp (\mescin_t - \mescin_\infty)\cdot\nabla_p^* h_t, h_t}_{L^2(\mescin_\infty)} \\
    \notag
    &\quad - 2\beta\slangle{ p\cdot\nabla_q W \star \projp (\mescin_t - \mescin_\infty), h_t}_{L^2(\mescin_\infty)},
\end{align}
The proof of this result is based on a standard~$L^2$-hypocoercivity argument applied to~$\adj$. The method is briefly described in~\cref{subsub:hypo_lin}. Then, in~\cref{subsub:bound_remainder}, we control the second term on the right hand side of~\eqref{eq:deriv_density2}.
Finally, in~\cref{subsub:equilibrium}, we prove exponentially fast convergence to equilibrium for the nonlinear equation for initial distributions close enough to the stationary state.
%
%
\subsubsection{Hypocoercivity for the linearized equation}
\label{subsub:hypo_lin}
Define~$U_\infty = V + W \star \projp\mescin_\infty \in \mc{C}^2(\torus^d)$. 
To establish the hypocoercivity of the generator of the kinetic Langevin dynamics with this potential, we need the following Poincar\'{e} inequality.
\begin{lemma}[{\cite[Lemma 2.12]{pavliotis2025linearizationergodicmckeansdes}}]
    \label{lem:poincare_mean_field}
    The probability measure~\eqref{eq:marginal_position}, satisfies Poincar\'{e}'s inequality: for any~$f \in L_0^2(\mespos_\infty)\cap H^1(\mespos_\infty)$,
    \[
        \|f\|_{L^2(\mespos_\infty)}^2 \leq \frac{1}{\poinccu^2}\|\nabla f\|_{L^2(\mespos_\infty)}^2,
    \]
    with constant
    \begin{align}
        \label{eq:const_poinc_linearized}
        \frac{1}{\poinccu^2}= \frac{\e^{2\beta \eta_\infty}}{4\pi^2},\qquad \eta_\infty = \normT{V} + \normT{W}.
    \end{align}
\end{lemma}
The proof of this result is a direct consequence of the Holley--Stroock perturbation lemma; see, for instance,~\cite[Theorem 5.1.6]{bakry2014analysis}.
We next introduce the following squared norm.
\begin{definition}[Modified squared norm]
    Fix~$\eps \in (0,1)$. For any smooth function~$\varphi$ with compact support,
    \begin{align}
        \mc{H}(\varphi) := \frac12\norml{\varphi}^2 + \eps\scal{A\varphi,\varphi},
    \end{align}
    with
    \begin{align}
        \label{eq:def_operator_A}
        A := \left(1+(\Lham_\infty\projvit)^*(\Lham_\infty\projvit)\right)^{-1}(\Lham_\infty\projvit)^*,
    \end{align}
    where~$\projvit\varphi = \int_{\real^d}\varphi\,\d \mesvit$.
\end{definition}
The next lemma gives some properties of the operator~$A$.
\begin{lemma}[{\cite[Lemma 1]{dolbeault2015hypocoercivity}}]
    \label{lem:bound_A_hypo}
    It holds that~$A = \projvit A(1- \projvit)$. Moreover, for any~$\varphi \in L_0^2(\mescin_\infty)$,
    \begin{align*}
        \norml{A\varphi} \leq \frac12\norml{(1-\projvit)\varphi}, \qquad \norml{\Lham_\infty A\varphi}\leq \norml{(1-\projvit)\varphi}.
    \end{align*}
\end{lemma}
In particular, this lemma shows that the operator~$A$ is bounded in~$L^2(\mescin_\infty)$ with an operator norm smaller than~$1/2$; consequently, $\mc{H}(\varphi)$ in~\eqref{eq:def_operator_A} induces a norm equivalent to the canonical one in~$L^2(\mescin_\infty)$ for~$\eps \in (-1,1)$:
\begin{align}
    \label{eq:norm_eq}
    \frac{1-\eps}{2}\norml{\varphi}^2\leq \mc{H}(\varphi) \leq \frac{1+\eps}{2}\norml{\varphi}^2.
\end{align}
By the polarization identity,~$2\mc{H}$ induces the scalar product~$\dblangle{\cdot,\cdot}$ defined by
\begin{align}
    \label{eq:inner_prod_L2_hypo}
    \dblangle{\varphi_1,\varphi_2} = \scal{\varphi_1,\varphi_2} + \eps\scal{A\varphi_1,\varphi_2} + \eps\scal{\varphi_1,A\varphi_2}.
\end{align}
In particular we have
\begin{align}
    \label{eq:cauchy_schwarz}
    |\dblangle{\varphi_1,\varphi_2}| \leq (1+\eps)\norml{\varphi_1}\norml{\varphi_2}.
\end{align}
By~\eqref{eq:deriv_h_t_two} we have
\begin{align}
    \label{eq:deriv_H}
    \frac{\d}{\d t}\mc{H}(h_t)
    &= \dblangle{\adj h_t, h_t}
    - \dblangle{\nabla_q W \star \projp(\mescin_t - \mescin_\infty)\nabla_p^*h_t, h_t} \\
    \notag
    &\quad - \beta\dblangle{ p\cdot\nabla_q W \star \projp (\mescin_t - \mescin_\infty), h_t}\!.
\end{align}
In the sequel, we will need the following classical hypocoercivity estimates, taken from~\cite{roussel2018spectral} which are reformulations of the estimates in~\cite{dolbeault2009hypocoercivity,dolbeault2015hypocoercivity} in the case of kinetic Langevin dynamics.
\begin{proposition}[{\cite[Proposition 6]{roussel2018spectral}}]
    \label{prop:coercivity_properties}
        Let~$\mc{C}_{\rm b}^\infty(\mc{E})$ denote the space of bounded~$C^\infty$ functions with all derivatives bounded on~$\mc{E}$. The operators~$\LFD$ and~$\Lham_\infty$ satisfy the following estimates:
        \begin{eqnarray}
            \label{eq:lfd_coercivity}
            &\vspace{-0.5cm}\forall \varphi \in \mc{C}_{\rm b}^\infty(\mc{E}),& -\langle \LFD\varphi,\varphi \rangle_{L^2(\mescin_\infty)}\geq \|(1-\projvit )\varphi\|^2, \\
            \label{eq:lham_coercivity}
            &\forall \varphi \in \mc{C}_{\rm b}^\infty(\mc{E})\cap L_0^2(\mescin_\infty),&\norml{\Lham_\infty\projvit \varphi}^2 \geq \frac{\poinccu^2}{\beta}\norml{\projvit \varphi}^2.
        \end{eqnarray}
        Furthermore, as a consequence of~\eqref{eq:lham_coercivity}, the following inequality holds in the sense of symmetric operators on~$L_0^2(\mescin_\infty)$:
        \begin{align}
            \label{eq:lambham}
            A\Lham\projvit \geq \lambham\projvit,\qquad \lambham = 1 - \left(1+\frac{\poinccu^2}{\beta}\right)^{-1}>0.
        \end{align}
\end{proposition}
The lower bound~\eqref{eq:lfd_coercivity} is called the microscopic coercivity property, while~\eqref{eq:lham_coercivity} is called the macroscopic coercivity property.
Finally, we need the following technical result.
\begin{proposition}[{\cite[Proposition 7]{roussel2018spectral}}]
    \label{prop:bounded_aux_op}
    For any~$\varphi \in \mc{C}_{\rm c}^\infty(\mc{E})$,
        \begin{align}
            \label{eq:torture}
            \norml{A\Lham_\infty(1-\projvit)\varphi}&\leq \kham \norml{(1-\projvit)\varphi},\\
            \label{eq:A_LFD_1-P}
            \norml{A\LFD(1-\projvit)\varphi}&\leq \frac{1}{2} \norml{(1-\projvit)\varphi},
        \end{align}
        where
        \begin{align}
            \label{eq:Kham}
            \kham = \sqrt{2\!\left(\frac1\beta+\frac{K_V+K_W}{\poinccu^2}\right)\!}.
        \end{align}
    \end{proposition}
    The proofs of~\cref{prop:coercivity_properties,prop:bounded_aux_op} are straightforward adaptations of the computations in~\cite[Section 5.3]{iacobucci2019convergence} and~\cite[Appendix A]{roussel2018spectral} (see also~\cite{bernard2022hypocoercivity} for precise estimates), which are adapted from~\cite{dolbeault2009hypocoercivity,dolbeault2015hypocoercivity}. The derivation of these estimates is standard and can be found in~\cref{sec:hypo_estimates} for completness. We also emphasize that we quantify the dependence of~$\kham$ in~$\beta$, ~ $V$ and~$W$ through~$K_V$ and~$K_W$ in the expression~\eqref{eq:Kham}. We now apply the previous results to bound the first term on the right hand side of~\eqref{eq:deriv_H}. This is the aim of the following estimate.
\begin{proposition}
    \label{prop:hypo_gen_lin}
    For any~$\eps \in(0,1)$,
    \begin{align}
        \label{eq:bound_adj_lin}
        \dblangle{h,\adj h} \leq -\eps X^\top S_\beta X -\frac{\gamma}{\beta}\norml{\nabla_p h}^2,\qquad \forall h \in \mc{C}_{\rm b}^\infty(\mc{E})\cap L_0^2(\mescin_\infty),
    \end{align}
    where
    \begin{align}
        \label{eq:matrixS}
        S_\beta &= \begin{pmatrix}
            \displaystyle\lambham & \displaystyle-\frac{1}{2}\!\left(\kham + \frac{\gamma}{2}\right) \\
            \displaystyle-\frac{1}{2}\!\left(\kham + \frac{\gamma}{2}\right) & \displaystyle - 1
            \end{pmatrix}, \\
        X &= \begin{pmatrix}
                \norml{\projvit h} \\
                \norml{(1-\projvit)h}
            \end{pmatrix}.
    \end{align}
\end{proposition}
\begin{proof}
    By using~\eqref{eq:generator_lin} in~$\dblangle{h,\adj h}$, we obtain
\begin{align*}
    \dblangle{h,\adj h} &=\gamma\slangle{\LFD h, h}_{L^2(\mescin_\infty)} - \eps\slangle{A \Lham_\infty \projvit h,h}_{L^2(\mescin_\infty)} \\
    &\quad - \eps\slangle{A \Lham_\infty(1 - \projvit) h, h}_{L^2(\mescin_\infty)}
    + \eps\gamma\slangle{A\LFD h, h}_{L^2(\mescin_\infty)} \\
    &\quad + \eps\slangle{\Lham_\infty A h, h}_{L^2(\mescin_\infty)},
\end{align*}
since~$\LFD A = \LFD\projvit A = 0$. Using~\eqref{eq:lambham},~\cref{prop:bounded_aux_op} and a Cauchy--Schwarz inequality, we obtain
\begin{align*}
    \dblangle{h,\adj h} &\leq -\frac{\gamma}{\beta}\norml{\nabla_p h}^2 -\eps\lambham \norml{\projvit h}^2 \\
    &\quad+\eps\!\left(\kham +\frac{\gamma}{2}\right)\!\norml{(1-\projvit)h}\norml{\projvit h} \\
    &\quad +\eps\slangle{\Lham_\infty A h, h}_{L^2(\mescin_\infty)}.
\end{align*}
By~\cref{lem:bound_A_hypo},
\begin{align*}
    \slangle{\Lham_\infty A h, h} = \slangle{(1-\projvit)\Lham A (1-\projvit)h,h} \leq \norml{(1-\projvit)h}^2,
\end{align*}
from which~\eqref{eq:bound_adj_lin} follows.
\end{proof}
\subsubsection{Bound on the remainder terms}
\label{subsub:bound_remainder}
To control the second term in~\eqref{eq:deriv_H},
we use the following lemma.

\begin{lemma}{{\cite[Lemma 4.2]{gervais2024well}}}
    \label{lem:gradp_adj}
    Let~$g\in H^1(\mescin_\infty)$. Then
    \begin{align}
        \label{eq:bound_gervais}
        \norml{\nabla_p^* g}^2 = \norml{\nabla_p g}^2 + d\beta \norml{g}^2.
    \end{align}
\end{lemma}
The proof of this result can be found in~\cref{ann:proof_grap_star}, where we correct the missing~$d\beta$ coefficient on the second term in the right hand side of~\eqref{eq:bound_gervais}. With this result, we have the following bound on the second and third terms in~\eqref{eq:deriv_H}.
\begin{proposition}
    \label{lem:bound_deriv_second_H}
    Let~$f \in H^1(\mescin_\infty)$ be such that~$f\mu_\infty$ is a probability measure and define~$h = f-\mathbf{1}$. Then,
    for any~$\beta\in\real_+$,~$\eps \in(0,1)$, and~$W \in \mc{C}^1(\torus^d)$, the following inequalities hold:
        \begin{align}
        \label{eq:bound_deriv_second_H_bis}
        \!\left|\dblangle{h,\nabla W\star\projp(h\mescin_\infty)\cdot \nabla_p^* h}\right|\!
        \leq M_\beta \|h\|_{L^1(\mescin_\infty)} \norml{h}^2
        + \frac{\gamma}{2\beta}\norml{\nabla_p h}^2,
        \end{align}
        with~$M_\beta = (1+\eps)\sqrt{d\beta}\normT{\nabla W} + \displaystyle\frac{(1+\eps)^2\beta}{\gamma} \normT{\nabla W}^2$ and,
        \begin{align}
        \label{eq:new_remainder}
         \left|\beta\dblangle{ p\cdot\nabla_q W \star \projp (h\mescin_\infty), h}\right|
         \leq \left(K_\beta + \norml{h}^2\right)\|h\|_{L^1(\mescin_\infty)},
        \end{align}
        with~$K_\beta = \displaystyle\frac{(1+\eps)^2d\beta}{4}\normT{\nabla W}^2$.
\end{proposition}
\begin{proof}
    We first prove~\eqref{eq:bound_deriv_second_H_bis} and then~\eqref{eq:new_remainder}.
For~\eqref{eq:bound_deriv_second_H_bis}, we start by noting that
\begin{align}
    \label{eq:bound_convolution}
    \left|\nabla_q W\star\projp(h\mescin_\infty)\right|
    \leq \normT{\nabla W} \|h\|_{L^1(\mescin_\infty)} 
\end{align}
Thus, using~\eqref{eq:cauchy_schwarz},~\cref{lem:gradp_adj} and the elementary inequality~$\sqrt{a^2 + b^2}\leq |a| + |b|$,
\begin{align}
    \label{eq:first_term_second_term_bis}
    &\left|\dblangle{h,\nabla W\star\projp(h\mescin_\infty)\cdot \nabla_p^* h}\right| \\
    \notag
    &\leq (1+\eps)\normT{\nabla W}\|h\|_{L^1(\mescin_\infty)}\norml{h}
    \left(\norml{\nabla_p h}+\sqrt{d\beta}\norml{h}\right)\!.
\end{align}
By Young's inequality,
with~$\delta >0$,
we bound the first term on the right hand side of~\eqref{eq:first_term_second_term_bis} by
\begin{align}
    \label{eq:ca_nen_fini_plus_bis}
    &(1+\eps)\normT{\nabla W}\!\|h\|_{L^1(\mescin_\infty)}\,\norml{h}\!\norml{\nabla_p h} \\
    \notag
    &\qquad\leq \frac{(1+\eps)^2\delta}{2}\normT{\nabla W}^2\|h\|_{L^1(\mescin_\infty)}^2\norml{h}^2
    +\frac{1}{2\delta}\norml{\nabla_p h}^2.
\end{align}
Taking~$\delta = \frac{\beta}{\gamma}$ and combining~\eqref{eq:ca_nen_fini_plus_bis} with~\eqref{eq:first_term_second_term_bis} and the fact that~$\|h\|_{L^1(\mescin_\infty)} \leq 2$, we obtain
\begin{align*}
    \!\left|\dblangle{h,\nabla W\star\projp(h\mescin_\infty)\cdot \nabla_p^* h}\right|\!
    \leq M_\beta\|h\|_{L^1(\mescin_\infty)}\norml{h}^2
    + \frac{\gamma}{2\beta}\norml{\nabla_p h}^2,
\end{align*}
which gives~\eqref{eq:bound_deriv_second_H_bis}.
We next prove~\eqref{eq:new_remainder}. By~\eqref{eq:cauchy_schwarz},
\begin{align*}
    \left|\beta\dblangle{ p\cdot\nabla_q W \star \projp (\mescin_t - \mescin_\infty), h}\right|
    &\leq (1+\eps)\beta\norml{p\cdot\nabla_q W \star \projp (\mescin_t - \mescin_\infty)}\norml{h}.
\end{align*}
Since~$p\cdot\nabla_q W \star \projp (\mescin_t - \mescin_\infty)$ is a scalar product of function of~$q$ and~$p$ only, we can write
\begin{align*}
    &\left|\beta\dblangle{ p\cdot\nabla_q W \star \projp (\mescin_t - \mescin_\infty), h}\right|
    \leq (1+\eps)\sqrt{d\beta}\normT{\nabla W}\|h\|_{L^1(\mescin_\infty)}\norml{h},
\end{align*}
where we used~\eqref{eq:bound_convolution}. Using Young's inequality with~$\delta = 2$ gives
\begin{align*}
    &(1+\eps)\sqrt{d\beta}\normT{\nabla W}\|h\|_{L^1(\mescin_\infty)}\norml{h_t} \\
    &\qquad \leq \frac{(1+\eps)^2d\beta}{4}\normT{\nabla W}^2\|h\|_{L^1(\mescin_\infty)} + \|h\|_{L^1(\mescin_\infty)}\norml{h_t}^2,
\end{align*}
from which the bound~\eqref{eq:new_remainder} follows.
\end{proof}
\subsubsection{Convergence to equilibrium}
\label{subsub:equilibrium}
We now have all the arguments we need to establish the main result of this section. The key property is the following coercivity result.
\begin{proposition}
    \label{prop:gen_non_lin_trou_spec}
    Let~$\mescin_\infty$ be a steady state of~\eqref{eq:fok-pl-vlas} and~$\mu_0$ be the initial probability distribution, with
    $
    \frac{\mescin_0}{\mescin_\infty} \in L^2(\mescin_\infty).
    $
    For all $\beta>0$. Then, there exists $\overline{\eps}_{\beta} >0$ such that, for $\eps_{\beta} = \overline{\eps}_{\beta}\min(\gamma,\gamma^{-1}) \in (0,1)$ in \eqref{eq:inner_prod_L2_hypo}, we have
        \begin{align}
        \label{eq:bound_with_cubic_norm}
         \frac{\d}{\d t}\mc{H}(h_t)
         &\leq -\left(\kappa_\beta
         - \left( M_\beta + 1
         \right)\!\|\mescin_t - \mescin_\infty\|_{L^1(\mc{E})}\right)\norml{h_t}^2 + K_\beta \|\mescin_t - \mescin_\infty\|_{L^1(\mc{E})},
        \end{align}
        where~$M_{\beta}$ and~$K_{\beta}$ are defined in~\cref{lem:bound_deriv_second_H} and~$\kappa_{\beta} > 0$ is the smallest eigenvalue of the matrix
        \begin{align}
        \label{eq:matrix_S}
        \widetilde S_{\beta} &= \begin{pmatrix}
            \displaystyle\eps_{\beta}\lambham & \displaystyle-\frac{\eps_{\beta}}{2}\!\left(\kham + \frac{\gamma}{2}\right) \\
            \displaystyle-\frac{\eps_{\beta}}{2}\!\left(\kham + \frac{\gamma}{2}\right) & \displaystyle \frac{\gamma}{2}- \eps_{\beta}
            \end{pmatrix}.
    \end{align}
\end{proposition}
\begin{proof}
    By~\cref{prop:hypo_gen_lin,lem:bound_deriv_second_H},
    \begin{align*}
        \frac{\d}{\d t}\mc{H}(h_t) &\leq -\eps_\beta X^\top S_\beta X - \frac{\gamma}{2\beta}\norml{\nabla_p h_t}^2\\
        &\quad + M_\beta \|\mescin_t - \mescin_\infty\|_{L^1(\mc{E})} \norml{h_t}^2
        + \left(K_\beta + \norml{h_t}^2\right)\|\mescin_t - \mescin_\infty\|_{L^1(\mc{E})}.
    \end{align*}
    Using the Poincaré inequality for Gaussian measures, we have
    \begin{align*}
        \frac{\d}{\d t}\mc{H}(h_t)
        &\leq -X^\top \widetilde S_{\beta} X
        + \left( M_\beta \|\mescin_t - \mescin_\infty\|_{L^1(\mc{E})}
        + \|\mescin_t - \mescin_\infty\|_{L^1(\mc{E})}\right)\!\norml{h_t}^2 \\
        &\quad + K_\beta \|\mescin_t - \mescin_\infty\|_{L^1(\mc{E})}.
    \end{align*}
    There exists~$\eps_{\beta} = \overline{\eps}_{\beta}\min(\gamma,\gamma^{-1})$ such that the smallest eigenvalue~$\kappa_{\beta}$ of~$\widetilde S_{\beta}$ is positive; see~\cite[Proof of Proposition 1]{roussel2018spectral} for more details. Then,~\eqref{eq:bound_with_cubic_norm} follows and concludes the proof.
\end{proof}
We are now able to prove~\cref{prop:exp_decay_nonlin}.
\begin{proof}[Proof of~\cref{prop:exp_decay_nonlin}]
    First, by the Csiszár--Kullback--Pinsker inequality, we have
    $
        \|\mescin_t - \mescin_\infty\|_{L^1(\mc{E})}^2 \leq 2 \, \text{Ent}(\mescin_t|\mescin_\infty).
    $
    Since~$\mescin_\infty$ is a local minimizer of the free energy functional~\eqref{eq:mf_free_energy}, by using~\cite[Theorem 20]{monmarche2024local}, whose assumptions are easily verified on~$\torus^d$ for smooth potentials, there exist~$\delta_\beta, c_\beta, C_\beta >0$ such that, if~$\mc{W}_2(\mescin_0,\mescin_\infty) \leq \delta_\beta$, then
    \begin{align}
        \label{eq:local_cvg_entropy}
        \|\mescin_t - \mescin_\infty\|_{L^1(\mc{E})}^2 \leq 2 \, \text{Ent}(\mescin_t|\mescin_\infty) \leq C_\beta^2 \text{Ent}(\mescin_0|\mescin_\infty)\ee^{-2c_\beta t}.
    \end{align}
Set $E_0=\text{Ent}(\mescin_0|\mescin_\infty)$ and
\[
    a_\beta=\frac{2\kappa_\beta}{1+\eps_\beta},\qquad
    b_\beta=\frac{2(M_\beta+1)C_\beta\sqrt{E_0}}{1-\eps_\beta},
    \qquad d_\beta=K_\beta C_\beta\sqrt{E_0}.
\]
Combining~\eqref{eq:local_cvg_entropy}, \eqref{eq:bound_with_cubic_norm}
and~\eqref{eq:norm_eq} gives, with $H(t)=\mc H(h_t)$,
\[
    H'(t)\leq(-a_\beta+b_\beta\e^{-c_\beta t})H(t)
    +d_\beta\e^{-c_\beta t}.
\]
Gr\"onwall's inequality implies
\[
    H(t)\leq\e^{b_\beta/c_\beta}
    \left(H(0)\e^{-a_\beta t}
    +d_\beta\int_0^t\e^{-a_\beta(t-s)}\e^{-c_\beta s}\,\d s\right)
    \leq C_\rho\e^{-\rho t},
\]
for any $0<\rho<\min(a_\beta,c_\beta)$.
The norm equivalence~\eqref{eq:norm_eq} yields
$\norml{h_t}\leq L_\beta\e^{-\rho t/2}$, proving the result.
\end{proof}
\subsection{Ergodicity for the mean-field dynamics}
\label{subsec:ergo_nonlin}
This section is dedicated to the proof of~\cref{lem:ergodicity}. The proof is mainly based on It\^{o}'s  formula and calculations from~\cref{sec:cvg_eq}. Before giving the proof of~\cref{lem:ergodicity}, we need the following estimates. The first one is a straightforward adaptation of the results of~\cite{talay2002stochastic,kopec2015weak} when~$q\in\torus^d$. The second one is an adaptation of~\cite[Lemma 3.4]{chaintron2021propagationII}. The proofs are ommitted for brevity.
\begin{lemma}[{\cite[Corollary 6.21]{kopec2015weak}}]
    \label{lem:Talay_Kopec}
    Let~$\SolPois$ be the solution of~\eqref{eq:sol_poisson_MV_lin}. Then, there exists~$C>0$ and~$\alpha \in\mathbb{N}$
    \begin{align}
    \label{eq:borne_Poisson}
    \forall (q,p) \in \mc{E},\qquad |\SolPois(q,p)| + |\nabla_p \SolPois(q,p)| \leq C(1+|p|^\alpha).
    \end{align}
\end{lemma}
\begin{lemma}
    \label{lemma:bound_moment_p}
    Let~$(q_t,p_t)_{t\geq 0}$ be the process~\eqref{eq:kin_mv_nonlin}.
    Under~\cref{assumption:moments_p}, for any~$n\geq0$, there exists~$c_n >0$ such that
    \begin{align}
        \label{eq:bound_moment_p}
        \sup_{t\geq0}\expect\!\left[1+|p_t|^{2n}\right] \leq c_n < +\infty.
    \end{align}
\end{lemma}
We mention that~\cref{lem:Talay_Kopec,lemma:bound_moment_p} will also be useful for the proof of~\cref{th:lim_nonlin_mvt_brow}. We are now able to give the proof of~\cref{lem:ergodicity}.
\begin{proof}[Proof of~\cref{lem:ergodicity}]
Redefining $f$ as $f - \theta$ if necessary,
we can assume without loss of generality that $\theta = 0$.
We wish to show the convergence of $\widehat \theta(T)$ to 0 as $T \to \infty$ in $L^2(\Omega)$.
To this end, we note that
\begin{align*}
    \expect\!\left[ \widehat \theta(T)^2 \right]
    &= \frac{1}{T^2} \int_{0}^{T} \int_{0}^{T} \expect\Bigl[ f(q_s, p_s) f(q_t, p_t) \Bigr] \, \d s \, \d t \\
    &= \frac{2}{T^2} \int_{0}^{T} \int_{0}^{t} \expect\Bigl[ f(q_s, p_s) f(q_t, p_t) \Bigr] \, \d s \, \d t \, .
\end{align*}
Write $u(s, t, q, p) = \expect\bigl[ f(q_t, p_t) | (q_s, p_s) = (q, p)\bigr]$.
Therefore,
\begin{align}
    \notag
    \expect\left[ \widehat \theta(T)^2 \right]
    &= \frac{2}{T^2} \int_{0}^{T} \int_{0}^{t} \expect\Bigl[ f(q_s, p_s) u(s, t, q_s, p_s) \Bigr] \, \d s \, \d t \, \\
    \notag
    &=
    \frac{2}{T^2} \int_{0}^{T} \int_{0}^{t} \int_{\mathcal E}  f(q, p) u(s, t, q, p) \, \left(\frac{\mu_s}{\mu_{\infty}} - 1\right) \mu_{\infty}(\d q\,\d p)  \, \d s \, \d t \\
    \label{eq:mse_decomposition}
    &\qquad + \frac{2}{T^2} \int_{0}^{T} \int_{0}^{t} \int_{\mathcal E}  f(q, p) u(s, t, q, p) \, \mu_{\infty} (\d q\,\d p) \, \d s \, \d t \, .
\end{align}
For the first term in the right-hand side of~\eqref{eq:mse_decomposition}, we use the Cauchy--Schwarz inequality to obtain
\begin{align}
    &\left| \int_{\mathcal E}  f(q, p) u(s, t, q, p) \, \left(\frac{\mu_s}{ \mu_{\infty}} - 1\right) \mu_{\infty}(\d q\,\d p) \right| \\
    &\qquad\qquad\leq \|f\|_{L^4(\mu_{\infty})} \lVert u(s, t, \cdot, \cdot) \rVert_{L^4(\mu_{\infty})}
    \left\|\frac{ \mu_s}{\mu_{\infty}} - 1\right\|_{L^2(\mu_{\infty})} \!.
\end{align}
By assumption, since~$f \in L_{\mc{K}}^\infty$, it follows that $f \in L^4(\mu_{\infty})$,
so the first factor on the right-hand side is bounded.
The second factor is also in~$L^4(\mu_{\infty})$,
since by definition of $u$ and by~\cref{lemma:bound_moment_p}, there exists~$\alpha \in \mathbb{N}$ such that
\begin{align*}
   |u(s, t, q, p)| &= \left|\expect\bigl[ f(q_t, p_t) | (q_s, p_s) = (q, p)\bigr]\right|\\
    &\leq \|f\|_{L^{\infty}_{\mathcal K }} \, \expect\bigl[ 1 + |p_t|^{\alpha} \bigm| (q_s, p_s) = (q, p)\bigr]
    \leq C \left( 1 + |p|^{\alpha} \right).
\end{align*}
The third factor decays exponentially fast by assumption.
It remains to bound the second term in~\eqref{eq:mse_decomposition}.
By the Cauchy--Schwarz inequality,
\begin{align*}
            &\left|\frac{2}{T^2} \int_{0}^{T} \int_{0}^{t} \int_{\mathcal E}  f(q, p) u(s, t, q, p) \, \mu_{\infty}(\d q\,\d p)  \, \d s \, \d t \right| \\
        &\qquad \quad\leq \frac{2}{T^2} \int_{0}^{T} \int_{0}^{t} \lVert f \rVert_{L^2(\mu_{\infty})}
        \lVert u(s, t, \cdot, \cdot) \lVert_{L^2(\mu_{\infty})} \,  \d s \, \d t \, .
\end{align*}
To conclude the proof,
it suffices to show that, for all $0 \leq s \leq t$,
\begin{align}
    \label{eq:dec_expo_kolmog}
    \lVert u(s, t, \cdot, \cdot) \lVert_{L^2(\mu_{\infty})}
    \leq C \e^{-\lambda (t-s)} \, .
\end{align}
To this end, we use the fact that, by It\^{o}'s formula,
the function $u$ satisfies the backward Kolmogorov equation
\[
    \partial_s u + \mathsf L_s u = 0 \, , \qquad u(t, t, q, p) = f(q, p) \, ,
\]
where we introduced the time-dependent generator
\[
    \mathsf L_s = p \cdot \nabla_q - \Bigl(\nabla V(q) + \nabla W \star \Pi \mu_s (q) \Bigr) \cdot \nabla_p - \gamma p \cdot \nabla_p + \frac{\gamma}{\beta} \Delta_p \, .
\]
Define $\widetilde{u}^t(r, q, p) = u(t - r, t, q, p)$ for~$0\leq r\leq t$.
Then, $\widetilde u$ satisfies
\begin{align*}
    \partial_r \widetilde u^t
    &= \mathsf L_{t-r} \widetilde u^t \, , \qquad \widetilde u^t(0, q, p) = f(q, p) \, .
\end{align*}
By changing the sign of~$\eps$ in~\eqref{eq:inner_prod_L2_hypo}, and since~$\Lham$ is antisymmetric, the estimates obtained in~\cref{sec:cvg_eq} hold for~$\genlin$ instead of~$\adj$ in all proofs. Consider
\begin{align}
    \label{eq:inner_prod_L2_hypo_bis}
    \dblangle{\varphi_1,\varphi_2} = \scal{\varphi_1,\varphi_2} - \eps\scal{A\varphi_1,\varphi_2} - \eps\scal{\varphi_1,A\varphi_2}.
\end{align}
Thus, we have, with the notation $\widetilde u_r^t = \widetilde u^t(r, \cdot, \cdot)$,
\begin{align*}
    \frac{1}{2} \frac{\d}{\d r} \dblangle{\widetilde u_r^t, \widetilde u_r^t}
    &=  \dblangle{\mathsf L_{t-r} \widetilde u_r^t, \widetilde u_r^t}
    =  \dblangle{\mathsf L_{\infty} \widetilde u_r^t, \widetilde u_r^t}
    +  \dblangle{(\mathsf L_{t-r} - \mathsf L_{\infty}) \widetilde u_r^t, \widetilde u_r^t} \, \\
    &=   \dblangle{\mathsf L_{\infty} \widetilde u_r^t, \widetilde u_r^t}
    +  \dblangle{\bigl(\nabla W \star (\Pi \mu_{\infty} - \Pi \mu_{t-r})\bigr) \cdot \nabla_p \widetilde u_r^t, \widetilde u_r^t} \, \\
    &=   \dblangle{\mathsf L_{\infty} \widetilde u_r^t, \widetilde u_r^t}
    -  \dblangle{\bigl(\nabla W \star \Pi (h_{t-r} \mu_{\infty})\bigr) \cdot \nabla_p \widetilde u_r^t, \widetilde u_r^t} \, ,
\end{align*}
with $h_{t-r} = \displaystyle\frac{\mu_{t-r}}{\mu_{\infty}} - \textbf{1}$.
Since $\nabla W$ is bounded,
it holds that
\begin{align*}
    \lVert \nabla W \star \Pi (h_{t-r} \mu_{\infty}) \rVert_{L^{\infty}(\torus^d)}
    &\leq \lVert \nabla W \rVert_{L^{\infty}(\torus^d)} \|\Pi (h_{t-r} \mu_{\infty})\|_{L^1(\torus^d)} \\
    &\leq \lVert \nabla W \rVert_{L^{\infty}(\torus^d)} \|h_{t-r}\|_{L^2(\mu_{\infty})}
    \leq L \lVert \nabla W \rVert_{L^{\infty}(\torus^d)} \e^{-\lambda (t-r)} \, ,
\end{align*}
where we used the assumption on $\mu_{t-r}$ in the last line.
Therefore, using~\eqref{eq:cauchy_schwarz} for~\eqref{eq:inner_prod_L2_hypo_bis}, we have
\begin{align*}
    &\left|\dblangle{\bigl(\nabla W \star \Pi (h_{t-r} \mu_{\infty})\bigr) \cdot \nabla_p \widetilde u_r, \widetilde u_r}\right| \\
    &\qquad \qquad\leq (1+\eps)L \lVert \nabla W \rVert_{L^{\infty}(\torus^d)} \e^{- \lambda (t - r)} \|\nabla_p \widetilde u_r\|_{L^2(\mu_{\infty})} \|\widetilde u_r\|_{L^2(\mu_{\infty})} \,
     \\
    &\qquad \qquad\leq 2L \lVert \nabla W \rVert_{L^{\infty}(\torus^d)} \e^{- \lambda (t - r)} \|\nabla_p \widetilde u_r\|_{L^2(\mu_{\infty})} \|\widetilde u_r\|_{L^2(\mu_{\infty})} \,.
\end{align*}
Now, the counterpart of~\cref{prop:hypo_gen_lin} in this context reads
\[
\dblangle{\mathsf L_{\infty} \widetilde u_r^t, \widetilde u_r^t} \leq -X_r^\top S_\beta X_r -\frac{\gamma}{\beta}\norml{\nabla_p \widetilde u_r^t}^2,
\]
with~$X_r = (\norml{\projvit\widetilde u_r^t},\norml{(1-\projvit)\widetilde u_r^t})^\top$ and~$S_\beta$ defined in~\eqref{eq:matrixS}.
By substituting this bound and using Young's inequality, we obtain
\begin{align*}
    &\frac{1}{2} \frac{\d}{\d r} \dblangle{\widetilde u_r^t, \widetilde u_r^t}\\
    &\leq
    -X_r^\top S_\beta X_r -\frac{\gamma}{\beta}\!\norml{\nabla_p \widetilde u_r^t}^2
    \!+ \!2L \lVert \nabla W \rVert_{L^{\infty}(\torus^d)} \!\e^{- \lambda (t - r)} \!\|\nabla_p \widetilde u_r^t\|_{L^2(\mu_{\infty})} \|\widetilde u_r^t\|_{L^2(\mu_{\infty})} \, \\
    &\leq
    -X_r^\top S_\beta X_r -\frac{\gamma}{2\beta}\norml{\nabla_p \widetilde u_r^t}^2
    + \frac{2\beta L^2}{\gamma} \lVert \nabla W \rVert_{L^{\infty}(\torus^d)}^2 \e^{- 2\lambda (t - r)} \|\widetilde u_r^t\|^2_{L^2(\mu_{\infty})} \, \\
    &\leq
    - \left( \lambda_\beta - \frac{ 2\beta L^2}{\gamma} \lVert \nabla W \rVert_{L^{\infty}(\torus^d)}^2
    \e^{- 2\lambda (t - r)} \right)\|\widetilde u_r^t\|^2_{L^2(\mu_{\infty})},
\end{align*}
where~$\lambda_\beta>0$ is the smallest positive eigenvalue of the matrix~\eqref{eq:matrix_S},
which is strictly positive provided that~$\eps$ is sufficiently small.
By Gr\"onwall's inequality and the equivalence of the norms~\eqref{eq:norm_eq}, we deduce that
\[
    \dblangle{\widetilde u_r^t, \widetilde u_r^t}
    \leq \dblangle{f, f}
    \exp\left( - \frac{2}{1+\eps} \lambda_\beta r
    + \frac{4\beta L^2}{(1 - \varepsilon) \gamma} \lVert \nabla W \rVert_{L^{\infty}(\torus^d)}^2
    \int_{0}^{r} \e^{- 2\lambda (t - s)} \, \d s \right).
\]
Since~$0\leq r \leq t$, we have $ \int_{0}^{r} \e^{- 2\lambda (t - s)} \, \d s = \frac{1}{2\lambda} \bigl(\e^{- 2\lambda (t - r)} - \e^{-2\lambda t}\bigr) \leq \frac{1}{2\lambda}.$ We deduce that, for some generic constant~$C > 0$, that
\[
    \dblangle{\widetilde u_r^t, \widetilde u_r^t} \leq C \e^{- \frac{2\lambda_\beta}{1+\eps} r},
\]
which gives~\eqref{eq:dec_expo_kolmog} by using~\eqref{eq:norm_eq} and thus concludes the proof.
\end{proof}
%
%
\section{Proof of~\texorpdfstring{\cref{th:lim_nonlin_mvt_brow}}{Theorem~\ref{th:lim_nonlin_mvt_brow}}}
\label{sec:diffusive_MF}
In this section, we prove~\cref{th:lim_nonlin_mvt_brow}.
The key observation is to think of the nonlinear generator of the McKean--Vlasov SDE~\eqref{eq:kin_mv_nonlin} as a perturbation of the generator of the linearized process.
Indeed, by~\eqref{eq:generator_nonlin} end~\eqref{eq:generator_lin} we have
\begin{align}
    \label{eq:gen_non_comp_lin}
    \gennonlin{t}
    = \genlin- \nabla W \star  \projp(\mescin_t -\mescin_\infty)\cdot \nabla_p.
\end{align}
The proof of~\cref{th:lim_nonlin_mvt_brow} relies on an adaptation of the proof of~\cite[Theorem 1.5]{olla1994homogenization} and~\cite[Theorem 1]{fathi2015error}.

\begin{proof}[Proof of~\cref{th:lim_nonlin_mvt_brow}]
The proof is essentially an adaptation of~\cite[Theorem 1.5]{olla1994homogenization} or~\cite[Theorem 1]{fathi2015error} but, thanks to~\eqref{eq:hyp_cvg_eq}, the remainder term that appears when using the It\^o's lemma vanishes. A direct application of It\^o's lemma on~$\SolPois(q_s,p_s)$, given in~\eqref{eq:sol_poisson_MV_lin}, gives
\begin{align*}
    Q_t - Q_0 &= \int_0^t p_s\, \d s = \int_0^t -\genlin\SolPois(q_s,p_s)\, \d s \\
    &= \SolPois(q_0,p_0)-\SolPois(q_t,p_t) \\
    &\quad +\sqrt{\frac{2\gamma}{\beta}}\int_0^t \nabla_p\SolPois(q_s,p_s)\cdot\d B_s
    - \int_0^t (\genlin- \gennonlin{s}) \SolPois(q_s,p_s)\, \d s.
\end{align*}
First, by using~\cref{lem:Talay_Kopec,lemma:bound_moment_p}, we have that
\[
\frac{1}{t^2}\expect\Bigl[|\SolPois(q_0,p_0)-\SolPois(q_t,p_t)|^2\Bigr],
\]
vanishes when~$t \to \infty$.
Then, by~\eqref{eq:generator_nonlin} and~\eqref{eq:generator_lin}, we have
\begin{align*}
    \left|\int_0^t (\genlin- \gennonlin{s}) \SolPois(q_s,p_s)\, \d s\right|
    &= \left|\int_0^t \nabla W \star \projp(\mescin_s - \mescin_\infty)(q_s)\cdot \nabla_p\SolPois(q_s,p_s)\, \d s\right| \\
    \notag
    &\leq \int_0^t \left|\nabla W \star \projp(\mescin_s - \mescin_\infty)(q_s)\right|\left|\nabla_p\SolPois(q_s,p_s)\right| \d s.
\end{align*}
For any~$q \in \torus^d$,~\eqref{eq:bound_convolution} gives
\begin{align}
    \label{eq:bound_remainder}
    |\nabla W \star \projp(\mescin_s - \mescin_\infty)(q)|
   \leq \normT{\nabla W} \|\mescin_s - \mescin_\infty\|_{L^1(\torus^d\times\real^d)}.
\end{align}
By~\eqref{eq:hyp_cvg_eq},
it then holds that
\begin{align}
    \label{eq:terme_perturb}
    \left|\int_0^t (\genlin- \gennonlin{s}) \SolPois(q_s,p_s)\, \d s\right|
    \leq C\normT{\nabla W}\int_0^t \e^{-\lambda s}\left|\nabla_p\SolPois(q_s,p_s)\right| \d s.
\end{align}
By considering~$Q_t^\eps$ and using~\eqref{eq:borne_Poisson} in~\eqref{eq:terme_perturb}, we know that there exists~$\alpha\in\mathbb{N}$ for which
\begin{align}
    \label{eq:ineq_finale_diff}
    \eps^2\left|\int_0^{t/\eps^2} (\genlin- \gennonlin{s}) \SolPois(q_s,p_s)\, \d s\right|^2
    \leq  K^2\eps^2\!\left(\int_0^{t/\eps^2} \e^{-\lambda s} (1+|p_s|^\alpha)\,\d s\right)^2,
\end{align}
with~$K = C\normT{\nabla W}$. Using a Cauchy--Schwarz inequality, we obtain
\begin{align*}
    &\eps^2\!\left|\int_0^{t/\eps^2} (\genlin- \gennonlin{s}) \SolPois(q_s,p_s)\, \d s\right|^2 \\
    &\qquad\qquad\leq K^2\eps^2 \!\left(\int_0^{t/\eps^2} \e^{-\lambda s}\, \d s\!\right)
    \!\left(\int_0^{t/\eps^2} \e^{-\lambda s} (1+|p_s|^\alpha)^2\,\d s\right)\!.
\end{align*}
Using the fact that~$(a+b)^2 \leq 2a^2 + 2b^2$, we have
\begin{align}
    \label{preque_fini}
    \eps^2\left|\int_0^{t/\eps^2} (\genlin- \gennonlin{s}) \SolPois(q_s,p_s)\, \d s\right|^2
    \leq  \frac{2 K^2}{\lambda} \eps^2\int_0^{t/\eps^2}\e^{-\lambda s} (1+|p_s|^{2\alpha})\,\d s.
\end{align}
By taking expectations and using~\cref{lemma:bound_moment_p}, there exists~$c_\alpha \geq 0$ such that
\begin{align}
    \label{fini}
     \frac{2 K^2}{\lambda} \eps^2\int_0^{t/\eps^2} \e^{-\lambda s} \expect\left[1+|p_s|^{2\alpha}\right]\,\d s \leq \frac{2K^2}{\lambda^2}  c_\alpha \eps^2.
\end{align}
This shows that the term on the left hand side of~\eqref{eq:ineq_finale_diff} vanishes as~$\eps \to 0$. Therefore, the limiting properties of~$Q_t^\eps$ depend only on the martingale term
\begin{align*}
 \mc{M}^{\eps}_t :=  \eps\sqrt{\frac{2\gamma}{\beta}}\int_0^{t/\eps^2} \nabla_p\SolPois(q_s,p_s)\cdot\d B_s,
\end{align*}
whose quadratic variation is
    \[
    \langle \mc{M}^{\eps} \rangle_t = \frac{2\gamma \eps^2}{\beta} \int_0^{t/\eps^2} (\nabla_p \SolPois)^\top \nabla_p \SolPois \, \d t.
    \]
    Since~\eqref{eq:hyp_cvg_eq} holds, we can apply~\cref{lem:ergodicity}. Using the results of~\cite[Section 4.2]{mattingly2010convergence}, which rely on Borel--Cantelli lemma, the convergence in~\cref{lem:ergodicity} holds for the almost sure convergence. Hence, we have that
    \[
    \lim_{\eps \to 0} \langle \mc{M}^{\eps} \rangle_t = \frac{2\gamma t}{\beta} \int_{\torus^d \times \real^d}( \nabla_p \SolPois)^\top \nabla_p \SolPois \, \d \mescin_{\infty}
    = 2t \mc{D}_\infty,
    \]
    almost surely.
    More generally, the martingale differences $\mathfrak{M}^{\eps}_{t,s} = \mc{M}^{\eps}_t - \mc{M}^{\eps}_{s}$
    are such that
    \[
    \lim_{\eps \to 0} \langle \mathfrak{M}^{\eps} \rangle_{t,s} = 2(t - s)\mc{D}_\infty.
    \]
    Now consider a unit vector $\mathbf{F} \in \mathbb{R}^d$ and~$\mathfrak{M}^{\eps, \mathbf{F}}_{t,s} = \mathbf{F}^\top \mathfrak{M}^{\eps}_{t,s}$.
    To conclude the proof of the convergence in law of the process~$Q^\eps_t$ to a Brownian motion with covariance matrix~$2\mc{D}_\infty$, we proceed as in~\cite[Proof of Theorem 1.5]{olla1994homogenization}. 
    Indeed, we have to prove the following:
    \begin{itemize}
        \item For any $0\leq t_0 \leq t_1 \leq \dots \leq t_n$, the process
        \[
        \left( \mathfrak{M}^{\eps, \mathbf{F}}_{t_k, t_{k-1}} \right)_{k=1,\dots,n}
        \]
        converges in law to a Gaussian vector with covariance matrix:
        \[
        \text{Diag} \left( 2(t_k - t_{k-1}) \mathbf{F}^\top \mc{D}_\infty \mathbf{F} \right)_{1\leq k\leq n}.
        \]
        This result can be proved by induction by considering the exponential martingales
        \[
        H^{\eps, \theta}_{t, s} = \exp \left( \mathrm{i}\theta \mathfrak{M}^{\eps, \mathbf{F}}_{t, s} + \frac{\theta^2}{2}      \langle \mathfrak{M}^{\eps, \mathbf{F}} \rangle_{t,s} \right),
        \]
        which are such that the conditional expectations converge almost surely as~$\eps \to 0$ to those of a Brownian motion with covariance~$2\mc{D}_\infty$:
        \[
        \lim_{\eps \to 0} \mathbb{E} \left[ \exp \left( \mathrm{i}\theta \mathfrak{M}^{\eps, \mathbf{F}}_{t, s}\right)\middle| \mc{F}_{s/\eps^2} \right] = \exp \left( -\theta^2(t - s) \mathbf{F}^\top \mc{D}_\infty \mathbf{F} \right),
        \]
        where~$\mc{F}_{s/\eps^2}$ denotes the filtration of the process up to time~$s/\eps^2$.
        \item the tightness of the process, proved using Prohorov's criterion (see for instance~\cite{billingsley2013convergence}):
        \[
        \forall \alpha, \tau > 0, \qquad \lim_{\delta \to 0} \limsup_{\eps \to 0} \mathbb{P}\left( \sup_{\substack{ |t-s|<\delta \\ 0 \leq s < t \leq \tau}} \left|  \mathbf{F}^\top \Big( Q_{t}^{\eps} - Q_s^{\eps} \Big)\right| \geq \alpha \right) = 0.
        \]
        This criterion is satisfied in view of the tightness of the martingale $\mathfrak{M}^{\mathbf{F},\eps}_{t,0}$, itself easily obtained using Doob's inequality.
    \end{itemize}
    Hence the proof is complete.
\end{proof}
\section{Proof of~\texorpdfstring{\cref{th:eps_then_N}}{Theorem~\ref{th:eps_then_N}}}
    \label{sec:lim_mat_system}
  We present in~\cref{sec:preli_poincare} preliminary results and key bounds on the~$L^2(\mescin_N)$-norm of the solution of the Poisson equation~\eqref{eq:poisson_eq_coordinate}. We then use these results in~\cref{sec:proof_lim_cov_particle_system} to prove~\cref{th:eps_then_N}.
\subsection{Preliminary results and estimates on the Poisson solution}
\label{sec:preli_poincare}
We establish preliminary results on the solution of~\eqref{eq:poisson_eq_coordinate}, which is unique for an underdamped Langevin dynamics with bounded potential~\cite[Theorem 1]{kozlov1989effective}.
We recall that we work on the configuration space~$\mc{E}^N = (\torus^d)^N \times (\real^d)^N$ and consider the potential~$U_N$ given by~\eqref{eq:potential_U}.
The Hessian matrix of~$U_N$ can be written as
\[
\nabla^2 U_N = \MNW + \MNV,
\]
where~$\MNW$ and~$\MNV$ are defined as follows.
For~$1 \leq i,j \leq N$, the~$(i,j)$-th block of size~$d \times d$ of the matrix~$\MNW$ is given by
\begin{align*}
    (\MNW)_{i,j} =
    \begin{cases}
        -\displaystyle\frac{1}{N} \displaystyle\nabla^2 W(\vq{i} - \vq{j}), & \text{if } i \neq j, \\
        \displaystyle\frac{1}{N} \displaystyle\sum_{\substack{k=1 \\ k \neq i}}^{N} \nabla^2 W(\vq{i} - \vq{k}), & \text{if } i = j,
    \end{cases}
\end{align*}
and the~$(i,j)$-th block of size~$d \times d$ of the matrix~$\MNV$ is given by
\begin{align*}
    (\MNV)_{i,j} =
    \begin{cases}
        0, & \text{if } i \neq j, \\
        \nabla^2 V(\vq{i}), & \text{if } i = j.
    \end{cases}
\end{align*}
We have the following bound on these matrices.
\begin{lemma}[{\cite[Lemma 6]{guillin2021kinetic}}]
    \label{lem:operator_norm_hessian}
    Suppose that~$V$ and~$W$ are~$\mc{C}^2$ on~$\torus^d$.
    Then, in the sense of symmetric matrices,
    \begin{align*}
        -K_W I_{Nd} \leq \MNW \leq K_W I_{Nd}, \qquad -K_V I_{Nd} \leq \MNV \leq K_V I_{Nd},
    \end{align*}
    where~$K_V = \normT{\nabla^2 V}$ and~$K_W = \normT{\nabla^2 W}$ do not depend on~$N$.
    Therefore,
    \begin{align}
        \label{eq:U_N_norm_sym_mat}
        -(K_W+K_V) I_{Nd} \leq \nabla^2 U_N \leq (K_W+K_V) I_{Nd}.
    \end{align}
\end{lemma}
We next establish some bounds on the gradients in~$q$ and~$p$ of~\eqref{eq:poisson_eq_coordinate}.
With this aim, we introduce for any~$a>0$ the following scalar product defined on~$H^1(\mescin_N)$:
\begin{align}
    \label{eq:prod_scal_H1_bis}
\forall u,v\in H^1(\mescin_N),\qquad
\slangle{u,v}_a := \slangle{u,v} + a \slangle{(\nabla_p-\nabla_q)u,(\nabla_p-\nabla_q)v},
\end{align}
where~$\slangle{\cdot,\cdot}$ is the canonical scalar product in~$L^2(\mescin_N)$. This choice of scalar product is motivated by~\cite{LetiziaOlla2017}. The first important result is the following.
\begin{lemma}
    \label{lem:born_inf_norm_H1}
    There exists~$a_*>0$ such that, for any~$a \in (0,a_*)$, there is~$\kappa_a >0$ for which
    \begin{align*}
        \forall N,\qquad \forall u\in H^1(\mescin_N),\qquad\slangle{-\mc{L}_N u ,u}_a \geq \kappa_a \left(\norm{\nabla_p u}^2+\norm{\nabla_q u}^2\right),
    \end{align*}
    with~$\kappa_a$ only depending on~$\gamma$, $\beta$,~$d$,~$V$ and~$W$.
\end{lemma}
\begin{proof}
  Note that
\begin{align}
    \label{eq:minoration}
\slangle{-\mc{L}_N u ,u}_a = \frac{\gamma}{\beta}\norm{\nabla_p u}^2
+ a\slangle{-(\nabla_p - \nabla_q)\mc{L}_N u ,(\nabla_p - \nabla_q)u}\!.
\end{align}
Straightforward computations give
\begin{align}
    \label{eq:commutators}
    \nabla_q \mc{L}_N = -\nabla^2U_N\nabla_p + \mc{L}_N\nabla_q, \\
    \nabla_p \mc{L}_N = \nabla_q - \gamma\nabla_p + \mc{L}_N\nabla_p,
\end{align}
where~$\mc{L}_N\nabla_p$ is to be understood as the operator~$\mc{L}_N$ applied to each coordinate of~$\nabla_q$.
Inserting~\eqref{eq:commutators} into~\eqref{eq:minoration} gives
\begin{align*}
    \slangle{-\mc{L}_N u ,u}_a = &\frac{\gamma}{\beta}\norm{\nabla_p u}^2 + a\slangle{-\mc{L}_N(\nabla_p - \nabla_q) u ,(\nabla_p - \nabla_q)u} \\
    &+ a \slangle{(\gamma -\nabla^2U_N) \nabla_pu,(\nabla_p - \nabla_q)u} + a\slangle{-\nabla_qu,(\nabla_p - \nabla_q)u}\!.
\end{align*}
Since
\[
\slangle{-\mc{L}_N(\nabla_p - \nabla_q) u ,(\nabla_p - \nabla_q)u} = \frac{\gamma}{\beta}\norm{\nabla_p(\nabla_p - \nabla_q)u}^2 \geq 0,
\]
the following lower bound holds:
\begin{align*}
    \slangle{-\mc{L}_N u ,u}_a &\geq \frac{\gamma}{\beta}\norm{\nabla_p u}^2
    + a \slangle{(\gamma -\nabla^2U_N) \nabla_p u,(\nabla_p - \nabla_q)u} \\
    &\quad + a\slangle{-\nabla_q u,(\nabla_p - \nabla_q)u}.
\end{align*}
Next, by~\eqref{eq:U_N_norm_sym_mat}, and denoting by~$K = K_W + K_V$,
\begin{align}
    \label{eq:hypo_vil1}
a\slangle{(\gamma -\nabla^2U_N) \nabla_p u,(\nabla_p - \nabla_q) u} \geq a(\gamma - K)\norm{\nabla_p u}^2 - a(\gamma + K)\norm{\nabla_p u}\norm{\nabla_q u}.
\end{align}
and
\begin{align}
    \label{eq:hypo_vil2}
a\slangle{-\nabla_q u,(\nabla_p - \nabla_q) u} \geq a \norm{\nabla_q  u}^2 - a \norm{\nabla_p u}\norm{\nabla_q u}.
\end{align}
Finally,
\[
\slangle{-\mc{L}_N  u , u}_a \geq \left(\frac{\gamma}{\beta} + a(\gamma - K)\right)\norm{\nabla_p  u}^2 + a \norm{\nabla_q u}^2 -a(\gamma + K +1)\norm{\nabla_p u}\norm{\nabla_q u}.
\]
With~$X = (\norm{\nabla_p u},\norm{\nabla_q u})^\top$, the latter inequality can be reformulated as
\[
\slangle{-\mc{L}_N u ,u}_a \geq X^\top\mc{M}_aX,
\]
with
\[
\mc{M}_a :=\begin{pmatrix}
            \displaystyle\frac{\gamma}{\beta} + a(\gamma - K) & \displaystyle-\frac{a}{2}(\gamma + K + 1) \\
            \displaystyle-\frac{a}{2}(\gamma + K + 1) & a
            \end{pmatrix}.
\]
The next step of the proof is to show the existence of an~$a >0$ such that the smallest eigenvalue of~$\mc{M}_a$ is positive, which is equivalent to requiring
$$
\text{Tr}(\mc{M}_a) = \frac{\gamma}{\beta} + a(\gamma - K +1) >0,
$$
and
$$
\text{Det}(\mc{M}_a) = a\left(\frac{\gamma}{\beta} + a(\gamma - K)\right)-\frac{a^2}{4}(\gamma + K + 1)^2>0.
$$
The first inequality is satisfied when~$a \in \left(0,\displaystyle\frac{\gamma}{\beta(K - \gamma -1)}\right)$ if~$\gamma - K +1 <0$, or when~$a>0$ if~$\gamma - K +1 \geq 0$. 
The second inequality is equivalent to
\[
\frac{\gamma}{\beta}a^{-1}+\gamma - K > \frac14(\gamma + K + 1)^2.
\]
Satisfying these two inequalities is of course possible for~$a>0$ sufficiently small since~$\gamma/\beta >0$. In conclusion, there exists $a_*>0$ such that, for any~$a \in (0,a_*)$, there is~$\kappa_a >0$ so that
\[
\slangle{-\mc{L}_N u ,u}_a \geq \kappa_a \left(\norm{\nabla_p u}^2+\norm{\nabla_q u}^2\right),
\]
which concludes the proof.
\end{proof}
\begin{proposition}
    \label{prop:bound_inde_N}
    Let~$R \in L_0^2(\mescin_N)$. Introduce~$\mc{R}_N \in L_0^2(\mescin_N)$, the unique solution of the Poisson equation
    \[
    -\mc{L}_N \mc{R}_N = (1-\projvit)R =: R_\kappa,
    \]
    where~$\projvit\varphi = \int_{\real^d}\varphi\,\d \mesvit$. Suppose that~$\norm{R_\kappa}_{L^2(\mescin_N)}$,~$\norm{\nabla_p R_\kappa}_{L^2(\mescin_N)}$ and~$\norm{\nabla_q R_\kappa}_{L^2(\mescin_N)}$ are bounded uniformly in~$N$. Then,~$\norm{\nabla_p \mc{R}_N}_{L^2(\mescin_N)}$ and~$\norm{\nabla_q \mc{R}_N}_{L^2(\mescin_N)}$ are also bounded uniformly in~$N$.
\end{proposition}
\begin{proof}
    Let~$a>0$ and recall the definition~\eqref{eq:prod_scal_H1_bis}. Then,
    \begin{align}
        \label{eq:bound_sup_H1}
    \slangle{-\mc{L}_N \mc{R}_N,\mc{R}_N}_a = \slangle{R_\kappa,\mc{R}_N}_a = \slangle{R_\kappa,\mc{R}_N} + a\slangle{(\nabla_p - \nabla_q) R_\kappa,(\nabla_p - \nabla_q)\mc{R}_N}.
    \end{align}
    For the first term on the right-hand side of~\eqref{eq:bound_sup_H1}, we have, by the Cauchy--Schwarz inequality and the Poincaré inequality for Gaussian measures,
    \begin{align*}
    \left|\slangle{R_\kappa,\mc{R}_N}\right| = \left|\slangle{R_\kappa,(1-\projvit)\mc{R}_N}\right|
     \leq \norm{R_\kappa}\norm{(1-\projvit)\mc{R}_N} \leq \norm{R_\kappa}\norm{\nabla_p \mc{R}_N}.
    \end{align*}
    For the second term on the right-hanside of~\eqref{eq:bound_sup_H1}, we have, first by the Cauchy--Schwarz inequality and then by the triangle inequality,
    \begin{align*}
    \left|a\slangle{(\nabla_p - \nabla_q) R_\kappa,(\nabla_p - \nabla_q)\mc{R}_N}\right| \leq a(\norm{\nabla_p R_\kappa}+\norm{\nabla_q R_\kappa})(\norm{\nabla_p \mc{R}_N} + \norm{\nabla_q \mc{R}_N}).
    \end{align*}
    Then, for any~$a>0$,
    \[
    \left|\slangle{-\mc{L}_N \mc{R}_N,\mc{R}_N}_a\right| \leq C_a(\norm{\nabla_p \mc{R}_N} + \norm{\nabla_q \mc{R}_N}),
    \]
    with $C_a = \norm{R_\kappa} + a(\norm{\nabla_p R_\kappa}+\norm{\nabla_q R_\kappa})$.
    By~\cref{lem:born_inf_norm_H1}, there exist~$a>0$ and~$\kappa_a >0$ independent of~$N$ such that
    \begin{align*}
    \kappa_a(\norm{\nabla_p \mc{R}_N}^2 + \norm{\nabla_q \mc{R}_N}^2)
     & \leq \slangle{-\mc{L}_N \mc{R}_N,\mc{R}_N}_a \\
     & \leq C_a(\norm{\nabla_p \mc{R}_N} + \norm{\nabla_q \mc{R}_N}).
    \end{align*}
    Therefore,~$\norm{\nabla_p \mc{R}_N}$ and $\norm{\nabla_q \mc{R}_N}$ can be bounded independently of~$N$, which concludes the proof.
\end{proof}
The direct and very useful consequence of~\cref{prop:bound_inde_N} is the following result providing a bound for the solution to the Poisson equation~\eqref{eq:poisson_eq_coordinate}.
\begin{corollary}
    \label{cor:bound_unif_grad_poisson}
    Let~$\Phi_i^N$, for $i=1,\dotsc, N$, be the solutions of~\eqref{eq:poisson_eq_coordinate}. Then, by~\cref{prop:bound_inde_N},  $\displaystyle\max_{1\leq i \leq N}\norm{\nabla_p \Phi_i^N}_{L^2(\mescin_N)}$ and~$\displaystyle\max_{1\leq i \leq N}\norm{\nabla_q \Phi_i^N}_{L^2(\mescin_N)}$ are bounded uniformly in~$N$.
\end{corollary}
We conclude this section with a useful consequence of~\cref{cor:bound_unif_grad_poisson}.
\begin{corollary}
    \label{cor:bound_unif_grad_partiel_poisson}
    Let~$k \in \{1,\dots,N\}$ and~$i \in \{1,\dots,Nd\}$ such that~$i \notin \{(k-1)d+1,\dots,kd\}$.
    Then
    \begin{align}
        \label{eq:estimate_grad_1particle}
        \left\|\nabla_{\vp{k}}\Phi_i^N\right\|_{L^2(\mescin_N)}^2 \leq \frac{1}{\gamma^2(N-1)} .
    \end{align}
\end{corollary}
\begin{proof}
    By testing~\eqref{eq:poisson_eq_coordinate} against~$\Phi_i^N$ we have
\begin{align*}
    \frac{\gamma}{\beta}\norm{\nabla_p\Phi_i^N}^2_{L^2(\mescin_N)} = \av{-\mathcal{L}_N \Phi_i^N, \Phi_i^N}_{L^2(\mescin_N)} = \av{\vp{i}, \Phi_i^N}_{L^2(\mescin_N)}.
\end{align*}
Using Cauchy--Schwarz and the Poincaré inequality for Gaussian measures,
\begin{align*}
    \frac{\gamma}{\beta}\norm{\nabla_p\Phi_i^N}^2_{L^2(\mescin_N)} & = \av{\vp{i}, \Phi_i^N}_{L^2(\mescin_N)}
     = \av{\vp{i}, (1 - \projvit){\Phi}_i^N}_{L^2(\mescin_N)}\\
    & \leq \norm{\vp{i}}_{L^2(\mescin_N)} \norm{(1 - \projvit){\Phi}_i^N}_{L^2(\mescin_N)} = \beta^{-1/2}\norm{(1 - \projvit){\Phi}_i^N}_{L^2(\mescin_N)} \\
    & \leq \beta^{-1}\norm{\nabla_p\Phi_i^N}_{L^2(\mescin_N)},
\end{align*}
which gives $\norm{\nabla_p\Phi_i^N}_{L^2(\mescin_N)} \leq \gamma^{-1}$.
The proof of~\eqref{eq:estimate_grad_1particle} is then a direct consequence of the exchangeability of the particles.
    More precisely,
    for any~$k,\ell\in \{1,\dots,N\}$ and~$i\notin \{(k-1)d,\dots,kd\} \cup \{(\ell-1)d,\dots,\ell d\}$ we have $\nabla_{\vp{k}}\Phi_i^N = \nabla_{\vp{\ell}}\Phi_i^N$.
    The inequalities then directly follow from the bound on~$\norm{\nabla_p\Phi_i^N}_{L^2(\mescin_N)}$.
\end{proof}
\subsection{Proof of~\texorpdfstring{\cref{th:eps_then_N}}{Theorem~\ref{th:eps_then_N}}}
\label{sec:proof_lim_cov_particle_system}
For the proof, we need to take the limit as $N \rightarrow +\infty$ of the covariance matrix
\begin{align}
    \label{eq:D_N}
    \DeffN  &= \displaystyle\int_{\mc{E}^N}\Phi^N \otimes \vP{N} \, \d\mu_N = \frac{\gamma}{\beta}\int_{\mc{E}^N} (\nabla_p\Phi^N)^\top\nabla_p\Phi^N \, \d\mu_N.
\end{align}
To this end, we decompose the proof into three steps.
\begin{enumerate}
    \item We first give two important results from~\cite[Proof of Theorem 1.16]{delgadino2021diffusive}.
    \item We then pass to the limit, in a weak sense, in the Poisson equation
    \begin{align}
        \label{eq:poisson_eq_Nbis}
        -\mathcal{L}_N\Phi_1^N = p_1,\qquad  \int_{\mc{E}^N}\Phi_1^N \,\d \mescin_N =0,
    \end{align}
    which, by exchangeability, holds for any~$i \in \{1,\dots,Nd\}$.
    \item Finally, we conclude on the convergence of each matrix element of~$\DeffN$.
\end{enumerate}

\paragraph{$\bullet$ Step 1: Decomposition of~$\mescin_N$}
The proof detailed in~\cite[Section 3]{delgadino2021diffusive} use a reformulation of~$\mespos_N$ that is valid in the limit as $N \to \infty$.
Under the assumption that~\eqref{eq:mf_free_energy} has a unique global minimizer, it is proved in~\cite[Proof of Theorem 1.16]{delgadino2021diffusive} that there exists~$C> 0$, independent of~$N$, such that
\begin{align}
    \label{eq:bound_inf_ratio_ZN}
    \|(\mespos_N)_1\|_{C^1(\mc{E})} \leq C,\qquad
    C^{-1} \leq \mespos_N/\mespos_{N-1} \leq C.
\end{align}
Furthermore,
~\cite[Proof of Theorem 1.16]{delgadino2021diffusive} gives that for any~$\vq{1} \in \torus^d$,
\begin{align}
    \label{eq:result_marg_delgadino2021diffusive}
    \lim\limits_{N \rightarrow +\infty} \int_{(\torus^d)^{N-1}}\left(\frac{\mespos_N}{\mespos_{N-1}} - (\mespos_N)_1\right)^2 \mespos_{N-1} = 0.
\end{align}


\paragraph{$\bullet$ Step 2: Limit for the Poisson equation}
This step is dedicated to obtaining a weak limit in the Sobolev space~$H^1(\mescin_{\min})$ for the solution of the Poisson equation~\eqref{eq:poisson_eq_Nbis}, which is necessary to pass to the limit for each matrix element of~$\DeffN$ in~\eqref{eq:D_N}. For~$j\in\{1,\dots,dN\}$, the matrix element~$(1,j)$ of~$\DeffN$ is given by
\begin{align*}
    (\DeffN)_{1,j} = \int_{\mc{E}^N} \Phi_1^N p_j \, \d\mescin_N = \frac{\gamma}{\beta}\int_{\mc{E}^N} \nabla_p \Phi_1^N \cdot \nabla_p \Phi_j^N \, \d\mescin_N.
\end{align*}
As we will show in Step 3, the case~$j \in \{1,\dots,d\}$ is the only one that contributes to the limit of~$\DeffN$ as~$N \rightarrow +\infty$, while the other matrix elements vanish in the limit. Therefore, we focus on the case where we test~\eqref{eq:poisson_eq_Nbis} against a test function that depends only on the first~$d$ coordinates.
To this end, introduce~$\mc{C}_{\rm b}^\infty\!\left(\torus^d \times \real^d\right)$, the space of~$\mc{C}^\infty$ bounded functions with all derivatives bounded on~$\torus^d \times \real^d$. We take a test function~$\varphi\in\mc{C}_{\rm b}^\infty\!\left(\torus^d \times \real^d\right)$.
More precisely, in view of~\eqref{eq:poisson_eq_Nbis},
\begin{align}
    \label{eq:Poisson_integral}
    \int_{\mc{E}^{N}}(-\mathcal{L}_N \Phi_1^N) (\vQ{N},\vP{N}) \varphi(\vq{1},\vp{1})\, \mescin_N(\d\vQ{N}\,\d\vP{N}) = \int_{\mc{E}^{N}} p_1 \varphi(\vq{1},\vp{1}) \, \mescin_N(\d\vQ{N}\,\d\vP{N}).
\end{align}
For the left-hand term,
we use~\eqref{eq:generator_mf}, \eqref{eq:LhamN} and~\eqref{eq:LSymN} to write
\begin{align}
    \notag
    &\int_{\mc{E}^{N}}-\mathcal{L}_N \Phi_1^N  \varphi \, \mescin_N \\
    \notag
    & = \frac1\beta\!\left(\gamma\int_{\mc{E}^{N}}\nabla_p^*\nabla_p\Phi_1^N \varphi \, \d\mescin_N
    - \int_{\mc{E}^{N}}\nabla_p^*\nabla_q \Phi_1^N  \varphi \, \d\mescin_N
    + \int_{\mc{E}^{N}}\nabla_q^*\nabla_p\Phi_1^N  \varphi \, \d\mescin_N \right) \\
    \notag
    & = \frac1\beta\!\left(\gamma\int_{\mc{E}^{N}}\nabla_p\Phi_1^N \cdot \nabla_p\varphi \, \d\mescin_N
    - \int_{\mc{E}^{N}}\nabla_q \Phi_1^N \cdot \nabla_p\varphi \, \d\mescin_N
    + \int_{\mc{E}^{N}}\nabla_p\Phi_1^N \cdot \nabla_q\varphi \, \d\mescin_N \right) \\
    \label{eq:rewriting_fc_test}
    &= \frac1\beta\!
    \left(\!\gamma\int_{\mc{E}^{N}}\!\nabla_{\vp{1}}\Phi_1^N \cdot\nabla_{\vp{1}}\varphi \, \d\mescin_N
    - \int_{\mc{E}^{N}}\!\nabla_{\vq{1}} \Phi_1^N \cdot \nabla_{\vp{1}}\varphi \,\d \mescin_N  + \int_{\mc{E}^{N}}\!\nabla_{\vp{1}}\Phi_1^N \cdot\nabla_{\vq{1}}\varphi \, \d\mescin_N\!\right)\!,
\end{align}
where we used that~$\varphi$ is nonconstant only on~$(\vq{1},\vp{1})$.
As in~\cite{delgadino2021diffusive},
we next see~$\mescin_N$ as a perturbation of~$(\mescin_N)_1\mescin_{N-1}$.
We have
\begin{align*}
    &\left|\int_{\mc{E}^{N}}(\nabla_{\vp{1}}\Phi_1^N \cdot \nabla_{\vp{1}}\varphi )\,
     \bigl(\mescin_N - (\mescin_N)_1\mescin_{N-1}\bigr)\,\d \vQ{N}\,\d \vP{N}\right| \\
    &\leq \norm{\varphi}_{\mc{C}^1(\mc{E})}
    \int_{\mc{E}^{N}}\!\left(\sum_{i=1}^{d}|\partial_{p_i}\Phi_1^N|\right)
    \left|\frac{\mescin_N}{\mescin_{N-1}} - (\mescin_N)_1\right|\mescin_{N-1}\,\d \vQ{N}\,\d \vP{N} \\
    &\leq \norm{\varphi}_{\mc{C}^1(\mc{E})}
    \int_{\mc{E}^{N}}\!\left(\sum_{i=1}^{d}|\partial_{p_i}\Phi_1^N|\right)
    \left|\frac{\mespos_N}{\mespos_{N-1}} - (\mespos_N)_1\right|\mespos_{N-1}\mesvit_N\,\d \vQ{N}\,\d \vP{N},
 \end{align*}
 with~$\norm{\varphi}_{\mc{C}^1(\mc{E})} =\max\{\max_{1\leq i\leq d}\norm{\partial_{p_i}\varphi}_{\mc{C}^0},\norm{\varphi}_{\mc{C}^0}\}$.
  By Fubini's theorem and a Cauchy--Schwarz inequality, we have
 \begin{align*}
    &\left|\int_{\mc{E}^{N}}(\nabla_{\vp{1}}\Phi_1^N \cdot \nabla_{\vp{1}}\varphi )\,
     ( \mescin_N - (\mescin_N)_1\mescin_{N-1})\,\d \vQ{N}\,\d \vP{N}\right| \\
     &\leq 2^{\frac{d-1}{2}}\norm{\varphi}_{\mc{C}^1(\mc{E})}
     \left(\int_{\mc{E}^N}|\nabla_p \Phi_1^N|^2 \,\mespos_{N-1}\mesvit_N\,\d \vQ{N}\,\d \vP{N}\right)^{1/2} \\
     & \quad \times \left(\int_{\mc{E}^N}\left|\frac{\mespos_N}{\mespos_{N-1}} - (\mespos_N)_1\right|^2\mespos_{N-1}\mesvit_N\,\d \vQ{N}\,\d \vP{N}\right)^{1/2}\!.
 \end{align*}
  Using~\eqref{eq:bound_inf_ratio_ZN}, we obtain
 \begin{align*}
    &\left|\int_{\mc{E}^{N}}(\nabla_{\vp{1}}\Phi_1^N \cdot \nabla_{\vp{1}}\varphi )\,
      (\mescin_N - (\mescin_N)_1\mescin_{N-1})\,\d \vQ{N}\,\d \vP{N}\right| \\
     &\leq 2^{\frac{d-1}{2}}\norm{\varphi}_{\mc{C}^1(\mc{E})}C^{\frac12} \norm{\nabla_p \Phi_i^N}_{L^2(\mescin_N)} \!\left(\!\int_{\torus^d}\left(\!\int_{\torus^{d(N-1)}}\left|\frac{\mespos_N}{\mespos_{N-1}} - (\mespos_N)_1\right|^2\mespos_{N-1}\right)\d \vq{1}\!\right)^{1/2}\!.
 \end{align*}
  Note that~$\norm{\nabla_p \Phi_i^N}_{L^2(\mescin_N)}$ is uniformly bounded in~$N$ by~\cref{cor:bound_unif_grad_poisson} and, by~\eqref{eq:bound_inf_ratio_ZN},
 \[
 \forall N\geq 2, \qquad \sup_{\vq{1}\in \torus^d} \left|\int_{\mc{E}^{N-1}}\left|\frac{\mespos_N}{\mespos_{N-1}} - (\mespos_N)_1\right|^2\mespos_{N-1}\right| \leq M,
 \]
 with~$M>0$ independent of~$N$.
 From~\eqref{eq:result_marg_delgadino2021diffusive} and the dominated convergence theorem, we therefore deduce
\begin{align}
    \label{eq:first_marg_poisson}
    \lim\limits_{N \rightarrow +\infty}
    \left|\int_{\mc{E}^{N}}(\nabla_{p_1}\Phi^N \cdot \nabla_{p_1}\varphi)
     (\mescin_N - (\mescin_N)_1\mescin_{N-1})\,\d \vQ{N}\,\d \vP{N}\right|
    = 0.
\end{align}
By performing similar calculations,
we obtain the same result for the other two terms in~\eqref{eq:rewriting_fc_test}, so that
\begin{align*}
    \int_{\mc{E}^{N}}(-\mathcal{L}_N \Phi_1^N) \varphi \,\mescin_{N}\,\d \vQ{N}\,\d \vP{N}
    &= \frac{\gamma}{\beta}\int_{\mc{E}^{N}}\nabla_{\vp{1}}\Phi_1^N \cdot \nabla_{\vp{1}}\varphi\, (\mescin_N)_1\mescin_{N-1}\,\d \vQ{N}\,\d \vP{N} \\
    &\quad - \frac{1}{\beta}\int_{\mc{E}^{N}}\nabla_{\vq{1}}\Phi_1^N \cdot \nabla_{\vp{1}}\varphi\, (\mescin_N)_1\mescin_{N-1}\,\d \vQ{N}\,\d \vP{N} \\
    &\quad + \frac{1}{\beta}\int_{\mc{E}^{N}}\nabla_{\vp{1}}\Phi_1^N \cdot \nabla_{\vq{1}}\varphi\, (\mescin_N)_1\mescin_{N-1}\,\d \vQ{N}\,\d \vP{N}
    + \rm{o}(1),
\end{align*}
the remainder term~$\rm{o}(1)$ going to 0 as~$N \to \infty$. Finally, by Fubini's theorem,
\begin{align*}
    &\int_{\mc{E}^{N}}(-\mathcal{L}_N \Phi_1^N) \varphi \,\mescin_{N}\,\d \vQ{N}\,\d \vP{N}
    = \int_{\mc{E}} p_1 \varphi(\vq{1},\vp{1})\, (\mescin_N)_1(\vq{1},\vp{1})\,\d q^1\,\d p^1 \\
    &= \frac{\gamma}{\beta}\int_{\mc{E}}\nabla_{\vp{1}}\Psi^N(\vq{1},\vp{1})\cdot \nabla_{\vp{1}}\varphi(\vq{1},\vp{1}) (\mescin_N)_1(\vq{1},\vp{1})\,\d q^1\,\d p^1 \\
    &\quad - \frac{1}{\beta}\int_{\mc{E}}\nabla_{\vq{1}}\Psi^N(\vq{1},\vp{1})\cdot \nabla_{\vp{1}}\varphi(\vq{1},\vp{1})(\mescin_N)_1(\vq{1},\vp{1})\,\d q^1\,\d p^1 \\
    &\quad + \frac{1}{\beta}\int_{\mc{E}}\nabla_{\vp{1}}\Psi^N(\vq{1},\vp{1})\cdot \nabla_{\vq{1}}\varphi(\vq{1},\vp{1})(\mescin_N)_1(\vq{1},\vp{1})\,\d q^1\,\d p^1
    + \rm{o}(1),
\end{align*}
where
\[
    \Psi^N(\vq{1},\vp{1})
    := \left(\int_{\mc{E}^{N-1}}\Phi_1^N \, \mescin_{N-1}\,\d\vq{2} \dots \d \vq{N}\,\d\vp{2} \dots \d \vp{N}\right)
    (\vq{1},\vp{1}).
\]
In the following, we sometimes omit~$(\vq{1},\vp{1})$ in the integrals. We get
\begin{align}
\notag
&\left|\int_{\mc{E}}\nabla_{\vp{1}}\Psi^N\cdot \nabla_{\vp{1}}\varphi \left((\mescin_N)_1 - \mescin_{\min}\right)\right|\\
\label{eq:jsp}
&\qquad \leq \left(\int_{\mc{E}}\left|\nabla_{\vp{1}}\Psi^N\cdot \nabla_{\vp{1}}\varphi \right|\, (\mesvit_N)_1\d q^1\,\d p^1\right)\norm{(\mespos_N)_1 - \mespos_{\min}}_{\mc{C}^0(\torus^d)}.
\end{align}
The left term in the product of the right-hand side of~\eqref{eq:jsp} is bounded uniformly in~$N$ in view of~\eqref{eq:bound_inf_ratio_ZN}. Moreover, by~\cite[Section 3]{delgadino2021diffusive}, the probability measure~$(\mespos_N)_1$ converges in~$\mc{C}^0(\torus^d)$ to~$\mespos_{\min}$ on~$\torus^d$.
Hence,
\[
\left|\int_{\mc{E}}\nabla_{\vp{1}}\Psi^N\cdot \nabla_{\vp{1}}\varphi \left((\mescin_N)_1 - \mescin_{\min}\right)\right| \xrightarrow[N\to\infty]{} 0.
\]
This result allows us to write
\begin{align*}
    &\int_{\mc{E}^{N}}(-\mathcal{L}_N \Phi_1^N \varphi )\,\mescin_{N}\,\d \vQ{N}\,\d \vP{N} \\
    &= \frac{\gamma}{\beta}\int_{\mc{E}}\nabla_{\vp{1}}\Psi^N(\vq{1},\vp{1})\cdot \nabla_{\vp{1}}\varphi(\vq{1},\vp{1}) \mescin_{\min}(\vq{1},\vp{1})\,\d q^1\,\d p^1 \\
    &\quad - \frac{1}{\beta}\int_{\mc{E}}\nabla_{\vq{1}}\Psi^N(\vq{1},\vp{1})\cdot \nabla_{\vp{1}}\varphi(\vq{1},\vp{1})\mescin_{\min}(\vq{1},\vp{1})\,\d q^1\,\d p^1 \\
    &\quad + \frac{1}{\beta}\int_{\mc{E}}\nabla_{\vp{1}}\Psi^N(\vq{1},\vp{1})\cdot \nabla_{\vq{1}}\varphi(\vq{1},\vp{1}) \mescin_{\min}(\vq{1},\vp{1})\,\d q^1\,\d p^1
    + \rm{o}(1).
\end{align*}
The above equality holds when adding a constant to~$\Psi^N$. Thus, we can assume that, without loss of generality,
\[
\forall N \geq 1,\qquad \int_{\mc{E}}\Psi^N\,\mescin_{\min} = 0.
\]
By~\cref{cor:bound_unif_grad_poisson} and Poincaré's inequality on~$\mc{E} = \torus^d \times \real^d$ for the probability measure~$\mescin_{\min}$, there exists~$K >0$, such that
\[
\forall N \geq 1,\qquad \norm{\Psi^N}_{H^1(\mescin_{\min})} \leq K.
\]
By the Banach--Alaoglu theorem, the sequence~$(\Psi^N)_{N\geq 1}$ admits at least one weak accumulation point in~$H^1(\mescin_{\min})$. However, any weak accumulation point~$\Psi^* \in H^1(\mescin_{\min})$ of $\left(\Psi^N\right)_{N\geq 1}$
satisfies the following equation for any~$\varphi \in \mc{C}_{\rm b}^\infty(\mc{E})$:
\begin{align*}
    &\frac{\gamma}{\beta}\int_{\mc{E}}\nabla_{p^1}\Psi^*((\vq{1},\vp{1})) \cdot \nabla_{p^1}\varphi(\vq{1},\vp{1})  \, \mescin_{\min}(\d\vq{1}\,\d\vp{1}) \\
    &-\frac{1}{\beta}\int_{\mc{E}}\nabla_{q^1}\Psi^*(\vq{1},\vp{1}) \cdot \nabla_{p^1}\varphi(\vq{1},\vp{1}) \, \mescin_{\min}(\d\vq{1}\,\d\vp{1})  \\
    & +\frac{1}{\beta}\int_{\mc{E}}\nabla_{p^1}\Psi^*(\vq{1},\vp{1})\cdot \nabla_{q^1}\varphi(\vq{1},\vp{1}) \, \mescin_{\min}(\d\vq{1}\,\d\vp{1}) \\
    &= \int_{\mc{E}} p_1 \varphi(\vq{1},\vp{1}) \, \mescin_{\min}(\d\vq{1}\,\d\vp{1}) ,
\end{align*}
with~$\int_{\mc{E}} \Psi^*\,\mescin_{\min} = 0$.
Consequently,~$\Psi^*$ is a weak solution to the Poisson equation:
    \begin{align*}
        \int_{\mc{E}}(-\genlin \Psi^* \cdot \varphi )\,\mescin_{\min} = \int_{\mc{E}} (p_1 \cdot \varphi )\, \mescin_{\min}.
    \end{align*}
By density of $C^{\infty}_{\rm b}$ in $H^1(\mescin_{\min})$ and by uniqueness of the solution,~$\Psi^* = \SolPois$ in~$H^1(\mescin_{\min})$.
Finally, by uniqueness of the accumulation point, the full sequence~$(\Psi^N)_{N\geq 1}$ converges weakly in~$H^1(\mescin_{\min})$ to~$\SolPois$.
\paragraph{$\bullet$ Step 3: Limit of the covariance matrix}
Finally, we are able to pass to the limit for the coefficient of the matrix
\begin{align*}
    \DeffN &= \int_{\mc{E}^N}\Phi^N \otimes \vP{N} \mu_N(\d \vQ{N} \, \d \vP{N}) = \frac{\gamma}{\beta}\int_{\mc{E}^N} (\nabla_p\Phi^N)^\top\nabla_p\Phi^N \,\mu_N(\d \vQ{N} \, \d \vP{N}).
\end{align*}
There are two cases to consider, depending on whether the indices~$i,j$ of the matrix element are such that there exists~$k\in\{1,\dots,N\}$ for which~$i,j\in\{(k-1)d +1 ,\dots,kd\}$ (\textit{i.e.} the coordinates correspond to the same particle) or not.
\paragraph{First case: same particle}
We focus on the case where both indices are in $\{1,\dots,d\}$, which, by exchangeability provides the limit for any~$k\in\{1,\dots,N\}$. Let~$1\leq j\leq d$.
We have
\[
(\DeffN)_{1,j} = \int_{\mc{E}^N}\Phi^N_1 p_j\,\mescin_N = \int_{\mc{E}^N}\widetilde\Phi^N_1 p_j\,\mescin_N,
\]
with~$\widetilde\Phi^N_1 = \Phi^N_1 - \projvit\Phi_1^N$.
In the following, we take~$i = 1$ without loss of generality. We have
\[
(\DeffN)_{1,j} = \int_{\mc{E}^N}\widetilde\Phi^N_1 p_j\,\mescin_{N-1}(\mescin_{N})_1
+ \int_{\mc{E}^N}\widetilde\Phi^N_1 p_j\left(\frac{\mescin_N}{\mescin_{N-1}}-(\mescin_{N})_1\right)\mescin_{N-1}.
\]
By using~\cref{prop:weak_limit_marginals} together with the results of Step 2, \textit{i.e.} that $\displaystyle\int_{\mc{E}^{N-1}} \Phi^N_1 \mescin_{N-1}$ weakly converges in~$H^1(\mescin_{\min})$  to~$\Phi^{\min}_1$,
we conclude that
\begin{align*}
    \int_{\mc{E}^N} \widetilde \Phi^N_1 p_j\,\mescin_{N-1}(\mescin_{N})_1 
    &= \int_{\mc{E}^N} \Phi^N_1 p_j\,\mescin_{N-1}(\mescin_{N})_1 
    \\
    &= \int_{\mc{E}}\left(\int_{\mc{E}^{N-1}} \Phi^N_1 \,\mescin_{N-1}\right) p_j\,(\mescin_{N})_1
    \xrightarrow[N\to \infty]{}\int_{\mc{E}}\Phi^{\min}_1 p_j\,\mescin_{{\min}} \, .
\end{align*}
We next show that
\begin{align}
    \label{eq:remainder_covariance}
    \int_{\mc{E}^N}\widetilde\Phi^N_1 p_j\left(\frac{\mescin_N}{\mescin_{N-1}}-(\mescin_{N})_1\right)\mescin_{N-1} \xrightarrow[N\to \infty]{} 0.
\end{align}
Since
\[
\left(\frac{\mescin_N}{\mescin_{N-1}}-(\mescin_{N})_1\right)\mescin_{N-1}
= \left(\frac{\mespos_N}{\mespos_{N-1}} - (\mespos_N)_1\right)\mespos_{N-1}\mesvit_N,
\]
using the Cauchy--Schwarz inequality we deduce
\begin{align}
    \notag
    &\left|\int_{\mc{E}^N}\widetilde\Phi^N_1 p_j\left(\frac{\mescin_N}{\mescin_{N-1}}-(\mescin_{N})_1\right)\mescin_{N-1}\right| \\
    \label{eq:bound_remainder_covariance}
    &\qquad\leq \left(\int_{\mc{E}^N}|\widetilde\Phi^N_1 p_j|^2\,\mespos_{N-1}\mesvit_N \right)\left(\int_{\mc{E}^N}\left|\frac{\mespos_N}{\mespos_{N-1}} - (\mespos_N)_1\right|^2\mespos_{N-1}\mesvit_N\right)\!.
\end{align}
By~\eqref{eq:result_marg_delgadino2021diffusive} and a dominated convergence argument,
\[
\int_{\mc{E}^N}\left|\frac{\mespos_N}{\mespos_{N-1}} - (\mespos_N)_1\right|^2\mespos_{N-1}\mesvit_N
\xrightarrow[N\to \infty]{} 0.
\]
The convergence~\eqref{eq:remainder_covariance} therefore holds if we can bound uniformly in~$N$ the 
first factor on the right-hand side of~\eqref{eq:bound_remainder_covariance}.
In view of the second inequality in~\eqref{eq:bound_inf_ratio_ZN},
\[
\int_{\mc{E}^N}|\widetilde\Phi^N_1 p_j|^2\,\mespos_{N-1}\mesvit_N \leq C \|p_j\widetilde\Phi^N_1\|_{L^2(\mescin_N)}^2.
\]
Since
\[
p_j = \frac{1}{\beta}(\partial^*_{p_j} + \partial_{p_j}),
\]
it holds
\[
\|p_j\widetilde\Phi^N_1\|_{L^2(\mescin_N)}^2
= \frac{1}{\beta^2}\norm{(\partial^*_{p_j} + \partial_{p_j})\widetilde\Phi_1^N}^2
\le \frac{2}{\beta^2}\left(\norm{\partial_{p_j}\widetilde\Phi_1^N}^2+\norm{\partial^*_{p_j}\widetilde\Phi_1^N}^2\right)\!.
\]
By~\cref{cor:bound_unif_grad_poisson}, we know that~$\norm{\partial_{p_j}\widetilde\Phi_1^N}^2$ is uniformly bounded in~$N$. Moreover,
\[
[\partial_{p_j},\partial_{p_j}^*]=\partial_{p_j}\partial^*_{p_j} - \partial^*_{p_j}\partial_{p_j} = \beta,
\]
and so
\begin{align*}
\norm{\partial^*_{p_j}\widetilde\Phi_1^N}^2 &= \slangle{\partial^*_{p_j}\widetilde\Phi_1^N,\partial^*_{p_j}\widetilde\Phi_1^N}_{L^2(\mescin_N)}
= \slangle{\partial_{p_j}\partial^*_{p_j}\widetilde\Phi_1^N,\widetilde\Phi_1^N}_{L^2(\mescin_N)} \\
& = \slangle{\partial_{p_j}^*\partial_{p_j}\widetilde\Phi_1^N,\widetilde\Phi_1^N}_{L^2(\mescin_N)} + \beta\slangle{\widetilde\Phi_1^N,\widetilde\Phi_1^N}_{L^2(\mescin_N)} = \norm{\partial_{p_j}\widetilde\Phi_1^N}^2 + \beta \norm{\widetilde\Phi_1^N}^2 \\
& \leq 2 \norm{\nabla_p\widetilde\Phi_1^N}^2,
\end{align*}
where the last line is obtained using Poincaré's inequality for Gaussian probability measures, since we considered~$\projvit\widetilde\Phi_1^N = 0$.
Gathering all the terms, we finally have
\[
\int_{\mc{E}^N}|\widetilde\Phi^N_1 p_j|^2\,\mespos_{N-1}\mesvit_N \leq \frac{6C}{\beta^2}\norm{\nabla_p\widetilde\Phi_1^N}^2,
\]
which is uniformly bounded in~$N$ by~\cref{cor:bound_unif_grad_poisson}. This allows us to conclude that~\eqref{eq:remainder_covariance} holds
and so
\[
\int_{\mc{E}^N}\Phi^N_1 p_j\,\mescin_{N} \xrightarrow[N\to \infty]{}\int_{\mc{E}}\Phi^{\min}_1 p_j\,\mescin_{{\min}}.
\]
\paragraph{Second case}
For the coefficients that are not on a diagonal block of size~$d \times d$, let~$1\leq i,j\leq Nd$ be such that~$(k_1-1)d \leq i \leq k_1d$ and~$(k_2-1)d \leq j \leq k_2d$ with~$k_1 \neq k_2$.
Considering the first equality in~\eqref{eq:D_N},
an integration by parts and a Cauchy--Schwarz inequality give
\begin{align*}
    \left|(\DeffN)_{i,j}\right|
    &= 2\left|\int_{\mc{E}^N}\Phi_i^N  p_j \,\mescin_N\right|
    =  \frac{2}{\beta}\left|\int_{\mc{E}^N}\partial_{p_j}\Phi_i^N \,\mescin_N\right|\\
    & \leq \frac{2}{\beta}\|\partial_{p_j}\Phi_i^N\|_{L^2(\mescin_N)} \leq \frac{2}{\beta}\|\nabla_{p_{k_2}}\Phi_i^N\|_{L^2(\mescin_N)}.
\end{align*}
Therefore, by~\eqref{eq:estimate_grad_1particle},
\begin{align*}
    \left|(\DeffN)_{i,j}\right| \leq \frac{2}{\gamma\beta\sqrt{N-1}} \xrightarrow[N\rightarrow\infty]{} 0.
\end{align*}
This leads to~\eqref{eq:lim_mat_con_N} and concludes the proof.

%
%
\section{Stationary states for the~$O(2)$ model in a magnetic field}
\label{sec:kuramoto}
This section is dedicated to the proof of~\cref{th:Kuramoto}. We give a detailed analysis on the steady states of the~$O(2)$ model in a magnetic field and thus extend the results of~\cite[Lemma 1.26]{delgadino2021diffusive}.
\begin{proof}[Proof of~{\cref{th:Kuramoto}}]
    From~\cite[Lemma 1.26]{delgadino2021diffusive} we know that the positions marginal stationary measures of~\eqref{eq:fok-pl-vlas} are of the form
    \begin{align}
        \label{eq:integral_eq}
        \nu(q)=Z^{-1}\e^{a_\beta \cos(2\pi q)},\qquad Z = \int_{\torus}\e^{a_\beta \cos(2\pi q)}\!\d q,
    \end{align}
    where~$a_\beta$ are the zeros of the function
    \begin{align}
        \label{eq:def_function_F}
        F(a,\beta) := \frac{I_1(a)}{I_0(a)} + \eta a - \frac{a}{\beta},
    \end{align}
    with
    \[
    \forall n \in \mathbb{N},\qquad I_n(x) = \int_{\torus}\cos(2\pi n u)\e^{x\cos(2\pi u)}\,\d u,
    \]
    the modified Bessel functions of the first kind.
    The self-consistency equation~$F(a,\beta)=0$ is obtained by testing~\eqref{eq:integral_eq} against~$\cos(2\pi q)$,
    then using symmetry arguments and a change of variables.

     Furthermore, it is shown in~\cite[Lemma 1.26]{delgadino2021diffusive} that, for any~$\beta > 0$, there exists a unique positive~$a$ such that~$F(a,\beta) = 0$.
     We provide below a more detailed understanding of the zeros of~$F(a,\beta)$ and show that for~$\beta$ sufficiently large, there exist exactly two negative zeros of~$F(a,\beta)$.
     The proof is divided into two steps:



        First, we fully characterize the zeros of~$F(a,\beta)$ with respect to~$\beta$.
        Then, we show that the other stationary measures of~\eqref{eq:fok-pl-vlas}, \textit{i.e.} for~$a<0$, are not local minimizers of the free energy~\eqref{eq:mf_free_energy}.

    \begin{figure}[ht]
\centering
    \begin{subfigure}{0.49\textwidth}
        \centering
        \includegraphics[width=\linewidth]{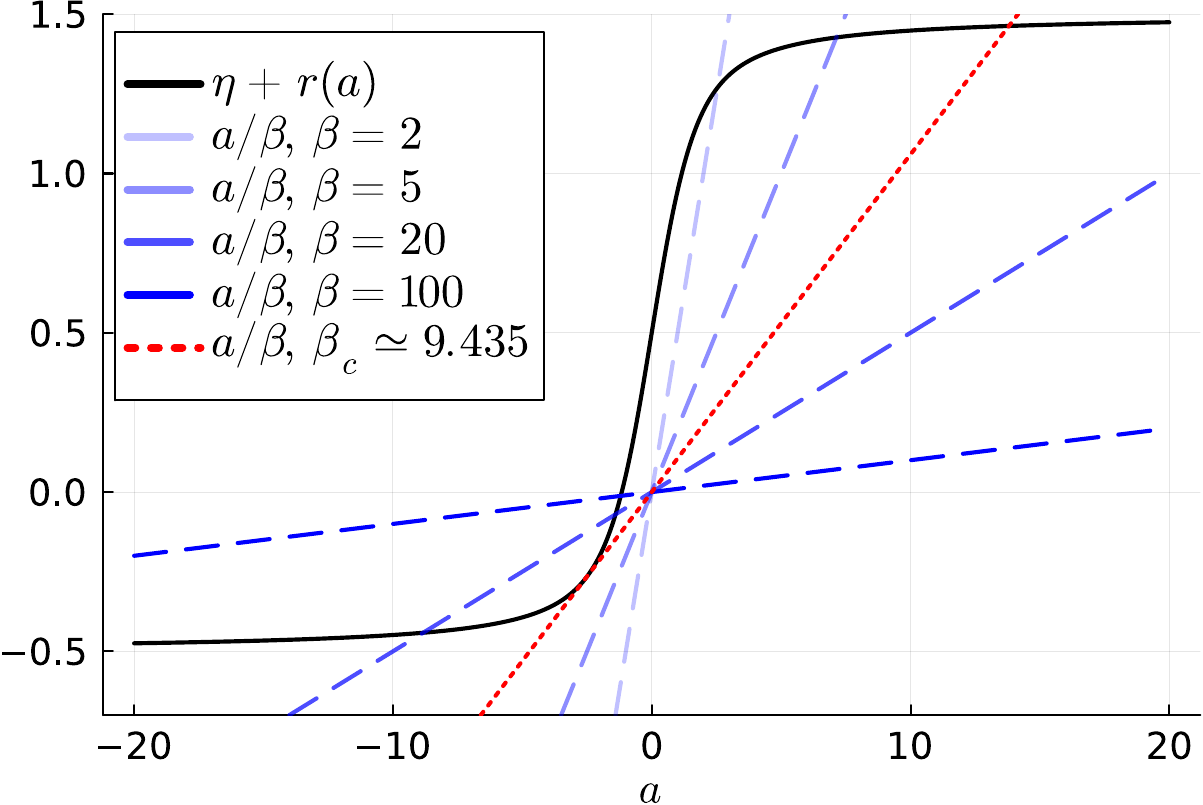}
    \end{subfigure}
    \begin{subfigure}{0.49\textwidth}
        \centering
        \includegraphics[width=\linewidth]{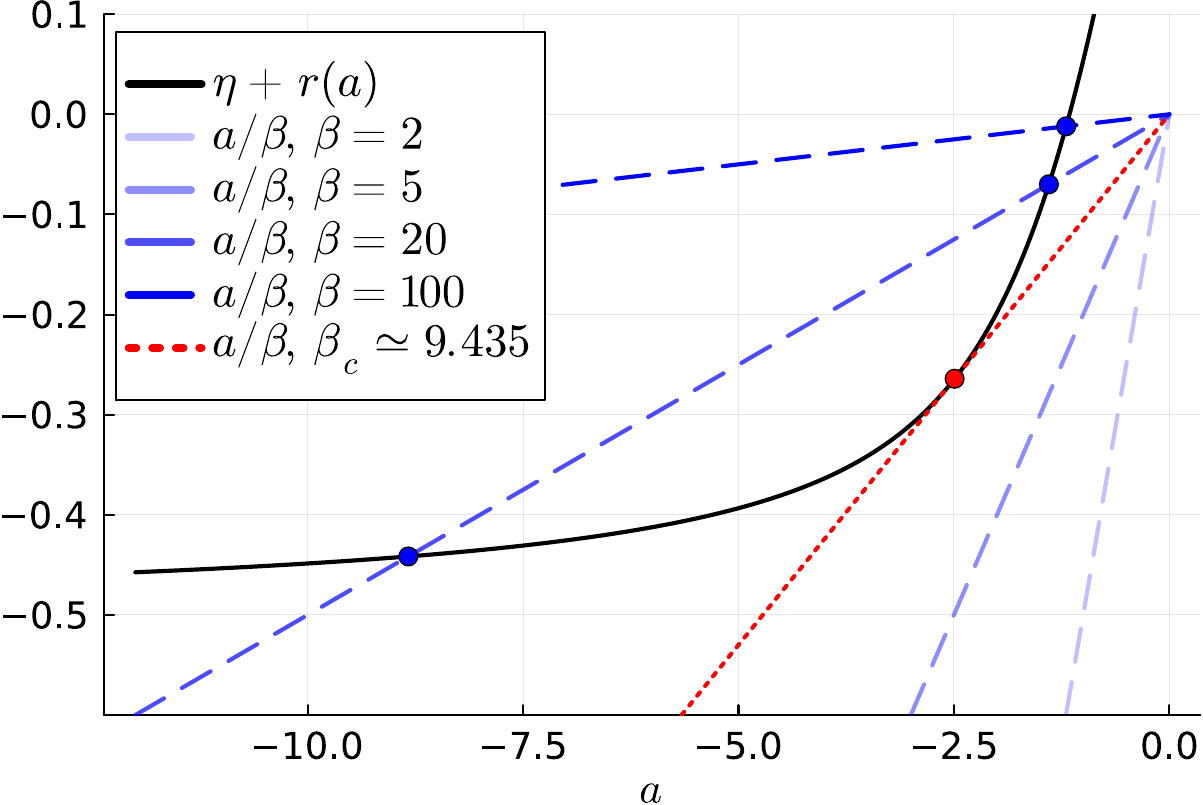}
    \end{subfigure}
\caption{Function~$a \mapsto \eta + r(a)$ and~$a \mapsto \frac{a}{\beta}$ for different values of~$\beta$. The intersections of the two curves correspond to the zeros of~$F(a,\beta)$ defined in~\eqref{eq:def_function_F}. In the right panel we zoom on the interval~$[-12,0]$ to show the existence of two negative zeros for~$\beta$ sufficiently large together with the critical case.}
\label{fig:intersection}
\end{figure}
\paragraph{$\bullet$ Step 1}
The function defined in \eqref{eq:def_function_F} can be written as
\[
F(a,\beta)=r(a)+\eta-\frac{a}{\beta},
\]
where $r(a)=I_1(a)/I_0(a)$. The function $r$ is smooth, odd and strictly increasing, convex on $(-\infty,0]$ and strictly concave on $[0,\infty)$. We also have that~$r'$ is increasing on~$(-\infty,0]$,~$r'(0)=\tfrac12$ and $\lim_{a\to\pm\infty}r(a)=\pm1$ (see, e.g., \cite{amos1974computation,simpson1984some}). The main idea is that the equation $F(a,\beta)=0$ is equivalent to $\eta+r(a)=a/\beta$, so steady states correspond to intersections of $\eta+r(x)$ with the line $a/\beta$. See~\cref{fig:intersection} for a graphical illustration of the different cases. We split the analysis of the zeros of $F(\cdot,\beta)$ according to the sign of $a$. 

\paragraph{Zero on $[0, +\infty)$}

Since $F(0,\beta)=\eta>0$ and $\lim_{a\to+\infty}F(a,\beta)=-\infty$,
there exists at least one zero in the interval~$(0, +\infty)$.
The strict concavity on~$[0,+\infty)$ limits the number of zeros to at most one:
if~$0<z_1<z_2$ were two zeros then strict concavity would give,
denoting by~$\ell$ the affine function with value~$\eta$ at~$a=0$ and~0 at $a=z_2$,
that $0=F(z_1) > \ell(z_1) > 0$, which is a contradiction.

\paragraph{Zeros on $(-\infty, 0]$}

For any~$\beta > 0$, we have that $\lim_{a \to - \infty} F(a, \beta) = +\infty$ and $F(0, \beta) > 0$.
By strict convexity of~$F(\,\cdot\,, \beta)$ in the interval $(- \infty, 0]$,
it follows that the number of zeros of~$F(\,\cdot\,, \beta)$ in this interval is either 0, 1 or~2,
depending on whether
\[
    M(\beta) \coloneqq \min_{a \in (-\infty, 0]} F(a, \beta)
\]
is positive, zero, or negative.
For~$\beta \in (0, 2]$, 
the minimum is attained at~$a = 0$ and~$M(\beta) = \eta > 0$,
since~$F(\,\cdot\,, \beta)$ is decreasing on~$(-\infty, 0]$.
For~$\beta > 2$, the minimum is attained at a unique point~$a_\beta < 0$~satisfying
\begin{equation}
    \label{eq:implicit}
    \frac{\partial F}{\partial a}(a_\beta, \beta) = r'(a_\beta) - \frac{1}{\beta} = 0.
\end{equation}
The implicit function theorem ensures that~$\beta \mapsto a_\beta$
is smooth on~$(2,+\infty)$,
and so
\[
    M'(\beta)
    = \frac{\partial F}{\partial a}(a_\beta,\beta)\frac{\d a_\beta}{\d \beta}
    + \frac{\partial F}{\partial \beta}(a_\beta,\beta)
    = 0 + \frac{a_\beta}{\beta^2}
    = \frac{a_\beta}{\beta^2} < 0,
\]
since~$a_\beta < 0$.
It follows from~\eqref{eq:implicit} and the properties of~$r$ that~$a_{\beta} \to 0$ as~$\beta \downarrow 2$,
so~$\lim_{\beta \downarrow 2} M(\beta) = \eta$.
Therefore~$M\colon (0, +\infty) \to \real$ is continuous, equal to~$\eta > 0$ on~$(0,2]$,
and strictly decreasing on~$(2,+\infty)$.
Furthermore, since~$\eta <1$, we have
\[
    \lim_{\beta \to +\infty} M(\beta) = -1 + \eta < 0,
\]
since
\(
    \lim_{\beta \to +\infty} F\bigl(-\sqrt{\beta}, \beta\bigr) = -1 + \eta
\)
and~$M(\beta) \geq -1 + \eta$ by definition of~$M$.
Therefore, since~$M$ is continuous and strictly decreasing on~$(2,+\infty)$ with~$\lim_{\beta \to +\infty} M(\beta) = -1 + \eta < 0$ and~$M(2) = \eta > 0$, by the intermediate value theorem there exists a unique zero~$\beta_{\rm{c}}$ of~$M$
such that~$F(\,\cdot\,, \beta)$ has 0, 1 or 2 zeros in~$(-\infty, 0]$ if
$\beta < \beta_{\rm{c}}$, $\beta = \beta_{\rm{c}}$ or $\beta > \beta_{\rm{c}}$, respectively.

\paragraph{$\bullet$ Step 2}
    We conclude the proof by proving that the critical points of the free energy are unstable when~$a < 0$. To this end we only consider the overdamped case since the arguments are similar in the kinetic case due to the product measure structure of the invariant measure. Let~$a \in \real$ and denote by~$\mespos_a$ the probability measure of the form~\eqref{eq:integral_eq} for a given~$a < 0$.
    Step 1 showed that~$\beta_{\rm c} > 2$. In the following, suppose that~$\beta > 2$ and that the free energy admits at least two critical points. The previous steps showed that there is exactly one critical point with~$a > 0$ for $\nu_a$. We now show that the free energy~\eqref{eq:mf_free_energy}, for the~$O(2)$ model in a magnetic field, can not admit a local minimizer of the form~$\mespos_a$ with~$a < 0$. Indeed, let~$a < 0$ and consider the second variation of the free energy at~$\mespos_a$ in the direction of~$h$, with~$h$ a smooth function such that~$\int_{\torus}h=0$:
    \begin{align}
        \label{eq:second_var_free_energy}
        D^2\emf[\mespos_a](h,h)
        &=\frac1\beta\int_{\torus}\frac{|h|^2}{\mespos_a}
        +\int_{\torus}W\star h(q)\,h(q)\,\d q.
    \end{align}
        We choose the direction $h^*(q):=\sin(2\pi q)\mespos_a(q)$.
    Since~$\mespos_a$ is even, the function~$h^*$ is odd and therefore has zero mean. For the first term on the right-hand side of~\eqref{eq:second_var_free_energy}, an integration by parts gives
    \begin{align}
        \label{eq:def_Sa}
        S_a
        &:=\int_{\torus}\sin^2(2\pi q)\mespos_a(q)\,\d q =\frac{1}{a}
        \int_{\torus}\cos(2\pi q)\mespos_a(q)\,\d q
        =\frac{r(a)}{a}>0,
    \end{align}
    where~$r(a)=I_1(a)/I_0(a)$. Therefore,
    \begin{align}
        \label{eq:entropic_term_second_var_free_energy}
        \frac1\beta\int_{\torus}\frac{|h^*|^2}{\mespos_a}
        =\frac{S_a}{\beta}.
    \end{align}
    Moreover, since~$W(q)=-\cos(2\pi q)$, a direct calculation using the parity of~$\mespos_a$ gives
    \[
        W\star h^*(q)=-S_a\sin(2\pi q),
    \]
    and hence
    \begin{align}
        \label{eq:convol_term_second_var_free_energy}
        \int_{\torus}W\star h^*(q)\,h^*(q)\,\d q=-S_a^2.
    \end{align}
    Combining~\eqref{eq:second_var_free_energy}, \eqref{eq:entropic_term_second_var_free_energy} and~\eqref{eq:convol_term_second_var_free_energy}, we obtain
    \begin{align}
        \label{eq:second_var_explicit_direction}
        D^2\emf[\mespos_a](h^*,h^*)
        =S_a\left(\frac1\beta-S_a\right).
    \end{align}
    Since~$\mespos_a$ is a stationary state,~$a$ is a zero of~\eqref{eq:def_function_F} for a given~$\beta>0$.
    Together with~\eqref{eq:def_Sa}, this yields
    \[
        \frac1\beta-S_a = \frac1\beta - \frac{r(a)}{a}
        =\frac{r(a)+\eta}{a}-\frac{r(a)}{a}
        =\frac{\eta}{a}<0.
    \]
    Consequently,
    \[
        D^2\emf[\mespos_a](h^*,h^*)
        =\frac{\eta S_a}{a}<0,
    \]
    where we used that~$\eta>0$, $S_a>0$ and~$a<0$. This proves that the critical points of the free energy with~$a < 0$ are unstable and hence concludes the proof.
\end{proof}
\section{Numerical illustrations}
\label{sec:numerical_simulation}
We numerically illustrate in this section our theoretical results.
In~\cref{subsec:numerical_model}, we describe the interacting particle system that we will study and its mean field limit and stationary states through its free energy, and provide some details about the simulation. In~\cref{subsec:computation_diffusion}, we detail the computation of the diffusion coefficient, showing the non-commutativity of the diffusive/mean field limit.
%
%
\subsection{The Interacting Particle System}
\label{subsec:numerical_model}
We consider the kinetic~$O(2)$ model in a magnetic field on~$\torus \times \real$. The dynamics of the system of~$N$ particles is given by
\begin{equation}
    \label{eq:kinetic_kuramoto}
    \left\{\begin{aligned}
    &\d \vqt{i} =  \vpt{i}\, \d t, \\
    &\d \vpt{i} = -V'(\vqt{i})\,\d t -\frac1N \sum_{j\neq i}^{N}W'(\vqt{i} -\vqt{j}) \,\d t - \gamma \vpt{i} \,\d t + \sqrt{\frac{2 \gamma}{\beta}}\,\d B_t^i,
    \end{aligned}\right.
\end{equation}
where~$V(q) = \eta W(q) = -\eta\cos(2\pi q)$.
The associated mean-field dynamics is given by
\begin{equation}
    \label{eq:kinetic_kuramoto_mf}
    \left\{\begin{aligned}
    \d q_t &=  p_t\, \d t, \\
    \d p_t &= -V'(q_t)\,\d t -W'\star\projp\mescin_t(q_t) \,\d t - \gamma \vpt{i} \,\d t + \sqrt{\frac{2 \gamma}{\beta}}\,\d B_t, \\
    \mescin_t &= {\rm Law}(q_t, p_t) \, .
    \end{aligned}\right.
\end{equation}
In all our numerical simulations, we choose~$\gamma = 1$,~$\beta= 6$ and~$\eta = 0.1$. In this case,~\eqref{eq:fok-pl-vlas} admits three stationary states since the free energy has three critical points. Indeed,
by considering the marginal position in~\eqref{eq:candidate_mes_inv},
we can compute the free energy as a function of~$a$ by evaluating~\eqref{eq:mf_free_energy} for the probability measure
\begin{align}
    \label{eq:candidate_mes_inv}
    \rho_a(q,p) = \mespos_a(q)\mesvit(p) = Z_a^{-1} \e^{a\cos(2\pi q)} \sqrt{\frac{\beta}{2\pi}}\e^{-\beta p^2 /2}.
\end{align}
Indeed,~\cref{th:Kuramoto} proves that any stationary state if of the form~\eqref{eq:candidate_mes_inv}.
For this one-parameter family of probability measures parametrized by~$a$,
the free energy admits three critical points,
as~\cref{fig:free_energy} shows. In particular,~\cref{th:Kuramoto} shows that only the one linked with~$a_{\min}$ is the global minimizer of the free energy, while the two others are unstable critical points.
\begin{figure}[ht]
    \centering
    \includegraphics[width=0.5\linewidth]{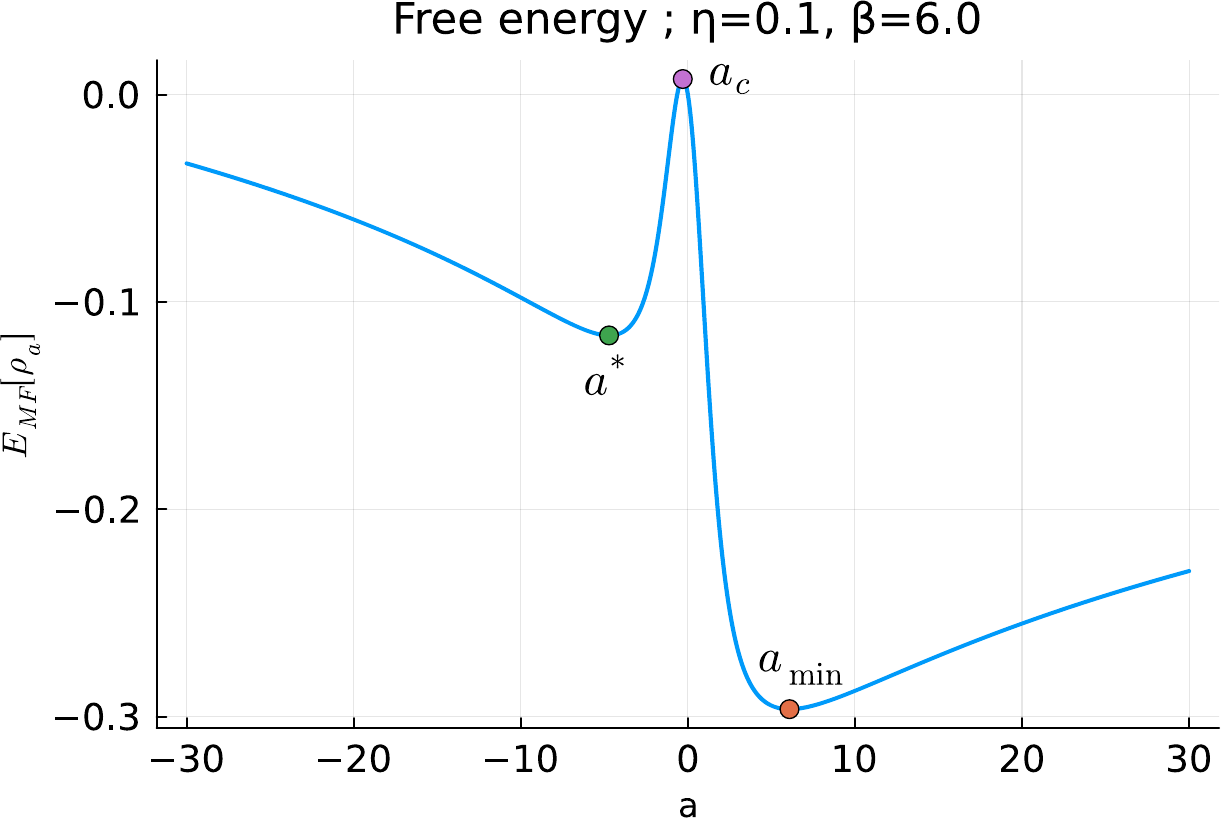}
    \caption{Configurational free energy, along the specific path~\eqref{eq:candidate_mes_inv}, for the $O(2)$ model in a magnetic field. Since~$a^*$ is an unstable critical point, the free energy is decreasing in directions orthogonal to the one of this plot ($a^*$ is a saddle point).}
    \label{fig:free_energy}
\end{figure}
In the sequel, we will numerically compare the behavior of dynamics~\eqref{eq:kinetic_kuramoto} for the values~$a_{\min}$ and~$a^*$ shown in~\cref{fig:free_energy}. We denote by~$\nu_{\min}$ and~$\nu^*$ the position marginals of the stationary states associated with these values~$a_{\min}$ and~$a^*$. Considering~\cref{th:lim_mvt_brow,th:lim_nonlin_mvt_brow}, it is sufficient to integrate the linearized dynamics around~$\nu_{\min}$ and~$\nu^*$ to compute the diffusion coefficent. Hence, the drift in~\eqref{eq:kinetic_kuramoto_mf} is computed using the expression~\eqref{eq:candidate_mes_inv} and the integral equation~\eqref{eq:integral_eq} satisfied by the position marginal. We conclude that
\[
\nu_a(q) = Z_\beta^{-1}\e^{-\beta(V(q) +W\star\nu_a(q))} =  Z_a^{-1} \e^{a\cos(2\pi q)}.
\]
Therefore, in view of~\eqref{eq:inf_potential},
\begin{align}
    \label{eq:explicit_drift}
  V'(q) +W'\star\nu_a(q) = -\frac{2\pi a}{\beta}\sin(2\pi q).
\end{align}
Note that the explicit expression~\eqref{eq:explicit_drift} allows us to integrate the mean-field dynamics without performing any numerical integration or Fourier transform of the convolution term.

For the simulations of the SDE system, we proceed as follows.
For both dynamics~\eqref{eq:kinetic_kuramoto} and~\eqref{eq:kinetic_kuramoto_mf}, we sample the initial conditions according to~$\nu_a$ using a rejection sampling method for marginal distribution in positions. For the velocity, we sample from a Gaussian distribution with mean~$0$ and variance~$\beta^{-1}$.
The dynamics~\eqref{eq:kinetic_kuramoto} and~\eqref{eq:kinetic_kuramoto_mf} (with the drift given by~\eqref{eq:explicit_drift}) are integrated using a BAOAB scheme~\cite{leimkuhler2015molecular} with a time step~$\Delta t = 10^{-2}$. All computations are performed in Julia.
\subsection{Computation of the diffusion coefficient}
\label{subsec:computation_diffusion}
We next compute the diffusion coefficient. We integrate~\eqref{eq:kinetic_kuramoto} and~\eqref{eq:kinetic_kuramoto_mf} up to~$T = 100$, taking~$N = 10,000$ for~\eqref{eq:kinetic_kuramoto}. The computation of the self diffusion coefficient is performed using the well-known Einstein formula (see for instance~\cite{rodenhausen1989}). With this aim, recall the integrated processes~\eqref{eq:rescaled_processN} and~\eqref{eq:integrated_McKV} and let~$Q_0$ be sampled from~$\nu_{\min}$ or~$\nu^*$ and~$\textbf{Q}_0^N$ sampled from~$(\nu^*)^{\otimes N}$. The self diffusion coefficients are obtained by computing
\begin{align}
    \label{eq:einstein_formula}
    D = \lim_{T\to\infty}\frac{1}{2T}\expect\!\left[\left|Q_T-Q_0\right|^2\right],\qquad D^N = \lim_{T\to\infty}\frac{1}{2NT}\expect\!\left[\left|\textbf{Q}_T^N-\textbf{Q}_0^N\right|^2\right],
\end{align}
 where $D$ denotes the diffusion coefficient for the dynamics~\eqref{eq:kinetic_kuramoto_mf} while~$D^N$ denotes the coefficient for the dynamics~\eqref{eq:kinetic_kuramoto}.
Let~$K$ be the number of time iteration for the integration scheme. We denote by~$\widehat{Q}_{k\Delta t}^{(j)}$ and~$\widehat{\textbf{Q}}_{k\Delta t}^{N,(j)}$ the numerical approximations of~$Q_{k\Delta t}^{(j)}$ and~$\textbf{Q}_{k\Delta t}^{N,(j)}$ at time~$k\Delta t$ obtained with the BAOAB scheme at the the iteration~$k\in\{1,\dots,K\}$.
The numerical estimators of~\eqref{eq:einstein_formula} rely, respectively, on
\begin{align}
    \widehat{D}_{K,J} = \frac{1}{2K\Delta tJ}\sum_{j=1}^{J}\left|\widehat{Q}_{K\Delta t}^{(j)}-Q_0^{(j)}\right|^2,\qquad \widehat{D}_{K,J}^N = \frac{1}{2K\Delta tNJ}\sum_{j=1}^{J}\left|\widehat{\textbf{Q}}_{K\Delta t}^{N,(j)}-{\textbf{Q}}_0^{N,(j)}\right|^2,
    \label{eq:einstein_formula_estimator}
\end{align}
with~$J$ the number of independent realizations of a dynamics and~$\widehat{Q}_{K\Delta t}^{(j)}$ and~$\widehat{\textbf{Q}}_{K\Delta t}^{N,(j)}$ respectively the position of the $j$-th realization of~\eqref{eq:kinetic_kuramoto_mf} and~\eqref{eq:kinetic_kuramoto} and~$Q_0^{(j)}$ and~${\textbf{Q}}_0^{N,(j)}$ the initial positions sampled by rejection sampling. 
In practice, we compute for~$k\in\{1,\dots,K\}$ the values~$k\Delta t \widehat{D}_{K,J}$ and ~$k\Delta t \widehat{D}_{K,J}^N$ and fit the results with an affine function of the form~$\mc{D} k\Delta t + b$ with the Julia library Polynomials, and report the slopes $\mc{D}$.
For the dynamics~\eqref{eq:kinetic_kuramoto}, we use {$J = 10^3$} for the particle system and~$J = 10^6$ for the mean-field dynamics.
\begin{figure}[ht] 
    \centering
    \begin{subfigure}{0.47\textwidth}
        \centering
        \includegraphics[width=\linewidth]{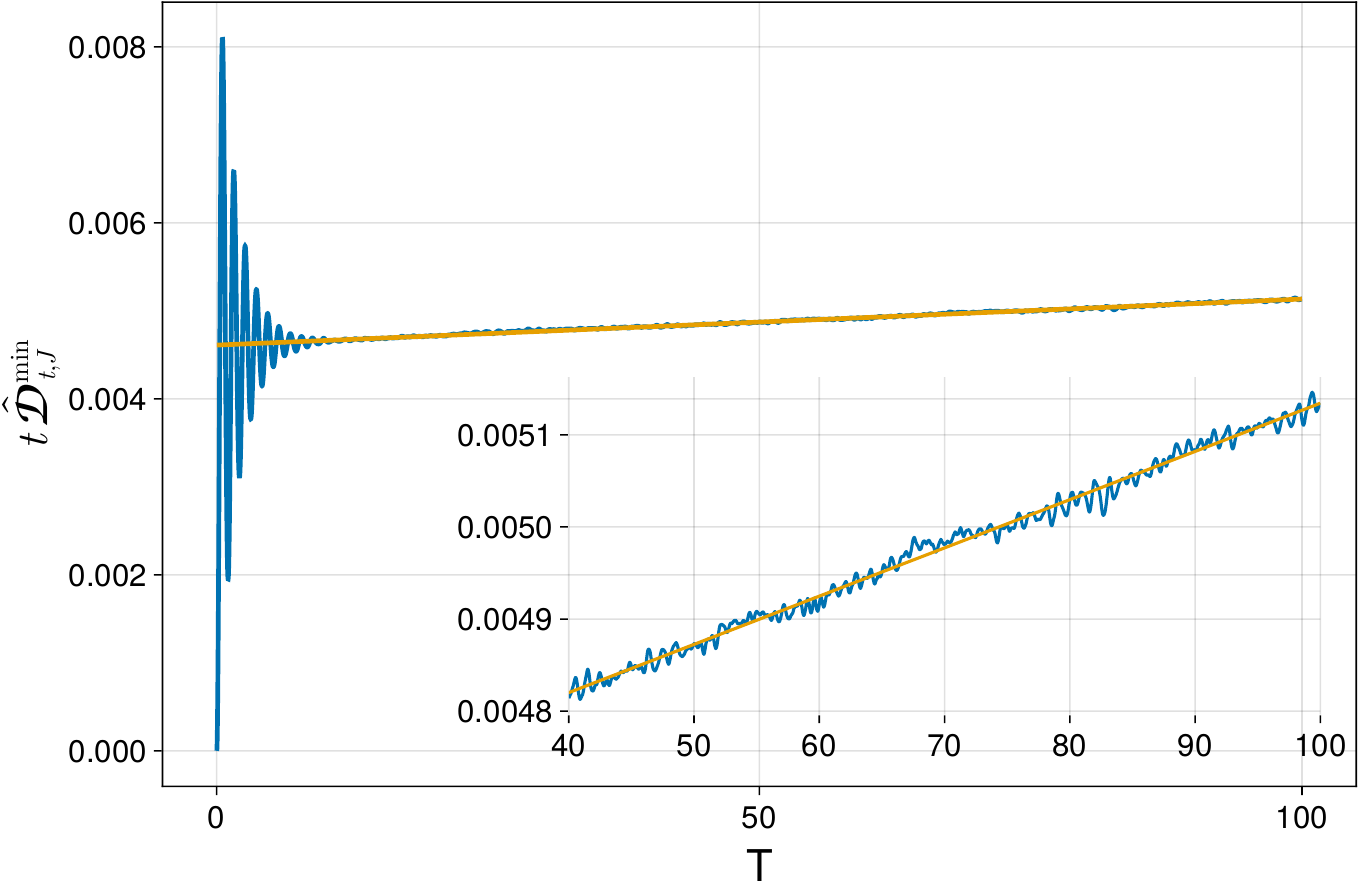}
        \caption{Diffusion coefficient for the dynamics in~\eqref{eq:kinetic_kuramoto_mf}, starting from~$\nu_{\min}$: $\mc{D} = 5.23 \times 10^{-6}$.}
        \label{fig:D_min}
    \end{subfigure}
    \hfill
    \begin{subfigure}{0.47\textwidth}
        \centering
        \includegraphics[width=\linewidth]{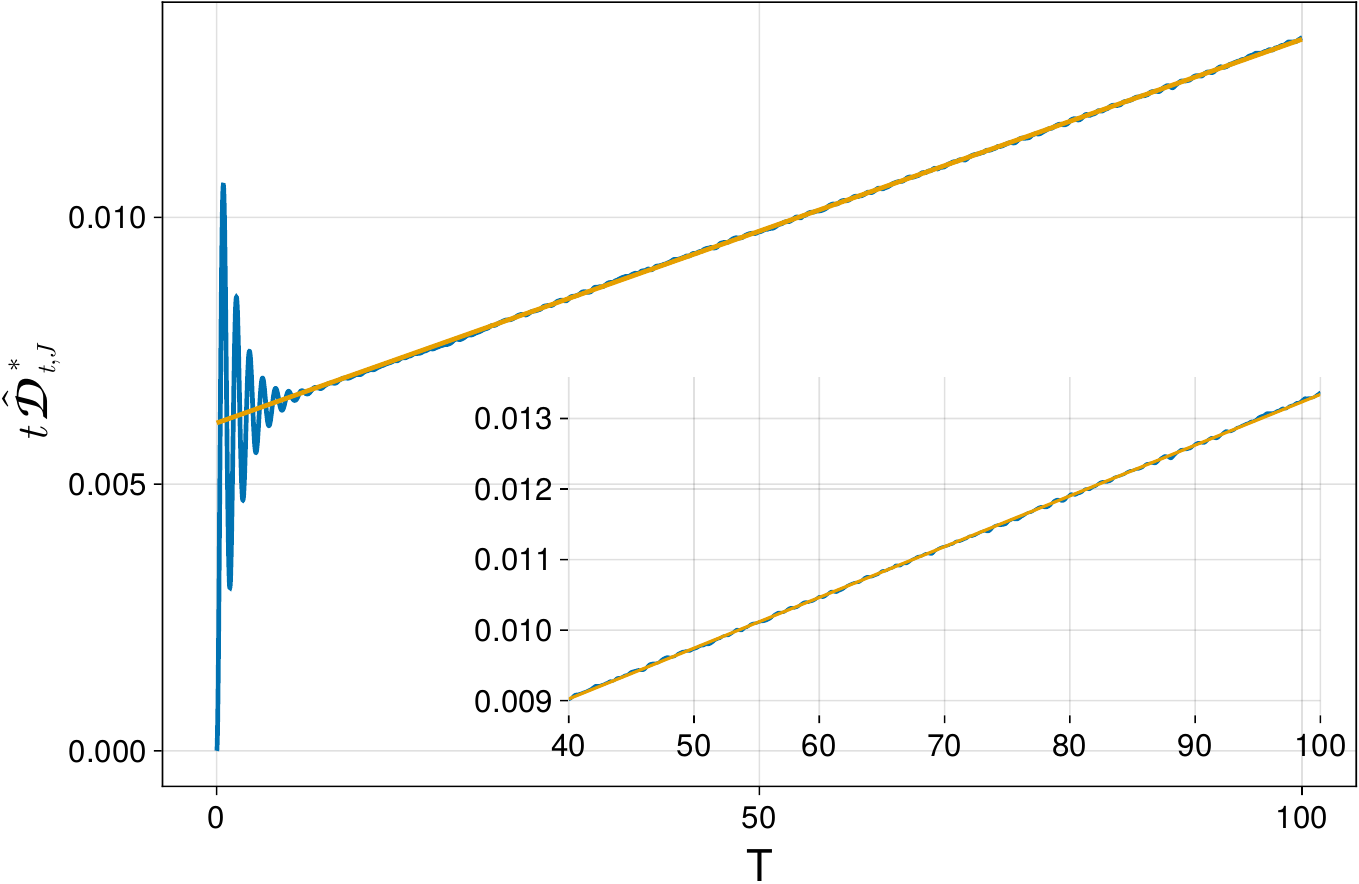}
        \caption{Diffusion coefficient for the dynamics~\eqref{eq:kinetic_kuramoto_mf} linearized around~$\nu^*$: $\mc{D} t + b$, where $\mc{D} = 7.19 \times 10^{-5}$.}
        \label{fig:D_star}
    \end{subfigure}
    \begin{subfigure}{0.47\textwidth}
        \centering
            \includegraphics[width=\linewidth]{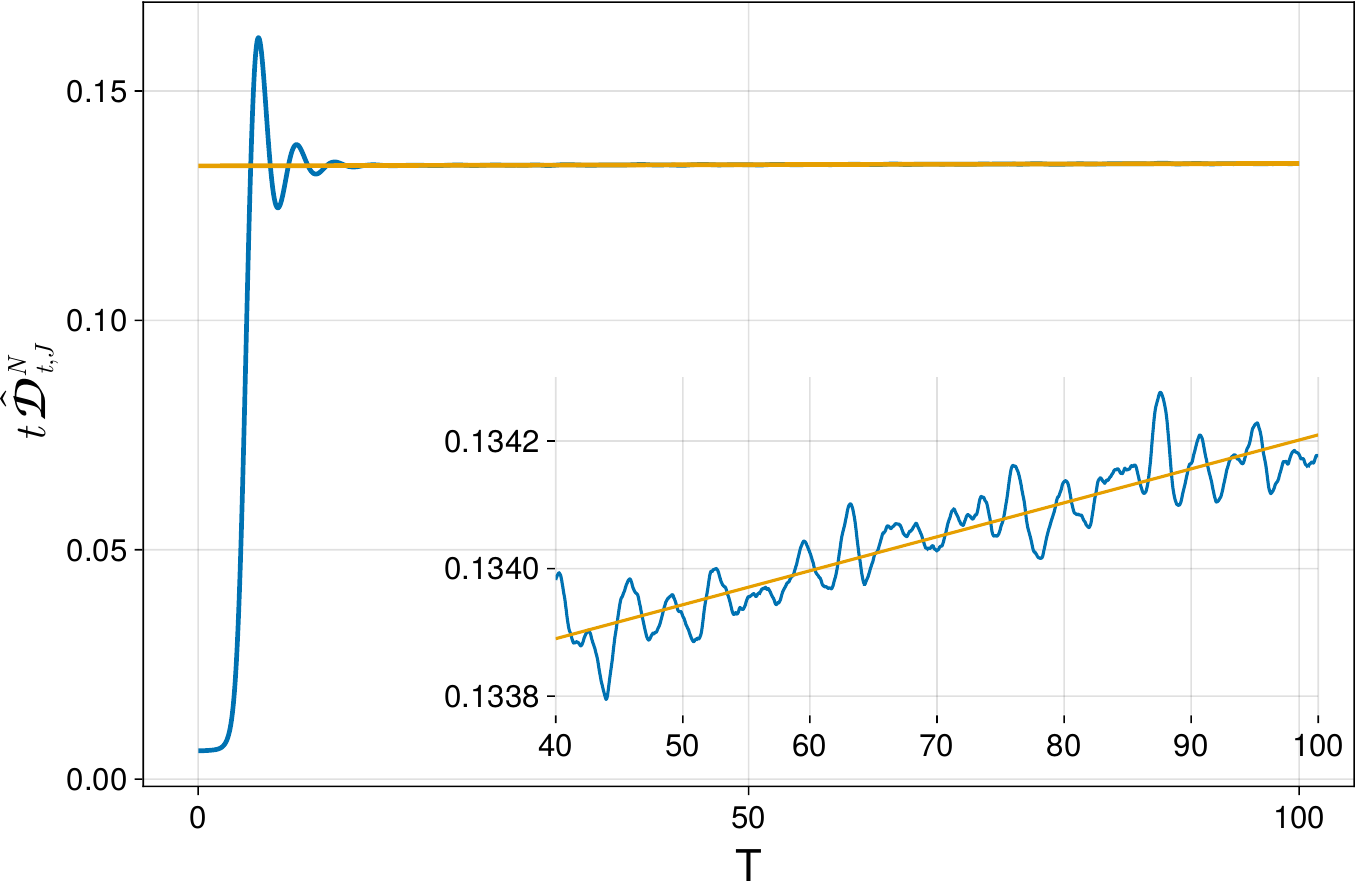}
            \caption{Diffusion coefficient for the particle system~\eqref{eq:kinetic_kuramoto}, starting from~$\nu^*$:  $\mc{D} = 5.32 \times 10^{-6}$.}
            \label{fig:D_N}
    \end{subfigure}
    \hfill
    \caption{Comparison of the self diffusion coefficients for different dynamics. Numerical calculation of the self diffusion coefficient~\eqref{eq:einstein_formula_estimator} for $V(q) = \eta W(q) = -\eta \cos(q)$. Each simulation is fitted on the time interval~$[40, 100]$ with an affine function fit of the form~$\mc{D} t+ b$ with~$\mc{D}$ corresponding to the self diffusion coefficient.}
    \label{fig:diffusion_coefficients}
\end{figure}
 We perform three different computations of the diffusion coefficient.  The first one is for the mean field dynamics linearized around~$\nu_{\min}$, the second one is for the mean field dynamics linearized around~$\nu^*$ and the last one is for the particle system, starting from~$\nu^*$. The results obtained for these three computations are shown in~\cref{fig:diffusion_coefficients}.
In particular,
\cref{fig:D_min,fig:D_star} show the mean-square displacement as a function of time.
The values of the slopes are noticeably different: for the dynamics linearized around~$\nu_{\min}$ we obtain~$\mc{D} = 5.23\times 10^{-6}$ and for the dynamics linearized around~$\nu^*$ we obtain~$\mc{D} = 7.19\times 10^{-5}$.
On the other hand,~\cref{fig:D_N} shows that $\mc{D}= 5.32\times 10^{-6}$ for the particle system when starting from~$\nu^*$. This value is very close to the one obtained for the dynamics linearized around~$\nu_{\min}$. In addition, the self diffusion coefficients for the linearized dynamics~\eqref{eq:kinetic_kuramoto_mf} have been computed using the spectral method described in~\cite[Section 4]{pavliotis2008diffusive}. We obtain
$5.22\times 10^{-6}$ for the dynamics linearized around~$\nu_{\min}$ and $7.11\times 10^{-5}$ for the dynamics linearized around~$\nu^*$. These values are in good agreement with the ones obtained using the Einstein formula~\eqref{eq:einstein_formula_estimator}.

In conclusion, the numerical results of this section show that the self diffusion coefficient for the particle system, starting from~$\nu^*$, is close to the one obtained for the dynamics linearized around~$\nu_{\min}$ and not around~$\nu^*$. This illustrates the non-commutativity of the diffusive/mean-field limit. One aim for future work is to numerically compute a dynamics with multiple local minima and one global minimizer, instead of one global minimizer and saddle points/local maximizer as for the~$O(2)$ model, to better illustrate the non-commutativity of the diffusive/mean-field limit. However, such computation will be more expensive since the transition time between local minima scales exponentially with the number of particles.
\appendix
\section{Proof of~\texorpdfstring{\cref{lem:gradp_adj}}{Lemma~\ref{lem:gradp_adj}}}
\label{ann:proof_grap_star}
We prove the result for~$\varphi\in\mc{C}_{\rm c}^\infty(\torus^d \times \real^d)$ and then conclude using a density argument.
Since~$\nabla_p^* = -\nabla_p + \beta p$, it holds
\begin{align}
    \norml{\nabla_p^* \varphi}^2 &= \norml{\nabla_p \varphi}^2 + \beta^2\norml{p \varphi}^2 - 2\beta \int_{\mc{E}}(\varphi\nabla_p\varphi)\cdot p\,\d \mescin_\infty.
    \label{eq:inch_inter}
\end{align}
For the last term of~\eqref{eq:inch_inter},
an integration by parts gives
\begin{align}
    \notag
    I &:= 2\int_{\mc{E}}(\varphi\nabla_p\varphi)\cdot p\,\d \mescin_\infty
     = \int_{\mc{E}}\nabla_p(\varphi^2)\cdot p \, \d \mescin_\infty \\
     \notag
     &= - \int_{\mc{E}}\varphi^2 \text{div}_p(p)\, \d \mescin_\infty + \beta\int_{\mc{E}}\varphi^2 |p|^2 \, \d \mescin_\infty \\
     \label{eq:inch_inter2}
     &= -d\norml{\varphi}^2 +\beta\norml{p \varphi}^2.
\end{align}
By~\eqref{eq:inch_inter} and~\eqref{eq:inch_inter2}, we conclude that $\norml{\nabla_p^* \varphi}^2 = \norml{\nabla_p \varphi}^2 + d\beta \norml{\varphi}^2$, from which~\eqref{eq:bound_gervais} follows by density.
\section{Hypocoercive estimates}
\label{sec:hypo_estimates}
The proofs presented in this section are standard and can be found, for instance, in~\cite{roussel2018spectral,iacobucci2019convergence,addala20212}. The end of the proof of~\cref{prop:bounded_aux_op} is taken from~\cite[Proposition 3.5(ii)]{bernard2022hypocoercivity} to obtain a quantitative expression of~$\kham$.
\begin{proof}[{Proof of~\cref{prop:coercivity_properties}}]
    The first inequality is a direct consequence of the Poincaré inequality for Gaussian probability measures. More precisely, recalling~\eqref{eq:def_LFD}, it holds for~$\varphi\in \mc{C}_{\rm b}^\infty(\mc{E})$ that
    \[
    -\av{\LFD\varphi,\varphi}_{L^2(\mescin_\infty)} = \frac{1}{\beta}\av{\nabla_p\varphi,\nabla_p\varphi}_{L^2(\mescin_\infty)}
    \geq \norml{(1-\projvit)\varphi}.
    \]
    For the second point, first note that
    \begin{align}
        \label{eq:reecriture_lham_projp}
        \Lham\projvit\varphi(q,p) = p\cdot \nabla_q \projvit \varphi(q).
    \end{align}
    Therefore, with~$\mesvit$ the centered Gaussian measure on~$\real^d$ with covariance matrix~$\beta^{-1}\text{I}_d$,
    \begin{align}
         \notag
        &\|\Lham\projvit \varphi\|_{L^2(\mescin_\infty)}^2 = \expect_{\mescin_\infty}\!\left[|p\cdot \nabla_q \projvit \varphi(q)|^2\right]
        = \sum_{1\leq i,j \leq d} \expect_{\mescin_\infty}\!\left[p_i p_j\partial_{q_i} \projvit \varphi(q)\partial_{q_j}\projvit \varphi(q)\right]  \\
        \label{eq:inter2}
        &\, = \sum_{1\leq i,j \leq d} \expect_{\mesvit}\!\left[p_i p_j\right]\expect_{\mespos}\!\left[\partial_{q_i}\projvit \varphi(q) \partial_{q_j}\projvit \varphi(q)\right] =  \sum_{1\leq i,j \leq d} \frac{\delta_{ij}}{\beta}\expect_{\mespos_\infty}\!\left[\partial_{q_i}\projvit \varphi(q) \partial_{q_j}\projvit \varphi(q)\right] \\
        \notag
        &\, = \frac{1}{\beta}\|\nabla_q \projvit \varphi(q)\|_{L^2(\mespos_\infty)}^2.
    \end{align}
   Since~$ \varphi \in \mc{C}_{\rm b}^\infty(\mc{E})\cap L_0^2(\mescin_\infty)$, it holds~$\expect_{\mespos}\!\left[\projvit\varphi\right] = 0$. Thus, using the Poincaré inequality given by~\cref{lem:poincare_mean_field}, we obtain
   \begin{align*}
    \|\Lham\projvit \varphi\|_{L^2(\mescin_\infty)}^2 \geq \frac{\poinccu^2}{\beta} \|\projvit \varphi(q)\|_{L^2(\mescin_\infty)}^2.
   \end{align*}
    We now establish~\eqref{eq:lambham}. The estimate~\eqref{eq:lham_coercivity} gives
   \begin{align}
    \label{eq:op_sym_ovd}
    (\Lham\projvit)^*\Lham\projvit \geq \frac{\poinccu^2}{\beta}\projvit,
   \end{align}
    in the sense of symmetric operators on~$L_0^2(\mescin_\infty)$.
   Furthermore,
   \begin{align*}
        A\Lham\projvit = (1+(\Lham\projvit)^*\Lham\projvit)^{-1}(\Lham\projvit)^*\Lham\projvit.
   \end{align*}
   Since~$(\Lham\projvit)^*\Lham\projvit$ is self-adjoint and the function~$t \mapsto 1 - 1/(1+t) = t/(1+t)$ is increasing, the inequality~\eqref{eq:lambham} is a consequence of spectral calculus.
\end{proof}
\begin{proof}[{Proof of~\cref{prop:bounded_aux_op}}]
First we prove~\eqref{eq:torture}. We use the fact that operators acting only on the variables~$q$, such as~$\nabla_q$ and~$\nabla_q^*$, commute with those acting only on the variables~$p$, such as~$\nabla_p$,~$\nabla_p^*$ and~$\projvit$.
We rely on the following equalities:
\begin{align*}
    \LFD = -\frac1\beta\sum_{i=1}^{d}\partial^*_{p_i}\partial_{p_i},\qquad \Lham = \frac1\beta\sum_{i=1}^{d}\partial^*_{p_i}\partial_{q_i} - \partial^*_{q_i}\partial_{p_i},
\end{align*}
with
\begin{align*}
    \partial_{q_i}^* = -\partial_{q_i} + \beta \partial_{q_i}U_\infty ,\qquad \partial_{p_i}^* = -\partial_{p_i} + \beta p_i.
\end{align*}
We have the following identities:
\begin{align}
    \label{eq:identies_projp}
    \partial_{p_i}\projvit = 0,\qquad\projvit\partial_{p_i}^*=0,\qquad\projvit\partial_{p_i}\partial_{p_j}^* = \partial_{p_i}\partial_{p_j}^*\projvit = \beta\projvit \delta_{ij}.
\end{align}
Combining~\eqref{eq:identies_projp} with~\eqref{eq:inter2} shows that
\begin{align*}
    (\Lham\projvit)^*\Lham\projvit = \frac1\beta \sum_{i=1}^{d}\partial_{q_i}^* \partial_{q_i} \projvit=:-\Lovd\projvit,
\end{align*}
where~$\Lovd = -\displaystyle\frac{1}{\beta}\sum_{i=1}^{d}\partial_{q_i}^*\partial_{q_i}$.
Therefore,~$A$ can be rewritten as
\begin{align*}
    A = (1-\Lovd)^{-1}(\Lham\projvit)^* = \frac1\beta (1-\Lovd)^{-1}\projvit\sum_{i=1}^{d}\partial^*_{q_i}\partial_{p_i}.
\end{align*}
To obtain the bound~\eqref{eq:torture}, we consider the adjoint of~$A\Lham(1-\projvit)$:
\begin{align*}
    &-(1-\projvit)\Lham A^* = -\frac{1}{\beta^2}(1-\projvit)\sum_{i,j=1}^{d}\!\left(\partial^*_{p_i}\partial_{q_i} - \partial^*_{q_i}\partial_{p_i}\right)\!\partial_{q_j}\partial_{p_j}^*\projvit(1-\Lovd)^{-1} \\
    & = -\frac{1}{\beta^2}(1-\projvit)\!\left[\sum_{i,j=1}^{d}\partial^*_{p_i}\partial^*_{p_j}\projvit\partial_{q_i}\partial_{q_j} - \beta \sum_{i=1}^{d}\partial^*_{q_i}\partial_{q_i}\projvit \right]\!(1-\Lovd)^{-1} \\
    & =-\frac{1}{\beta^2}(1-\projvit)\sum_{i,j=1}^{d}\partial^*_{p_i}\partial^*_{p_j}\projvit\partial_{q_i}\partial_{q_j} (1-\Lovd)^{-1} \\
    &= -(1-\projvit)\sum_{i,j=1}^{d}\mc{U}_{ij}\projvit\partial_{q_i}\partial_{q_j} (1-\Lovd)^{-1},
\end{align*}
where we used~$(1-\projvit)\partial^*_{q_i}\partial_{q_i}\projvit = 0$ in the third line and where~$\mc{U}_{ij} = p_i p_j -\frac{1}{\beta}\delta_{ij}$. By~\cite[Lemma A.2]{roussel2018spectral} we have~$\|\mc{U}_{ij}\|_{L^2(\mesvit)}^2 \leq 2/\beta^2$. Therefore, for any~$\varphi\in\mc{C}_{\rm b}^\infty(\torus^d)$,
\begin{align}
    \notag
    &\norml{(1-\projvit)\Lham A^* \varphi}^2
    = \sum_{i,j=1}^{d}\|\mc{U}_{ij}\projvit\|_{L^2(\mesvit)}^2\left\|\partial_{q_i}\partial_{\vq{j}}(1-\Lovd)^{-1}\projvit\right\|_{L^2(\mespos_\infty)}^2 \\
    \label{ineq:op_adjoint_hypo}
    &\leq \frac{2}{\beta^2}\sum_{i,j=1}^{d}\left\|\partial_{q_i}\partial_{\vq{j}}\projvit(1-\Lovd)^{-1}\right\|_{L^2(\mespos_\infty)}^2
        = \frac{2}{\beta^2}\!\left\|\nabla_q^2 (1-\Lovd)^{-1}\projvit\varphi\right\|_{L^2(\mespos_\infty)}^2.
\end{align}
To conclude on the bound~\eqref{eq:torture} we use~\cite[Lemma 3.6]{bernard2022hypocoercivity} to write
    \begin{align*}
        \forall u\in\mc{C}^\infty(\torus^d), \quad\left\|\nabla_q^2 u\right\|_{L^2(\mespos_\infty)}^2 \leq \left\|\nabla_q^*\nabla_q u\right\|_{L^2(\mespos_\infty)}^2 + (K_V + K_W)\left\|\nabla_q u\right\|_{L^2(\mespos_\infty)}^2.
    \end{align*}
Since
$$
\left\|\displaystyle\!\frac{1}{\beta}\nabla_q^*\nabla_q u\right\|_{L^2(\mespos_\infty)}^2\leq\left\|\!\left(1+\displaystyle\frac{1}{\beta}\nabla_q^*\nabla_q\right)\! u\right\|_{L^2(\mespos_\infty)}^2,
$$
we have that for any~$u\in\mc{C}^\infty(\torus^d)$,
\begin{align}
    \label{eq:bochner}
   \frac{1}{\beta^2}\left\|\nabla_q^2 u\right\|_{L^2(\mespos_\infty)}^2 \leq \left\|\!\left(1+\displaystyle\frac{1}{\beta}\nabla_q^*\nabla_q\right)\! u\right\|_{L^2(\mespos_\infty)}^2 + \frac{K_V + K_W}{\beta^2}\left\|\nabla_q u\right\|_{L^2(\mespos_\infty)}^2.
\end{align}
Applying~\eqref{eq:bochner} to~\eqref{ineq:op_adjoint_hypo} gives
\begin{align*}
    \norml{(1-\projvit)\Lham A^* \varphi}^2 &\leq 2\left\|\projvit\varphi\right\|_{L^2(\mespos_\infty)}^2 \\
    &\quad + \frac{2(K_V + K_W)}{\beta^2}\left\|\nabla_q (1-\Lovd)^{-1}\projvit\varphi\right\|_{L^2(\mespos_\infty)}^2.
\end{align*}
Finally, we bound the second term of the right-hanside of the above inequality as follows:
\begin{align*}
    &\frac{1}{\beta^2}\left\|\nabla_q (1-\Lovd)^{-1}\projvit\varphi\right\|_{L^2(\mespos_\infty)}^2
    = \frac1\beta\slangle{\frac{1}{\beta}\nabla_q^*\nabla_q (1-\Lovd)^{-1}\projvit\varphi, (1-\Lovd)^{-1}\projvit\varphi} \\
    &\qquad \leq \frac1\beta\slangle{\left(1+\frac{1}{\beta}\nabla_q^*\nabla_q\right) (1-\Lovd)^{-1}\projvit\varphi, (1-\Lovd)^{-1}\projvit\varphi} \\
    &\qquad = \frac1\beta\slangle{\projvit\varphi, (1-\Lovd)^{-1}\projvit\varphi} \leq \frac{1}{\beta}\|(1-\Lovd)^{-1}\|\left\|\projvit\varphi\right\|_{L^2(\mespos_\infty)}^2 \\
    &\qquad\leq \frac{1}{\beta} \frac{1}{\poinccu^2/\beta}\left\|\projvit\varphi\right\|_{L^2(\mespos_\infty)}^2 \leq \frac{1}{\poinccu^2}\left\|\projvit\varphi\right\|_{L^2(\mespos_\infty)}^2,
\end{align*}
where the last line is a consequence of~\cref{lem:poincare_mean_field}. This leads to
\begin{align*}
\norml{(1-\projvit)\Lham A^* \varphi}^2 &\leq 2\!\left(1+\frac{K_V+K_W}{\poinccu^2}\right)\!\left\|\projvit\varphi\right\|_{L^2(\mespos_\infty)}^2 \\
&= \kham^2 \norml{\projvit\varphi}^2,
\end{align*}
and gives~\eqref{eq:torture}.
To prove~\eqref{eq:A_LFD_1-P}, we use the fact that
\begin{align*}
    \projvit\Lham\LFD &= -\frac{1}{\beta^2}\projvit \sum_{i,j=1}^{d} (\partial^*_{p_i}\partial_{q_i} - \partial^*_{q_i}\partial_{p_i})\partial^*_{p_j}\partial_{p_j} = \frac{1}{\beta^2}\projvit \sum_{i,j=1}^{d} \partial^*_{q_i}\partial_{p_i}\partial^*_{p_j}\partial_{p_j} \\
    &=  \frac{1}{\beta^2}\sum_{i,j=1}^{d} \partial^*_{q_i}(\projvit\partial_{p_i}\partial^*_{p_j})\partial_{p_j} =
    \frac{1}{\beta^2}\sum_{i,j=1}^{d} \partial^*_{q_i}(\projvit \beta \delta_{ij})\partial_{p_j} = -\projvit\Lham,
\end{align*}
where in the first line we use the fact that~$\projvit\partial_{p_i}^* = 0$ and in the second line, we use~$\projvit\partial_{p_i}\partial^*_{p_j} = \projvit \beta \delta_{ij}$. This equality shows that~$A\LFD = -A$ and by~\cref{lem:bound_A_hypo} it concludes the proof.
\end{proof}
\section*{Acknowledgements}
The authors are grateful to Pierre Monmarché and Mathias Rousset for very helpful discussions and Clément Guillot for insightful comments on the numerical simulations.
Claude and ChatGPT were used for proofreading. ChatGPT (Sol 5.6) identified an error in the proof of~\cref{th:Kuramoto}, and suggested a correction. All mathematical statements, proofs, and numerical results were independently verified by the authors, who take full responsibility for them.
This project has received funding from the European Research Council (ERC) under the European Union’s Horizon 2020 research and innovation programme (project EMC2, grant agreement No 810367).
We also acknowledge funding from the Agence Nationale de la Recherche, under grants ANR-21-CE40-0006 (SINEQ) and ANR-23-CE40-0027 (IPSO).
GAP is partially supported by an ERC-EPSRC Frontier Research
Guarantee through grant no. EP/X038645, ERC through Advanced grant no. 247031,
and a Leverhulme Trust Senior Research Fellowship, SRF{$\setminus$}R1{$\setminus$}241055.
\bibliographystyle{siamplain}
\bibliography{main}
\end{document}